\let\chooseClass1   
\let\chooseClass2   
\let\chooseClass4   

\ifx\chooseClass1
\RequirePackage{fix-cm}
\documentclass[smallcondensed,numbook,envcountsame]{svjour3}
\smartqed
\journalname{Applied Categorical Structures}
\fi

\ifx\chooseClass2
\documentclass[14pt]{extarticle}
\fi

\ifx\chooseClass3
\documentclass[17pt]{extarticle}
\fi

\ifx\chooseClass4
\documentclass[12pt]{article}
\fi

\ifx\chooseClass1
	\else
\makeatletter
\def\@seccntformat#1{\csname the#1\endcsname.\quad}
\renewcommand\section{\@startsection {section}{1}{\z@}%
                                   {-3.5ex \@plus -1ex \@minus -.2ex}%
                                   {2.3ex \@plus.2ex}%
                                   {\normalfont\large\bfseries}}
\renewcommand\subsection{\@startsection{subsection}{2}{\z@}%
                        {3.25ex plus 1ex minus .2ex}{-.5em}%
                        {\normalfont\normalsize\bfseries}}
\renewcommand\subsubsection{\@startsection{subsubsection}{3}{\z@}%
                        {3.25ex plus 1ex minus .2ex}{-.5em}%
                        {\normalfont\normalsize\it}}
\@addtoreset{equation}{section}
\makeatother

\usepackage{amsthm}
\newtheoremstyle{boldhead}
{\topsep}
{\topsep}
{\slshape}
{}
{\bfseries}
{.}
{ }
{\thmname{#1}\thmnumber{ #2}\thmnote{ (#3)}}

\newtheoremstyle{boldremark}
{\topsep}
{\topsep}
{\upshape}
{}
{\bfseries}
{.}
{ }
{\thmname{#1}\thmnumber{ #2}\thmnote{ (#3)}}

\swapnumbers

\theoremstyle{boldhead}
\newtheorem{theorem}[subsection]{Theorem}

\newtheorem{lemma}[subsection]{Lemma}
\newtheorem{proposition}[subsection]{Proposition}

\theoremstyle{boldremark}

\newtheorem{definition}[subsection]{Definition}

\newtheorem{example}[subsection]{Example}

\newtheorem{remark}[subsection]{Remark}

\fi

\usepackage{		amsmath,
                amsfonts,
                amssymb,
}

\usepackage[matha,mathx]{mathabx}

\numberwithin{equation}{section}

\usepackage{ifpdf}
\ifpdf
	\usepackage[pdftex]{hyperref}
\else
    \usepackage[hypertex]{hyperref}
\fi

\ifx\chooseClass1
\let\QED\qed
  \else
\newcommand\QED{}
\usepackage[mathcal]{euscript}
\fi

\usepackage{diagrams}
\diagramstyle[height=2em,balance,righteqno,PostScript=dvips,nohug]

\usepackage[sans]{dsfont}
\ifx\mathds\undefined
\newcommand\1{{1\mkern-5mu {\mathrm I}}}
\else
\newcommand\1{{\mathds 1}}
\fi

\newlength{\mylabelwidths}
\newenvironment{myitemize}{\begin{list}{}{%
\setlength{\labelwidth}{\mylabelwidths}%
\setlength{\leftmargin}{\mylabelwidths}\addtolength{\leftmargin}{0.5em}%
\setlength{\itemsep}{-0\baselineskip}}}%
{\end{list}}

\def\rhaha{\raise.20ex\hbox{$\rightharpoonup$}\kern-1em\lower.20ex\hbox{$\rightharpoondown$}}
\def\lhaha{\raise.20ex\hbox{$\leftharpoonup$}\kern-1em\lower.20ex\hbox{$\leftharpoondown$}}
\def\dhaha{\kern.01em\downharpoonleft\kern-.26em\downharpoonright\kern.01em}
\def\uhaha{\kern.01em\upharpoonleft\kern-.26em\upharpoonright\kern.01em}
\newarrowhead{twoharpoons}\rhaha\lhaha\dhaha\uhaha

\newarrow{DashTo}{}{dash}{}{dash}{->}
\newarrow{Epi}----{triangle}
\newarrow{Id}===={twoharpoons}
\newarrow{Mono}{boldhook}---{->}
\newarrow{TTo}----{->}

\newcommand\DD{{\mathbb D}}
\newcommand\NN{{\mathbb N}}
\newcommand\ZZ{{\mathbb Z}}

\newcommand{\cc}{{\mathcal C}}
\newcommand{\cd}{{\mathcal D}}
\newcommand{\co}{{\mathcal O}}
\newcommand{\cv}{{\mathcal V}}
\newcommand{\cw}{{\mathcal W}}
\newcommand{\cz}{{\mathcal Z}}

\newcommand{\sfj}{{\mathsf j}}
\newcommand{\sfh}{{\mathsf h}}
\newcommand{\sfv}{{\mathsf v}}
\newcommand{\sfw}{{\mathsf w}}

\newcommand{\beps}{{\boldsymbol\eps}}
\newcommand{\bfeta}{{\boldsymbol\eta}}
\newcommand{\bv}{{\mathbf v}}
\newcommand{\bull}{{\scriptscriptstyle\bullet}}
\newcommand{\bw}{{\mathbf w}}
\newcommand{\sk}{{\mathsf{sk}}}
\newcommand{\tdt}{\otimes\dots\otimes}

\newcommand{\sS}[2]{\vphantom{#2}#1 #2}
\newcommand{\n}[1]{\nobreakdash-\hspace{0pt}}
\newcommand{\ainf}[1]{$A_\infty$\nobreakdash-\hspace{0pt}}

\newcommand{\lesscup}{\bigsqcup\mkern-16mu{\sss<}\mkern5mu}

\let\con\triangleright

\let\eps\varepsilon
\let\ge\geqslant
\let\kk\Bbbk
\let\le\leqslant
\let\mb\mathbf

\let\rto\xrightarrow
\let\sss\scriptstyle
\let\tens\otimes
\let\ttt\textstyle
\let\und\underline

\let\wt\widetilde

\newcommand\sff{{\mathsf f}}
\newcommand\sfg{{\mathsf g}}

\newcommand{\coder}{\textup{-coder}}
\newcommand{\modul}{\textup{-mod}}

\DeclareMathOperator\Ab{Ab}
\DeclareMathOperator{\Alg}{Alg}
\DeclareMathOperator{\Bbar}{Bar}
\DeclareMathOperator{\sfBar}{{\sf Bar}}
\DeclareMathOperator{\cAlg}{cAlg}
\DeclareMathOperator{\cCoalg}{cCoalg}
\DeclareMathOperator{\CCoalg}{CCoalg}
\DeclareMathOperator{\Cobar}{Cobar}
\DeclareMathOperator{\sfCobar}{{\sf Cobar}}
\DeclareMathOperator{\Coalg}{Coalg}
\DeclareMathOperator\dg{\mathbf{dg}}
\DeclareMathOperator\End{End}

\DeclareMathOperator\ev{ev}
\DeclareMathOperator\gr{\mathbf{gr}}

\DeclareMathOperator\id{id}

\DeclareMathOperator\inj{in}
\DeclareMathOperator\Ker{Ker}
\DeclareMathOperator{\chCoalg}{chCoalg}
\DeclareMathOperator{\fhCoalg}{fhCoalg}
\DeclareMathOperator{\hCoalg}{hCoalg}

\DeclareMathOperator\Ob{Ob}
\DeclareMathOperator{\oin}{\overline{\textup{in}}}
\newcommand{\op}{{\operatorname{op}}}
\DeclareMathOperator\opr{\overline{\textup{pr}}}
\DeclareMathOperator\pr{pr}
\DeclareMathOperator{\restr}{restr}
\DeclareMathOperator\Set{\mathcal Set}
\DeclareMathOperator{\Tw}{Tw}
\DeclareMathOperator{\ucAlg}{ucAlg}
\DeclareMathOperator{\ucCoalg}{ucCoalg}

\newcommand{\defref}[1]{Definition~\ref{#1}}
\newcommand{\exaref}[1]{Example~\ref{#1}}
\newcommand{\lemref}[1]{Lemma~\ref{#1}}
\newcommand{\propref}[1]{Proposition~\ref{#1}}
\newcommand{\remref}[1]{Remark~\ref{#1}}
\newcommand{\secref}[1]{Section~\ref{#1}}

\begin{document}
\title{Curved homotopy coalgebras}
\author{Volodymyr Lyubashenko}

\ifx\chooseClass1
\dedication{To the memory of Ukrainian mathematician Yuriy Victorovych Bodnarchuk}
\institute{V. Lyubashenko \at
              Institute of Mathematics NASU,
			  3 Tereshchenkivska st.,
			  Kyiv-4, 01601 MSP, Ukraine \\
              Tel.: +380-44-2357819,
              Fax: +380-44-2352010\\
              \email{lub@imath.kiev.ua}
}
\date{Received: date / Accepted: date}
\fi

\maketitle

\allowdisplaybreaks[1]

\begin{abstract}
We describe the category of homotopy coalgebras, concentrating on properties of relatively cofree homotopy coalgebras, morphisms and coderivations from an ordinary coalgebra to a relatively cofree homotopy coalgebra, morphisms and coderivations between coalgebras of latter type.
Cobar- and bar-constructions between counit-complemented curved coalgebras, unit-complemented curved algebras and curved homotopy coalgebras are described.
Using twisting cochains an adjunction between cobar- and bar-constructions is derived under additional assumptions.
\ifx\chooseClass1
\fi
\end{abstract}

\ifx\chooseClass1
	\else
\begin{flushright}
\it To the memory of Ukrainian mathematician Yuriy Victorovych Bodnarchuk
\end{flushright}
\fi

\subsection{Introduction}
In this article we discuss curved algebras in sufficiently general monoidal category $(\cv,\tens,\1)$ \textit{e.g.} of graded modules over a graded commutative ring.
Attempts to construct a dual notion lead us to curved homotopy coalgebras.
The notion of a homotopy coalgebra is studied in detail here.
It is a particular case of homotopy comonoid defined by Leinster \cite[Definition~2.2]{math/9912084}.
This concept under the name of quasi-comonoid is important also for Pridham \cite[Definition 1.4]{0908.0116}.

Definition of homotopy coalgebras is quite simple: it is a lax Monoidal functor $C$ from the category $\co_\sk^\op$ to $\cv$, where $\co_\sk$ is the category of finite ordered sets \(n=\{1<2<\dots<n\}\), $n\ge0$, and their non-decreasing maps.
Thus $\co_\sk$ coincides with the algebraist's $\Delta$.
A different notation following \cite{BesLyuMan-book} is chosen in order not to confuse with another $\Delta$, the category of non-empty finite ordered sets.
Certainly, an ordinary counital coassociative coalgebra gives rise to a homotopy coalgebra.

Existence of cobar- and bar-constructions depends on the chosen supply of coalgebras.
For conilpotent coalgebras these constructions are presented in \cite{Lyu-curved-coalgebras}.
The cofree conilpotent coalgebras are tensor objects \(XT=\coprod_{n\in\NN}X^{\tens n}\) equipped with the cut comultiplication.
The main example of a homotopy coalgebra \(C=X\tilde{T}\) uses \(C(1)=X\hat{T}=\prod_{n\in\NN}X^{\tens n}\).
This homotopy coalgebra is cofree with respect to ordinary coalgebras in $\cv$.
It is the main ingredient of curved homotopy coalgebras defined in this article.

Curved (dg-)algebras and curved (dg-)coalgebras were defined by Positselski over a field \cite{0905.2621} and with some topology over a local ring \cite{1202.2697}.
In order to make his cobar- and bar-constructions work in $\cv$ we assume that the image of the unit \(\eta:\1\rMono A\) of an algebra admits a direct complement in $\cv$, similarly, the kernel of the counit \(\eps:C\rEpi \1\) admits a direct complement in $\cv$.
Choosing one of these complements we obtain unit-complemented algebras and counit-complemented coalgebras.
We provide cobar-construction as a functor from counit-complemented curved coalgebras to unit-complemented curved algebras and bar-construction as a functor from unit-complemented curved algebras to curved homotopy coalgebras.
The latter target is chosen in order that an analogue of adjunction holds true.
Namely, two bimodules over the categories of curved coalgebras and curved algebras are isomorphic to the bimodule of twisting cochains.
Using additional grading it is possible for some class of coalgebras such that \(XT\simeq X\tilde{T}\) to redefine the bar-construction as a functor to ordinary (not homotopy) curved coalgebras, so that cobar- and bar-constructions were adjoint to each other.

\subsection{Conventions}
There are two reasons to write down an operator (mapping, functor) on the right of its argument.
First of all, some of our mappings are homogeneous of certain degree and the Koszul rule does not allow arbitrary placing of symbols of certain degree requiring a sign.
We use right internal homs as in \ainf-category setting of \cite{BesLyuMan-book}.
In particular, composition of two morphisms \(f:X\to Y\) and \(g:Y\to Z\) is denoted \(fg=f\cdot g\).
The tensor product of two morphisms $f$, $g$ of graded modules of degree $\deg f$, $\deg g$ acts on homogeneous elements as
\[ (x\tens y).(f\tens g) =(-1)^{\deg y\cdot\deg f} xf\tens yg =(-1)^{y\cdot f} xf\tens yg.
\]

Secondly, the shift or translation functor $[n]$ taking a graded module $M$ to itself with the shifted grading $M[n]$ is usually written on the right.
Among the two possible ways to write down the tensor module, $TM$ and $MT$, the latter is preferable because it is clear what mean $M[n]T$ and $MT[n]$, however ambiguous $TM[n]$ requires additional parentheses.
Besides we do not follow the `operator on the right' rule strictly and expressions similar to $f(x)$ occur in the article as well.

Note that Monoidal categories and functors \cite[Definitions 2.5, 2.6]{BesLyuMan-book} that we speak about in this article are the same as unbiased monoidal categories and functors \cite[Definition~3.1.1]{math.CT/0305049}.
It is safe to view Monoidal categories as strongly or even strictly monoidal.
In the terminology `lax/colax Monoidal categories and functors' we also follow \cite[Definitions 2.5, 2.6]{BesLyuMan-book}.
Discrepancy between these terms and other existing terminology (oplax) should not confuse the reader, for we indicate the direction of structure morphisms explicitly.
\textit{E.g.} our colax Monoidal category is identified with a lax monoidal category in the sense of Leinster \cite[Definition~3.1.1]{math.CT/0305049}.

The set of natural numbers $\NN$ consists of non-negative integers.
Each $n\in\NN$ can be viewed as the standard set of $n$ elements which is chosen to be \(\{1,2,\dots,n\}\).

\subsection{The ground category}
In this article $\cv$ denotes a closed symmetric Monoidal additive category with countable coproducts and products.
We assume that $\cv$ is idempotent-\hspace{0pt}split (Karoubian) and the tensor product preserves countable coproducts.
For some results we make a weaker assumption: \((\cv,\tens,\lambda)\) is a colax Monoidal category, defined as an opposite to a lax Monoidal category \((\cv^\op,\tens,\lambda^\op)\), \textit{cf.} \cite[Definition~2.5]{BesLyuMan-book}.
For the sake of convenience we use the structure morphisms \(\lambda^\phi:\tens^{j\in J}\tens^{i\in\phi^{-1}j}X_i\to\tens^{i\in I}X_i\), \(\phi:I\to J\in\co_\sk\), for colax Monoidal categories, assuming them invertible in Monoidal case.

Define a symmetric strict monoidal category $\cz$ with \(\Ob\cz=\ZZ$.
The sets of morphisms $\cz(m,n)$ are empty if $m\ne n$ and \(\cz(n,n)=\mu_2\overset{\text{def}}=\{1,-1\}\) is the group of two elements for $n\in\ZZ$.
The tensor product on objects is the addition, the tensor product on morphisms \(\tens:\cz(a,a)\times\cz(b,b)\to\cz(a+b,a+b)\) is the multiplication \(\mu_2\times\mu_2\to\mu_2\) in the group $\mu_2$.
Finally, the symmetry is chosen as
\[ c_{a,b} =(-1)^{ab} \in \cz(a+b,b+a), \qquad a,b\in\ZZ.
\]

We assume given a symmetric monoidal functor \(\1[-]:\cz\to\cv\), \(n\mapsto\1[n]\).
For simplicity we suppose that this monoidal functor is strict.
In particular \(\1[0]=\1\) is the unit object of $\cv$.
We require that \(-1\in\cz(n,n)\) were represented by $-\id_{\1[n]}$.
Let us derive some consequences of this structure.

First of all, there are functors \([n]=\_\tens\1[n]:\cv\to\cv\), \(X\mapsto X[n]=X\tens\1[n]\), $n\in\ZZ$.
Together they form a translation structure \cite[Definition~13.2]{BesLyuMan-book}, that is, a monoidal functor
\[ [-]: \ZZ \rMono \cz \rTTo^{\1[-]} \cv \rTTo^{R^\tens} \End\cv,
\]
where $\ZZ$ is the discrete monoidal subcategory of $\cz$ with the set of objects $\ZZ$, \(X(YR^\tens)=X\tens Y\).
Almost the same notion is called a weak action of the group $\ZZ$ on a category $\cv$ \cite[D\'efinition~1.2.2]{MR1453167}.
For the sake of simplicity we act as if the monoidal functor $[-]$ were strict.

Let us use the closedness of $\cv$.
The identity map $\id_{X[n]}$ admits a presentation
\[ \id_{X[n]} =\bigl\langle X\tens\1[n] \rTTo^{1\tens\sigma^n} X\tens\und\cv(X,X[n]) \rTTo^\ev X[n] \bigr\rangle
\]
for a unique morphism
\[ \sigma^n =\sigma_X^n: \1[n] \to \und\cv(X,X[n]) \in \cv.
\]
One easily shows that for all $a,b\in\ZZ$
\ifx\chooseClass1
\begin{multline*}
\bigl\langle \1[a+b] =\1[a]\tens\1[b] \rTTo^{\sigma_X^a\tens\sigma_{X[a]}^b} \und\cv(X,X[a])\tens\und\cv(X[a],X[a+b])
\\
\rto m \und\cv(X,X[a+b]) \bigr\rangle =\sigma_X^{a+b}.
\end{multline*}
	\else
\[ \bigl\langle \1[a+b] =\1[a]\tens\1[b] \rTTo^{\sigma_X^a\tens\sigma_{X[a]}^b} \und\cv(X,X[a])\tens\und\cv(X[a],X[a+b]) \rto m \und\cv(X,X[a+b]) \bigr\rangle =\sigma_X^{a+b}.
\]
\fi
Notice also that $\sigma_X^n$ is dinatural in $X$: for all \(f:X\to Y\in\cv\)
\begin{diagram}[LaTeXeqno]
\1[n] &\rTTo^{\sigma_X^n} &\und\cv(X,X[n])
\\
\dTTo<{\sigma_Y^n} &= &\dTTo>{\und\cv(X,f[n])}
\\
\und\cv(Y,Y[n]) &\rTTo^{\und\cv(f,Y[n])} &\und\cv(X,Y[n])
\label{dia-sigma-dinatural}
\end{diagram}

Equip the category $\Ab^\ZZ$ of graded abelian groups $A=(A^n)_{n\in\ZZ}$ with the usual monoidal product \((A\tens B)^n=\coprod_{l+m=n}A^l\tens_\ZZ B^m\) and the signed symmetry \(c:a\tens b\mapsto(-1)^{ab}b\tens a\).
Consider the functor
\begin{equation}
-^\bull: \cv \to \Ab^\ZZ, \qquad X \mapsto X^\bull =(X^n)_{n\in\ZZ}, \qquad X^n =\cv(\1[-n],X).
\label{eq-V-AbZ}
\end{equation}
Equipped with the natural transformation
\[ \tens: X^n\tens_\ZZ Y^m =\cv(\1[-n],X)\tens_\ZZ\cv(\1[-m],Y) \to \cv(\1[-n-m],X\tens Y) =(X\tens Y)^{n+m},
\]
it becomes lax symmetric monoidal.
The graded abelian group $\kk=\1^\bull$, \(\kk^n=\1^\bull=\cv(\1[-n],\1)\) is actually a graded ring, graded commutative in the sense that \(ba=(-1)^{ab}ab\) for all \(a,b\in\kk^\bull\).
Moreover, the graded abelian groups coming from $\cv$ are actually commutative $\kk$\n-bimodules (whose category is denoted $\gr=\gr_\kk$), so lax monoidal functor~\eqref{eq-V-AbZ} has the latter category as the target: \(-^\bull:\cv\to\gr\).
We use the term `$\kk$\n-modules' instead of longer term `commutative $\kk$\n-bimodules'.

Being closed $\cv$ is enriched into itself (or rather $\und\cv$ is enriched into $\cv$).
The lax monoidal functor $-^\bull$ makes $\cv$ enriched in $\gr$.
Denote this $\gr$\n-enriched category \(\overline\cv\).
Thus for each pair $X,Y$ of objects of $\cv$ there is a graded $\kk$\n-module \(\overline\cv(X,Y)=(\und\cv(X,Y)^n)_{n\in\ZZ}\).
Elements of \(\und\cv(X,Y)^n=\cv(\1[-n],\und\cv(X,Y))\) are called morphisms $X\to Y$ of degree $n$.
For instance, \(\und\cv(X,Y)^0=\cv(X,Y)\) and \(\sigma_X^n\in\und\cv(X,X[n])^{-n}\) are morphisms of degree $-n$.
The $\gr$\n-category structure of $\overline\cv$ includes, in particular, the composition of morphisms of certain degrees.
\textit{E.g.} equation~\eqref{dia-sigma-dinatural} for \(f:X\to Y\in\cv\) can be presented as
\[ \bigl(X \rTTo^{\sigma_X^n} X[n] \rTTo^{f[n]} Y[n] \bigr) =\bigl(X \rTTo^f Y \rTTo^{\sigma_Y^n} Y[n] \bigr) \in \overline\cv,
\]
which can be converted to
\[ f[n] =\bigl( X[n] \rTTo^{\sigma_X^{-n}} X \rTTo^f Y \rTTo^{\sigma_Y^n} Y[n] \bigr) \in \overline\cv.
\]
Therefore, there are natural isomorphisms (of degree 0)
\[ \sigma^{-n}(1^{\tens(a-1)}\tens\sigma^n\tens1^{\tens(n-a)}): (\tens_{i=1}^kX_i)[n] \to X_1\tdt X_{a-1}\tens X_a[n]\tens X_{a+1}\tdt X_k
\]
which we use explicitly or implicitly.
Given \(X,Y,Z\in\Ob\cv\) and \(k\in\ZZ\) we have the following diagram.
More precisely, for each \(p:X\tens Y\to Z\in\cv\) there is a unique \(p^!:Y\to\und\cv(X,Z)\in\cv\) which makes the small triangle commutative
\begin{diagram}
X\tens\und\cv(X,Z)[k] &\rTTo^{(1\tens\sigma^{-k})\sigma^k} &(X\tens\und\cv(X,Z))[k] &\rTTo^{\ev[k]} &Z[k]
\\
&= &\uTTo<{(1\tens p^!)[k]} &\ruTTo^{\ttt=\quad}_{p[k]} &
\\
\uTTo<{1\tens p^![k]} &&(X\tens Y)[k] &&\uTTo>q
\\
&\ruTTo^{(1\tens\sigma^{-k})\sigma^k} &&= &
\\
X\tens Y[k] &&\rEq &&X\tens W
\end{diagram}
where \(W=Y[k]\).
Thus for an arbitrary \(q:X\tens W\to Z[k]\in\cv\) there is a unique \(q^!:W\to\und\cv(X,Z)[k]\in\cv\) which makes the exterior commutative.
Therefore, $\und\cv(X,Z)[k]$ is isomorphic to $\und\cv(X,Z[k])$ and the isomorphism $i$ satisfies
\begin{diagram}
X\tens\und\cv(X,Z)[k] &\rTTo^{(1\tens\sigma^{-k})\sigma^k} &(X\tens\und\cv(X,Z))[k]
\\
\dTTo<{1\tens i}>\wr &= &\dTTo>{\ev[k]}
\\
X\tens\und\cv(X,Z[k]) &\rTTo^\ev &Z[k]
\end{diagram}
so we have natural bijections
\begin{equation}
\und\cv(X,Z)^k =\cv(\1[-k],\und\cv(X,Z)) \simeq \cv(\1,\und\cv(X,Z)[k]) \simeq \cv(\1,\und\cv(X,Z[k])) \simeq \cv(X,Z[k]).
\label{eq-V(XZ)k-V(XZk)}
\end{equation}
Let us consider the case when \(-^\bull:\cv\to\gr\) is an equivalence.

\begin{example}
Let $\cv=\gr=\gr\text-\kk\modul$.
For any graded \(\kk\)-module $M$ and an integer $a$ denote by $M[a]$ the same module with the cohomological grading shifted by $a$: \(M[a]^k=M^{a+k}\).
The unit object of $\gr$ is $\1=\kk$ and its shifts are \(\1[a]=\kk[a]\).
We identify $M[a]$ with \(M\tens\kk[a]\) via \(x\mapsto x\tens1\) where \(1\in\kk[a]^{-a}\).
Then \(\sigma^a:M\to M[a]\) is the ``identity map'' \(M^k\ni x\mapsto x\in M[a]^{k-a}\) of degree \(\deg\sigma^a=-a\).
Write elements of $M[a]$ as \(m\sigma^a\).
When \(f:V\to X\) is a homogeneous map of certain degree, the map \(f[a]:V[a]\to X[a]\) is defined as \(f[a]=(-1)^{a\deg f}\sigma^{-a}f\sigma^a=(-1)^{af}\sigma^{-a}f\sigma^a\).
In particular, the differential \(d:M\to M[-1]\) in a $\dg$\n-module $M$ induces the differential \(d[a]:M[a]\to M[a-1]\) in $M[a]$.
The degree 0 isomorphisms \(\sigma^{-a}\cdot(\sigma^a\tens1):(V\tens W)[a]\to(V[a])\tens W\),
\((v\tens w)\sigma^a\mapsto(-1)^{wa}v\sigma^a\tens w\), and \(\sigma^{-a}\cdot(1\tens\sigma^a):(V\tens W)[a]\to V\tens(W[a])\), \((v\tens w)\sigma^a\mapsto v\tens w\sigma^a\), are graded natural.
This means that for arbitrary homogeneous maps \(f:V\to X\), \(g:W\to Y\) the following squares commute:
\begin{diagram}[w=5em]
(V[a])\tens W &\lTTo^{\sigma^{-a}\cdot(\sigma^a\tens1)}_\sim &(V\tens W)[a]
&\rTTo^{\sigma^{-a}\cdot(1\tens\sigma^a)}_\sim &V\tens(W[a])
\\
\dTTo<{(f[a])\tens g} &&\dTTo<{(f\tens g)[a]} &&\dTTo>{f\tens(g[a])}
\\
(X[a])\tens Y &\lTTo^{\sigma^{-a}\cdot(\sigma^a\tens1)}_\sim &(X\tens Y)[a]
&\rTTo^{\sigma^{-a}\cdot(1\tens\sigma^a)}_\sim &X\tens(Y[a])
\end{diagram}
Actually, the second isomorphism is ``more natural'' than the first one, not only because it does not have a sign, but also because it suits better the right operator system of notations, accepted in this paper.
We often identify \((V\tens W)[a]\) with \(V\tens(W[a])\) via \(\sigma^{-a}\cdot(1\tens\sigma^a)\).
\end{example}

\section{Homotopy coalgebras}
We describe the category of homotopy coalgebras, concentrating on properties of relatively cofree homotopy coalgebras, morphisms and coderivations from an ordinary coalgebra to a relatively cofree homotopy coalgebra, morphisms and coderivations between relatively cofree homotopy coalgebras.

\subsection{Homotopy comonoids}
A coalgebra (=comonoid) in \((\cv,\tens,\lambda)\) is defined as an algebra in \((\cv^\op,\tens,\lambda^\op)\), or as a colax Monoidal functor \(C:\1\to\cv\), \textit{cf.} \cite[Definition~2.25]{BesLyuMan-book}.
Equivalently, it is an object $C$ of $\cv$ equipped with a morphism \(\Delta_I:C\to C^{\tens I}\) for each $I\in\Ob\co_\sk$ such that \(\Delta_{\mb1}=\id\) and for every map $\phi:I\to J\in\co_\sk$ the following equation holds:
\begin{equation}
\Delta_I = \bigl( C \rTTo^{\Delta_J} C^{\tens J} \rTTo^{\tens^{j\in J}\Delta_{\phi^{-1}j}\;} \tens^{j\in J}C^{\tens\phi^{-1}j} \rTTo^{\lambda^\phi} C^{\tens I} \bigr).
\label{eq-Delta-CCCC}
\end{equation}
A lax Monoidal functor between colax Monoidal categories \((F,\varphi):(\cc,\tens,\lambda)\to(\cd,\tens,\lambda)\) is the opposite to the colax Monoidal functor between lax Monoidal categories \((F^\op,\varphi^\op):(\cc^\op,\tens,\lambda^\op)\to(\cd^\op,\tens,\lambda^\op)\), see \cite[Definition~2.26]{BesLyuMan-book}.
The following result of \cite{BesLyuMan-book} treats algebras in lax Monoidal categories, however, we cite it in dual form:

\begin{proposition}[Proposition~2.27 of \cite{BesLyuMan-book}]
\label{pro-coalgebra-lax-Monoidal-functor}
A coalgebra $C$ in a colax Monoidal category $\cv$ defines a lax Monoidal functor
\begin{multline*}
\hfill (F,\varphi^I): (\co_\sk^\op,\sqcup_I,\id) \to (\cv,\tens^I,\lambda^\phi), \qquad F(J)=C^{\tens J}, \hfill
\\
(\phi^\op:J\to I)\in\co_\sk^\op \leftrightarrow (\phi:I\to J)\in\co_\sk
\ifx\chooseClass1
\hfill \\ \hfill
\fi
\mapsto \Delta_C^\phi = \bigl(C^{\tens J} \rTTo^{\tens^{j\in J}\Delta_{\phi^{-1}j}} \tens^{j\in J}C^{\tens\phi^{-1}j} \rto{\lambda^\phi} C^{\tens I} \bigr).
\end{multline*}
Let \(n_i\in\NN=\Ob\co_\sk^\op\) for \(i\in I\in\Ob\co_\sk\).
The natural transformation \(\varphi^I:\tens^{i\in I}C^{\tens n_i}\to C^{\tens\sum_{i\in I}n_i}\) is defined as \(\lambda^\psi:\tens^{i\in I}C^{\tens \psi^{-1}i}\to C^{\tens N}\) for \(N=\sqcup_{i\in I}\mb n_i\), \(\psi:N\to I\in\co_\sk\) such that \(|\psi^{-1}i|=n_i\).
\end{proposition}

When $\cv$ is a strict monoidal category, a coalgebra $C$ in $\cv$ gives rise to a strict monoidal functor, as proven by Mac Lane \cite[Proposition~VII.5.1]{MacLane}.
\propref{pro-coalgebra-lax-Monoidal-functor} gives reasons for

\begin{definition}[cf. Leinster \cite{math/9912084} Definition~2.2]
A \emph{homotopy comonoid} in a colax Monoidal category \((\cv,\tens,\lambda)\) is a lax Monoidal functor \((F,\varphi^I):(\co_\sk^\op,\sqcup_I,\id)\to(\cv,\tens^I,\lambda^\phi)\).
\end{definition}

When \((\cv,\tens,\lambda)\) is a Monoidal and homotopical category, the above notion is a particular case of homotopy comonoid due to Leinster \cite[Definition~2.2]{math/9912084}, see detailed exposition in \cite{math.QA/0002180}.
In our case all morphisms of $\cv$ are declared to be homotopy equivalences.

A homotopy comonoid $(C,\chi)$ in a colax Monoidal category \((\cv,\tens,\lambda)\) is the collection of
\begin{myitemize}
\item[--] objects \(C(k)\in\Ob\cv\) for all $k\in\NN$;

\item[--] morphisms \(C(\phi^\op):C(J)\to C(I)\in\cv\) for all \(\phi:I\to J\in\co_\sk\) (with the corresponding \(\phi^\op:J\to I\in\co_\sk^\op\));

\item[--] morphisms \(\chi^I_{N_1,\dots,N_I}:\tens^{i\in I}C(N_i)\to C(\sqcup_{i\in I}N_i)\in\cv\) for all \(I,N_1,\dots,N_I\in\Ob\co_\sk\)
\end{myitemize}
such that
\begin{myitemize}
\item[--] \(C:\co_\sk^\op\to\cv\) is a functor;

\item[--] \(\chi^I_{N_1,\dots,N_I}\) is natural in $N_1$, \dots, $N_I$, that is, for all families \(\phi_i:M_i\to N_i\in\co_\sk\), \(i\in I\),
\begin{diagram}[LaTeXeqno]
\tens^{i\in I}C(N_i) &\rTTo^{\chi^I_{N_1,\dots,N_I}} &C(\sqcup_{i\in I}N_i)
\\
\dTTo<{\tens^{i\in I}C(\phi_i^\op)} &= &\dTTo>{C((\sqcup_{i\in I}\phi_i)^\op)}
\\
\tens^{i\in I}C(M_i) &\rTTo^{\chi^I_{M_1,\dots,M_I}} &C(\sqcup_{i\in I}M_i)
\label{dia-CN-CN-CM-CM}
\end{diagram}

\item[--] \(\chi^{\mb1}_N=\id\);

\item[--] for every map $\phi:I\to J\in\co_\sk$ of and all families \((N_i)_{i\in I}\) of objects of $\co_\sk$ the following equation holds:
\begin{diagram}[LaTeXeqno]
\tens^{j\in J}\tens^{i\in\phi^{-1}j}C(N_i) &\rTTo^{\tens^{j\in J}\chi^{\phi^{-1}j}} &\tens^{j\in J}C(\sqcup_{i\in\phi^{-1}j}N_i) &\rTTo^{\chi^J} &C(\sqcup_{j\in J}\sqcup_{i\in\phi^{-1}j}N_i)
\\
\dTTo<{\lambda^\phi} &&= &&\dEq
\\
\tens^{i\in I}C(N_i) &&\rTTo^{\chi^I} &&C(\sqcup_{i\in I}N_i)
\label{dia-2phi-1phi-Monoidal-functor}
\end{diagram}
\end{myitemize}

\begin{example}\label{exa-lax-comonoid-tilde-T}
Let $(\cv,\tens^I,\lambda^\phi)$ be a colax Monoidal category with countable products.
A homotopy comonoid \((C,\chi)=X\tilde{T}\) is associated with an object $X$ of $\cv$ as follows.
The functor $X\tilde{T}$ is defined on objects \(I\in\Ob\co_\sk\) by
\[ X\tilde{T}(I) =\prod_{n\in\NN^I} X^{\tens\|n\|}, \qquad n=(n_1,\dots,n_I), \qquad \|n\| \overset{\text{def}}= \sum_{i\in I}n_i,
\]
and on morphisms $\phi:I\to J\in\co_\sk$ by
\begin{equation}
X\tilde{T}(\phi^\op)\pr_n =\pr_{\phi_*n}: X\tilde{T}(J) =\prod_{k\in\NN^J}X^{\tens\|k\|} \to X^{\tens\|\phi_*n\|} =X^{\tens\|n\|}, \qquad \forall\, n\in\NN^I,
\label{eq-XT(fop)prn}
\end{equation}
where \(\phi_*n\in\NN^J\) has the components
\begin{equation}
(\phi_*n)_j =\sum_{i\in\phi^{-1}j}n_i.
\label{eq-(f*n)j}
\end{equation}
Note that \(\|\phi_*n\|=\|n\|\).

For an arbitrary pair of composable maps \(I\rto\phi J\rto\psi K\in\co_\sk\) and an element $n\in\NN^I$ we have \(\psi_*\phi_*n=(\psi\circ\phi)_*n\in\NN^K\).
This implies for \(C=X\tilde{T}\)
\begin{multline*}
C(\psi^\op)C(\phi^\op)\pr_n =C(\psi^\op)\pr_{\phi_*n} =\pr_{\psi_*\phi_*n} =\pr_{(\phi\psi)_*n} =C((\phi\psi)^\op)\pr_n:
\ifx\chooseClass1
\\
\fi
C(K) \to X^{\tens\|n\|},
\end{multline*}
therefore, \(C(\psi^\op)C(\phi^\op)=C((\phi\psi)^\op)\) and $X\tilde{T}$ is a functor.

Let \(I,N_1,\dots,N_I\Ob\co_\sk\) and
\begin{equation}
n =n_1\oplus\dots\oplus n_I \in \NN^{N_1}\oplus\dots\oplus\NN^{N_I} =\NN^{\sqcup_{i\in I}N_i}
\label{eq-nnn-NNNNNN}
\end{equation}
The transformation $\chi^I$ is defined by
\begin{equation}
\chi^I_{N_1,\dots,N_I}\pr_n =\Bigl( \bigotimes_{i\in I} \prod_{n_i\in\NN^{N_i}}X^{\tens\|n_i\|} \rTTo^{\tens^{i\in I}\pr_{n_i}\;} \bigotimes_{i\in I} X^{\tens\|n_i\|} \rTTo^{\lambda^\psi} X^{\tens\|n\|} \Bigr),
\label{eq-chi-pr-X-X-X}
\end{equation}
where \(\psi:\mb n=\lesscup_{i\in I}\mb{n_i}\to I\) takes $\mb{n_i}$ to $i$.
The collection $\chi^I$ is natural in $N_1$, \dots, $N_I$.
In fact, for a family of maps \(\phi_i:N_i\to M_i\in\co_\sk\), $i\in I$, and the induced \(\sqcup_{i\in I}\phi_i:N=\sqcup_{i\in I}N_i\to\sqcup_{i\in I}M_i=M\) we have
\begin{diagram}[h=2.4em]
\bigotimes_{i\in I} \prod_{n_i\in\NN^{M_i}} X^{\tens\|m_i\|} &\rTTo^{\chi^I_{M_1,\dots,M_I}} &\prod_{m\in\NN^M} X^{\tens\|m\|}
\\
\dTTo<{\tens^{i\in I}X\tilde{T}(\phi_i^\op)} &= &\dTTo>{X\tilde{T}((\sqcup_{i\in I}\phi_i)^\op)}
\\
\bigotimes_{i\in I} \prod_{n_i\in\NN^{N_i}} X^{\tens\|n_i\|} &\rTTo^{\chi^I_{N_1,\dots,N_I}} &\prod_{n\in\NN^N} X^{\tens\|n\|}
\end{diagram}
In order to prove this, postcompose the equation with $\pr_n$ and notice that
\begin{gather*}
(\sqcup_{i\in I}\phi_i)_*n =\phi_{1*}n_1\oplus\dots\oplus \phi_{I*}n_I \in \NN^{M_1}\oplus\dots\oplus\NN^{M_I} =\NN^M,
\\
\begin{split}
&\bigl[\tens^{i\in I}X\tilde{T}(\phi_i^\op)\bigr] \chi^I_{N_1,\dots,N_I} \pr_n =\bigl[\tens^{i\in I}X\tilde{T}(\phi_i^\op)\bigr] \bigl(\tens^{i\in I}\pr_{n_i}\bigr) \lambda^\psi
\\
&=\bigl(\tens^{i\in I}\pr_{\phi_{i*}n_i}\bigr) \lambda^\psi =\chi^I_{M_1,\dots,M_I} \pr_{(\sqcup_{i\in I}\phi_i)_*n} =\chi^I_{M_1,\dots,M_I} X\tilde{T}((\sqcup_{i\in I}\phi_i)^\op) \pr_n.
\end{split}
\end{gather*}

Clearly, \(\chi^{\mb1}_N=\id\).
For every $\phi:I\to J\in\co_\sk$ and each family \(N_1,\dots,N_I\Ob\co_\sk\) we have to prove equation~\eqref{dia-2phi-1phi-Monoidal-functor} which takes the form of square (?) in the following diagram, where \(M_j=\sqcup_{i\in\phi^{-1}j}N_i\) and \(N=\sqcup_{i\in I}N_i=\sqcup_{j\in J}M_j\):
\begin{diagram}[h=2.4em]
\bigotimes_{j\in J} \bigotimes_{i\in\phi^{-1}j} \prod_{n_i\in\NN^{N_i}} X^{\tens\|n_i\|} &\rTTo^{\tens^{j\in J}\chi^{\phi^{-1}j}} &\bigotimes_{j\in J} \prod_{m_j\in\NN^{M_j}} X^{\tens\|m_j\|} &\rTTo^{\tens^{j\in J}\pr_{m_j}\;} &\bigotimes_{j\in J} X^{\tens\|m_j\|}
\\
\dTTo<{\lambda^\phi} &(?) &\dTTo>{\chi^J} &= &
\\
\bigotimes_{i\in I} \prod_{n_i\in\NN^{N_i}} X^{\tens\|n_i\|} &\rTTo^{\chi^I} &\prod_{n\in\NN^N} X^{\tens\|n\|} &&\dTTo>{\lambda^{\sqcup_{j\in J}\mb{m_j}\to J}}
\\
\dTTo<{\tens^{i\in I}\pr_{n_i}} &= &&\rdTTo^{\pr_n} &
\\
\bigotimes_{i\in I} X^{\tens\|n_i\|} &&\rTTo^{\lambda^{\sqcup_{i\in I}\mb{n_i}\to I}} &&X^{\tens\|n\|}
\end{diagram}
With notation~\eqref{eq-nnn-NNNNNN} and 
\[ m_j =\oplus_{i\in\phi^{-1}j}n_i \in \bigoplus_{i\in\phi^{-1}j}\NN^{N_i} =\NN^{M_j}
\]
both trapezia in the above diagram commute.
The top-right exterior path is equal to
\begin{multline*}
\Bigl[ \bigotimes_{j\in J} \bigotimes_{i\in\phi^{-1}j} \prod_{n_i\in\NN^{N_i}} X^{\tens\|n_i\|} \rTTo^{\tens^{j\in J}\tens^{i\in\phi^{-1}j}\pr_{n_i}\;} \bigotimes_{j\in J} \bigotimes_{i\in\phi^{-1}j} X^{\tens\|n_i\|}
\\
\hfill \rTTo^{\tens^{j\in J}\lambda^{\sqcup_{i\in\phi^{-1}j}\mb{n_i}\to\phi^{-1}j}} \bigotimes_{j\in J} X^{\tens\|m_j\|} \rTTo^{\lambda^{\sqcup_{j\in J}\mb{m_j}\to J}} X^{\tens\|n\|} \Bigr] \hskip\multlinegap
\\
=\Bigl[ \bigotimes_{j\in J} \bigotimes_{i\in\phi^{-1}j} \prod_{n_i\in\NN^{N_i}} X^{\tens\|n_i\|} \rTTo^{\tens^{j\in J}\tens^{i\in\phi^{-1}j}\pr_{n_i}\;} \bigotimes_{j\in J} \bigotimes_{i\in\phi^{-1}j} X^{\tens\|n_i\|} \rto{\lambda^\phi} \bigotimes_{i\in I} X^{\tens\|n_i\|}
\ifx\chooseClass1
\\
\fi
\rTTo^{\lambda^{\sqcup_{i\in I}\mb{n_i}\to I}} X^{\tens\|n\|} \Bigr]
\end{multline*}
which equals the left-bottom exterior path in the above diagram.
The equation here is the ``tetrahedron equation'' for the colax Monoidal category $(\cv,\tens^I,\lambda^\phi)$ written for maps \(\sqcup_{i\in I}\mb{n_i}\rto\psi I\rto\phi J\) (see \cite[Eq.~(2.5.4)]{BesLyuMan-book} with the opposite orientation of arrows).
Thus, $(X\tilde{T},\chi)$ is a homotopy coalgebra.
\end{example}

\begin{proposition}\label{pro-lax-coalgebra-ordinary}
Let $(C,\chi)$ be a homotopy coalgebra in a colax Monoidal category.
It comes from an ordinary coalgebra as in \propref{pro-coalgebra-lax-Monoidal-functor} iff
\begin{equation}
C(I) =C(1)^{\tens I}, \qquad \chi^I_{n_1,\dots,n_I} =\lambda^{\sqcup_{i\in I}\mb{n_i}\to I}.
\label{eq-CC-chi-lambda}
\end{equation}
\end{proposition}

\begin{proof}
Necessity follows from \propref{pro-coalgebra-lax-Monoidal-functor}.

Assume that equations~\eqref{eq-CC-chi-lambda} hold.
Denote
\[ \Delta_K =C((\con:K\to1)^\op): C =C(1) \to C^{\tens K}.
\]
Given \(\phi:I\to J\in\co_\sk\) represent as \(\sqcup_{j\in J}\phi_j:\sqcup_{j\in J}\phi^{-1}j\to J\), where \(\phi_j=\con:M_j=\phi^{-1}j\to\{j\}=\mb1=N_j\).
Then \(\sqcup_{j\in J}N_j=J\) and \(\chi^J_{N_1,\dots,N_J}=\lambda^{\id_J}=\id\).
Naturality \eqref{dia-CN-CN-CM-CM} of $\chi^J$ implies that
\[ C(\phi^\op) =C((\sqcup_{i\in J}\phi_j)^\op) =\bigl[ \tens^{j\in J}C(\phi_j^\op) \bigr] \lambda^\phi =\bigl( \tens^{j\in J}\Delta_{\phi^{-1}j} \bigr) \lambda^\phi
\]
just as in \propref{pro-coalgebra-lax-Monoidal-functor}.
Being a functor, $C$ takes the equation \(\con=\bigl(I\rto\phi J\rto\con \mb1\bigr)\) to equation~\eqref{eq-Delta-CCCC}.
Thus \((C=C(1),\Delta_I)\) is an ordinary coalgebra.
\QED\end{proof}

\subsection{Morphisms of homotopy coalgebras}
We define morphisms of homotopy coalgebras, describe morphisms from an ordinary coalgebra to $X\tilde{T}$, $X\in\Ob\cv$, give a supply of morphisms between $X\tilde{T}$ and $Y\tilde{T}$.

\begin{definition}
A \emph{morphism between homotopy comonoids} \(t:(C,\chi)\to(G,\gamma)\) is a Monoidal transformation of lax Monoidal functors
\[ t: (C,\chi^I) \to (G,\gamma^I): (\co_\sk^\op,\sqcup_I,\id) \to (\cv,\tens^I,\lambda^\phi).
\]
The category of homotopy comonoids=coalgebras is denoted $\hCoalg$.
\end{definition}

In detail, $t$ is a family \(t(k):C(k)\to G(k)\in\cv\), $k\in\NN$, which is\\ (a) natural: for all \(\phi:I\to J\in\co_\sk\) and (b) monoidal: for every $I\in\Ob\co_\sk$, $n_i\in\NN$, $i\in I$,
\begin{equation}
\begin{diagram}[inline]
C(J) &\rTTo^{t(J)} &G(J)
\\
\dTTo<{C(\phi^\op)} &= &\dTTo>{G(\phi^\op)}
\\
C(I) &\rTTo^{t(I)} &G(I)
\end{diagram}
\quad,\quad
\begin{diagram}[width=5em,inline]
\tens^{i\in I}C(n_i) &\rTTo^{\chi^I} &C(n_1+\dots+n_I)
\\
\dTTo<{\tens^It} &= &\dTTo>t
\\
\tens^{i\in I}G(n_i) &\rTTo^{\gamma^I} &G(n_1+\dots+n_I)
\end{diagram}
\;.
\label{eq-dia-Monoidal-trans}
\end{equation}

\begin{example}
The construction of \exaref{exa-lax-comonoid-tilde-T} is a functor \(\tilde{T}:\cv\to\hCoalg\).
\end{example}

\begin{example}\label{exa-e:XT-tildeT}
Let $X\in\Ob\cv$.
The tensor object \(XT=\sqcup_{n\ge0}X^{\tens n}\) equipped with the cut comultiplication
\[ (x_1\dots x_n)\Delta =\sum_{k=0}^n x_1\dots x_k\tens x_{k+1}\dots x_n \in \bigsqcup_{k=0}^nX^{\tens k}\tens X^{\tens(n-k)} \subset XT\tens XT
\]
and the counit \(\eps=\pr_0:XT\to\kk\) is a coalgebra in $\cv$.
All comultiplications $\Delta_I$ on the summand indexed by $k\in\NN$ are found via
\[ \inj_k\Delta_I =\sum_{n_1+\dots+n_I=k} \Bigl( X^{\tens k} \rTTo^{(\lambda^{\sqcup_{i\in I}n_i\to I})^{-1}} \tens^{i\in I}X^{\tens n_i} \rTTo^{\tens^{i\in I}\inj_{n_i}\;} XT^{\tens I} \Bigr).
\]
They satisfy the identity for \(n=(n_1,\dots,n_I)\in\NN^I\)
\[ \bigl( XT \rTTo^{\Delta_I} XT^{\tens I} \rTTo^{\tens^{i\in I}\pr_{n_i}\;} \tens^{i\in I}X^{\tens n_i} \rTTo^{\lambda^{\sqcup_{i\in I}n_i\to I}} X^{\tens\|n\|} \bigr) =\pr_{\|n\|}.
\]

We claim that the family \(e(I):XT^{\tens I}\to X\tilde T(I)\) defined by the equation
\begin{equation}
e(I)\pr_n = \bigl( XT^{\tens I} \rTTo^{\tens^{i\in I}\pr_{n_i}\;} \tens^{i\in I}X^{\tens n_i} \rTTo^{\lambda^{\sqcup_{i\in I}n_i\to I}} X^{\tens\|n\|} \bigr),
\label{eq-e(I)prn}
\end{equation}
\(n=(n_1,\dots,n_I)\in\NN^I\), is a homotopy coalgebra morphism.
In fact, $e$ is natural, that is, the exterior of the following diagram commutes due to commutativity of the diagram obtained by postcomposing the exterior with \(\pr_n:X\tilde T(I)\to X^{\tens\|n\|}\) for $n\in\NN^I$:
\begin{diagram}[nobalance]
XT^{\tens J} &&&\rTTo^{e(J)} &&&X\tilde T(J)
\\
\dTTo>{\tens^{j\in J}\Delta_{\phi^{-1}j}} &\rdTTo(4,2)>{\tens^{j\in J}\pr_{(\phi_*n)_j}} &&&&\ldTTo(2,4)>{\pr_{\phi_*n}}
\\
&&\tens^{j\in J}\tens^{i\in\phi^{-1}j}X^{\tens n_i} &\rTTo_{\hspace*{-2em}\tens^{j\in J}\lambda^{\sqcup_{i\in\phi^{-1}j}n_i\to\phi^{-1}j}} &\tens^{j\in J}X^{\tens(\phi_*n)_j}
\\
&\ruTTo^{\tens^{j\in J}\tens^{i\in\phi^{-1}j}\pr_{n_i}} &\dTTo>{\lambda^\phi} &&\dTTo<{\lambda^{\sqcup_{j\in J}(\phi_*n)_j\to J}}
\\
\tens^{j\in J}XT^{\tens\phi^{-1}j} &\hspace*{3em} &\tens^{i\in I}X^{\tens n_i} &\rTTo_{\lambda^{\sqcup_{i\in I}n_i\to I}} &X^{\tens\|n\|} &&\dTTo<{X\tilde{T}(\phi^\op)}
\\
\dTTo<{\lambda^\phi} &\ruTTo>{\tens^{i\in I}\pr_{n_i}} &&&&\luTTo<{\pr_n}
\\
XT^{\tens I} &&&\rTTo^{e(I)} &&&X\tilde T(I)
\end{diagram}

Monoidality of $e$ is expressed by the exterior of the following diagram, where \(N=\sqcup_{i\in I}N_i\).
Its commutativity is proven by postcomposing with \(\pr_m:X\tilde T(N)\to X^{\tens\|m\|}\), where \(m=(m_n)_{n\in N}\in\NN^N\) is represented also as \(m^1\oplus\dots\oplus m^I\in\NN^{N_1}\oplus\dots\oplus\NN^{N_I}\), \(m^i=(m^i_{n_i})_{n_i\in N_i}\):
\begin{diagram}
\tens^{i\in I}XT^{\tens N_i} &&&\rTTo^{\lambda^{\sqcup_{i\in I}N_i\to I}} &&&XT^{\tens N}
\\
&\rdTTo>{\tens^{i\in I}\tens^{n_i\in N_i}\pr_{m^i_{n_i}}} &&&&\ldTTo<{\tens^{n\in N}\pr_{m_n}}
\\
&&\tens^{i\in I}\tens^{n_i\in N_i}X^{\tens m^i_{n_i}} &\rTTo^{\lambda^{\sqcup_{i\in I}N_i\to I}} &\tens^{n\in N}X^{\tens m_n}
\\
\dTTo<{\tens^{i\in I}e(N_i)} &&\dTTo>{\tens^{i\in I}\lambda^{\sqcup_{n_i\in N_i}m^i_{n_i}\to N_i}} &&\dTTo>{\lambda^{\sqcup_{n\in N}m_n\to N}} &&\dTTo>{e(N)}
\\
&&\tens^{i\in I}X^{\tens\|m^i\|} &\rTTo_{\lambda^{\sqcup_{i\in I}\|m^i\|\to I}} &X^{\tens\|m\|}
\\
&\ruTTo>{\tens^{i\in I}\pr_{m^i_{n_i}}} &&&&\luTTo<{\pr_m}
\\
\tens^{i\in I}X\tilde T(N_i) &&&\rTTo^{\chi^I} &&&X\tilde T(N)
\end{diagram}
Each trapezium here commutes.
Thus, $e$ is a morphism of homotopy coalgebras.
\end{example}

The following statement is obvious.

\begin{lemma}\label{lem-lax-coalgebra-C-to-G}
Let $(C,\chi)$ be a homotopy coalgebra in a colax Monoidal category.
Let morphisms \(t(n):C(n)\to G(n)\in\cv\), $n\ge0$, be invertible.
Then there is a unique homotopy coalgebra structure on the family \((G(n))_{n\in\NN}\) such that \(t:C\to G\) is an isomorphism of homotopy coalgebras.
\end{lemma}

\begin{proposition}
Let $\cv$ be Monoidal.
A homotopy coalgebra \((C,\chi)\) is isomorphic to an ordinary coalgebra iff all \(\chi^I_{\mb1,\dots,\mb1}\) are invertible.
\end{proposition}

\begin{proof}
Necessity is obvious.

Assume that all \(\chi^I_{\mb1,\dots,\mb1}\) are invertible.
Define a homotopy coalgebra structure $(G,\gamma)$ on the family \(G(n)=C(1)^{\tens n}\), $n\ge0$, so that the family \(t(n)=(\chi^n_{\mb1,\dots,\mb1})^{-1}:C(n)\to G(n)\) were an isomorphism of homotopy coalgebras, see \lemref{lem-lax-coalgebra-C-to-G}.
Then \(t(1)=\id\) and \(\gamma^n_{\mb1,\dots,\mb1}=\id\) due to the second diagram of \eqref{eq-dia-Monoidal-trans}.
Applying \eqref{eq-CC-chi-lambda} with $N_i=\mb1$ we obtain for an arbitrary \(\phi:I\to J\in\co_\sk\) that \(\gamma^I_{\phi^{-1}1,\dots,\phi^{-1}J}=\lambda^{\phi:I\to J}\).
By \propref{pro-lax-coalgebra-ordinary} $(G,\gamma)$ is an ordinary coalgebra.
\QED\end{proof}

\begin{proposition}\label{pro-functor-Coalg-lCoalg}
The functor $\Coalg\to\hCoalg$ from \propref{pro-coalgebra-lax-Monoidal-functor} is full and faithful.
\end{proposition}

\begin{proof}
The map on morphisms
\[ \Coalg(C,G) \to \hCoalg(C,G), \quad (t:C\to G) \mapsto (t(I)=t^{\tens I}:C^{\tens I} \to G^{\tens I})_{I\in\NN}
\]
is obviously injective.

Let \(t\in\hCoalg(C,G)\).
Then the second of diagrams~\eqref{eq-dia-Monoidal-trans} with \(n_1=\dots=n_I=1\) shows that \(t(I)=t(1)^{\tens I}\) since \(\chi^I=\lambda^{\id}=\id\) and \(\gamma^I=\lambda^{\id}=\id\).
Since \(C((\con:I\to\mb1)^\op)=\Delta^{\con:I\to\mb1}=\Delta_I:C\to C^{\tens I}\) and similarly for $G$, naturality of $t$ (the first of diagrams~\eqref{eq-dia-Monoidal-trans}) implies that $t(1)$ is a homomorphism of coalgebras.
\QED\end{proof}

\begin{proposition}\label{pro-lCoalg(CXT)-V(CX)}
Let $\cv$ be Monoidal, $C$ be ordinary coalgebra in $\cv$, and $X\in\Ob\cv$.
Then the map
\[ \hCoalg(C,X\tilde{T}) \to \cv(C,X), \qquad t \mapsto \check t =(C \rTTo^{t(1)} X\hat{T} \rTTo^{\pr_1} X),
\]
is bijective.
\end{proposition}

\begin{proof}
First we prove that a morphism \(t:C\to X\tilde{T}\in\hCoalg\) is determined by $\check t$.
For an arbitrary $I\in\NN$ consider the commutative diagram
\begin{diagram}[LaTeXeqno]
C(1) &\rTTo^{t(1)} &X\tilde{T}(1)
\\
\dTTo<{C((\con:I\to\mb1)^\op)} &&\dTTo<{X\tilde{T}(\con^\op)} &\rdTTo^{\pr_I} &
\\
C(I) &\rTTo^{t(I)} &X\tilde{T}(I) &\rTTo^{\pr_{(1,\dots,1)}} &X^{\tens I}
\\
\uId<{\chi^I=\lambda^{\id}=\id} &&\uTTo<{\chi^I} &\ruTTo_{\pr_1^{\tens I}} &
\\
C(1)^{\tens I} &\rTTo^{t(1)^{\tens I}} &X\tilde{T}(1)^{\tens I}
\label{dia-C-C-C-XT-XT-XT-X}
\end{diagram}
where $C(1)=C$, \(X\tilde{T}(1)=X\hat{T}\), of course.
It follows that
\begin{equation}
\bigl( C \rTTo^{t(1)} X\hat{T} \rTTo^{\pr_I} X^{\tens I} \bigr) =\bigl( C \rTTo^{C((\con:I\to\mb1)^\op)}~\parallel_{\Delta_I} C(I) =C^{\tens I} \rTTo^{\check t^{\tens I}} X^{\tens I} \bigr).
\label{eq-CXTX-CC(I)CX}
\end{equation}
Denote by \(\hat\Delta:C\to C\hat{T}\) the only map such that \(\hat\Delta\pr_I=\Delta_I\) for all $I\in\NN$.
Thus,
\begin{equation}
t(1) =\bigl( C \rTTo^{\hat\Delta} C\hat{T} \rTTo^{\check t\hat{T}} X\hat{T} \bigr).
\label{eq-t(1)-C-XT}
\end{equation}
The monoidality of $t$ (the left bottom square of the diagram) gives
\begin{equation}
t(I) =\bigl( C^{\tens I} \rTTo^{t(1)^{\tens I}} X\hat{T}^{\tens I} \rto{\chi^I} X\tilde{T}(I) \bigr) =\bigl( C^{\tens I} \rTTo^{\hat\Delta^{\tens I}} C\hat{T}^{\tens I} \rTTo^{\und t\hat{T}^{\tens I}} X\hat{T}^{\tens I} \rto{\chi^I} X\tilde{T}(I) \bigr).
\label{eq-t(I)-CCTXTXT(I)}
\end{equation}
This proves injectivity of the map \(t\mapsto\check t\).
Bijectivity will be proven when we show that for an arbitrary \(\check t:C\to X\in\cv\) formulae~\eqref{eq-t(I)-CCTXTXT(I)} define a morphism of homotopy coalgebras \(t:C\to X\tilde{T}\).

Naturality of this $t$ means commutativity of the exterior of the following diagram for an arbitrary map \(\phi:I\to J\in\co_\sk\)
\begin{diagram}
C^{\tens J} &&\rTTo^{\hat\Delta^{\tens J}} &&C\hat{T}^{\tens J} &\rTTo^{\check t\hat{T}^{\tens J}} &X\hat{T}^{\tens J} &\rTTo^{\chi^J} &X\tilde{T}(J)
\\
&\rdTTo(4,2)_{\tens^{j\in J}\Delta_{(\phi_*n)_j}} &&&\dTTo>{\tens^{j\in J}\pr_{(\phi_*n)_j}} &&\dTTo<{\tens^{j\in J}\pr_{(\phi_*n)_j}} &\ldTTo(2,4)>{\pr_{\phi_*n}}
\\
\dTTo>{\tens^{j\in J}\Delta_{\phi^{-1}j}} &&&&\tens^{j\in J}C^{\tens(\phi_*n)_j} &\rTTo^{\tens^{j\in J}\check t^{\tens(\phi_*n)_j}} &\tens^{j\in J}X^{\tens(\phi_*n)_j}
\\
&&&\ruTTo(2,4)<{\tens^{j\in J}\lambda^{\sqcup_{i\in\phi^{-1}j}n_i\to\phi^{-1}j}} &\dTTo>{\lambda^{\sqcup_{j\in J}(\phi_*n)_j\to J}} &&\dTTo<{\lambda^{\sqcup_{j\in J}(\phi_*n)_j\to J}}
\\
\tens^{j\in J}C^{\tens\phi^{-1}j}\hspace*{-3.1em} &&&&C^{\tens\|n\|} &\rTTo^{\check t^{\tens\|n\|}} &X^{\tens\|n\|} &&\dTTo<{X\tilde{T}(\phi^\op)}
\\
&\rdTTo>{\tens^{j\in J}\tens^{i\in\phi^{-1}j}\Delta_{n_i}} &&&\uTTo>{\lambda^{\sqcup_{i\in I}n_i\to I}} &&\uTTo<{\lambda^{\sqcup_{i\in I}n_i\to I}} &\luTTo(2,4)>{\pr_n}
\\
\dTTo<{\lambda^\phi} &&\hspace*{-2.1em}\tens^{j\in J}\tens^{i\in\phi^{-1}j}C^{\tens n_i} &\rTTo^{\lambda^\phi\hspace*{-0.6em}} &\tens^{i\in I}C^{\tens n_i} &\rTTo^{\tens^{i\in I}\check t^{\tens n_i}} &\tens^{i\in I}X^{\tens n_i}
\\
&&&\ruTTo(4,2)^{\tens^{i\in I}\Delta_{n_i}} &\uTTo>{\tens^{i\in I}\pr_{n_i}} &&\uTTo<{\tens^{i\in I}\pr_{n_i}}
\\
C^{\tens I} &&\rTTo^{\hat\Delta^{\tens I}} &&C\hat{T}^{\tens I} &\rTTo^{\check t\hat{T}^{\tens I}} &X\hat{T}^{\tens I} &\rTTo^{\chi^I} &X\tilde{T}(I)
\end{diagram}
In order to prove this postcompose the equation between both exterior paths with \(\pr_n:X\tilde{T}(I)\to X^{\tens\|n\|}\) for an arbitrary \(n=(n_1,\dots,n_I)\in\NN^I\).
The obtained tadpole partitioned into commutative cells as in the above diagram clearly commutes.

Monoidality of $t$ is expressed by the exterior of the following commutative diagram
\begin{diagram}
\tens^{i\in I}C^{\tens n_i} &\rTTo^{\tens^{i\in I}\hat\Delta^{\tens n_i}} &\tens^{i\in I}C\hat{T}^{\tens n_i} &\rTTo^{\tens^{i\in I}\check t\hat{T}^{\tens n_i}} &\tens^{i\in I}X\hat{T}^{\tens n_i} &\rTTo^{\tens^{i\in I}\chi^{n_i}} &\tens^{i\in I}X\tilde{T}(n_i)
\\
\dTTo>{\lambda^{\sqcup_{i\in I}n_i\to I}} &&\dTTo>{\lambda^{\sqcup_{i\in I}n_i\to I}} &&\dTTo>{\lambda^{\sqcup_{i\in I}n_i\to I}} &&\dTTo>{\chi^I}
\\
C^{\tens\|n\|} &\rTTo^{\hat\Delta^{\tens\|n\|}} &C\hat{T}^{\tens\|n\|} &\rTTo^{\check t\hat{T}^{\tens\|n\|}} &X\hat{T}^{\tens\|n\|} &\rTTo^{\chi^{\|n\|}} &X\tilde{T}(\|n\|)
\end{diagram}
Therefore, \((t(I))_{I\in\NN}\) is a homotopy coalgebra morphism.
\QED\end{proof}

Thanks to the above proposition we call $X\tilde{T}$, $X\in\Ob\cv$, a relatively cofree homotopy coalgebra.

\begin{proposition}\label{pro-f-XT-Y-defines-f}
Let a morphism \(f=(f_k:X^{\tens k}\to Y)_{k\in\NN}:XT\to Y\in\cv\) have finite support, that is, there exists \(K\in\NN\) such that $f_k=0$ for all $k>K$.
For each $I\in\co_\sk$ define a morphism \(\hat{f}(I):X\tilde{T}(I)\to Y\tilde{T}(I)\) by the equations, $n\in\NN^I$, \([K]=\{0,1,\dots,K\}\subset\NN\),
\begin{multline}
\bigl\langle X\tilde{T}(I) \rTTo^{\hat{f}(I)} Y\tilde{T}(I) \rto{\pr_n} Y^{\tens\|n\|} \bigr\rangle
\\
\hskip\multlinegap =\sum_{k\in[K]^{\|n\|}} \bigl\langle X\tilde{T}(I) \rTTo^{X\tilde{T}((\sqcup_{i\in I}n_i\to I)^\op)} X\tilde{T}(\|n\|) \rto{\pr_k} X^{\tens\|k\|} \hfill
\\
\hfill \rTTo^{(\lambda^{\sqcup_{p\in\|n\|}k_p\to\|n\|})^{-1}} \tens^{p\in\|n\|}X^{\tens k_p} \rTTo^{\tens^{p\in\|n\|}f_{k_p}} Y^{\tens\|n\|} \bigr\rangle \quad
\\
\hskip\multlinegap =\sum_{k\in[K]^{\|n\|}} \bigl\langle X\tilde{T}(I) \rTTo^{\pr_{(\sqcup_{i\in I}n_i\to I)_*k}} X^{\tens\|k\|} \rTTo^{(\lambda^{\sqcup_{p\in\|n\|}k_p\to\|n\|})^{-1}} \tens^{p\in\|n\|}X^{\tens k_p}
\ifx\chooseClass1
\hfill \\
\fi
\rTTo^{\tens^{p\in\|n\|}f_{k_p}} Y^{\tens\|n\|} \bigr\rangle.
\label{eq-XT(I)-f(I)-YT(I)-pr-Y}
\end{multline}
By definition \eqref{eq-(f*n)j} for each \(k\in\NN^{\|n\|}\) there is \((\sqcup_{i\in I}n_i\to I)_*k\in\NN^I\) with coordinates
\begin{equation}
((\sqcup_{i\in I}n_i\to I)_*k)_i =\sum_{p\in n_i\hookrightarrow\|n\|} k_p, \qquad i\in I,
\label{eq-nIk-k-iI}
\end{equation}
\begin{multline*}
(\sqcup_{i\in I}n_i\to I)_*k
\ifx\chooseClass1
\\
\fi
=(k_1+\dots+k_{n_1},k_{n_1+1}+\dots+k_{n_1+n_2},k_{n_1+n_2+1}+\dots+k_{n_1+n_2+n_3},\dots).
\end{multline*}
Then the family \(\hat{f}:X\tilde{T}\to Y\tilde{T}\) is a morphism of homotopy coalgebras.
\end{proposition}

\begin{proof}
Equality of the two expressions for \(\hat{f}(I)\pr_n\) follows from definition \eqref{eq-XT(fop)prn}.

For an arbitrary map \(\phi:I\to J\in\co_\sk\) we have to prove naturality of $\hat{f}$, expressed by the exterior of the following diagram.
After postcomposing with \(\pr_n:Y\tilde{T}(I)\to Y^{\tens\|n\|}\), $n\in\NN^I$, the diagram becomes commutative:
\begin{diagram}
X\tilde{T}(J) &&&\rTTo_{\hat{f}(J)} &&&Y\tilde{T}(J)
\\
&\rdTTo>{\pr_{(\sqcup_{j\in J}(\phi_*n)_j\to J)_*k}} &&&&\ldTTo<{\pr_{\phi_*n}}
\\
\dTTo>{X\tilde{T}(\phi^\op)} &&\sum_{k\in[K]^{\|n\|}} \Bigl\langle X^{\tens\|k\|} &\rto[\hspace*{-0.7em}(\lambda^{\sqcup_{p\in\|n\|}k_p\to\|n\|})^{-1}\hspace*{-0.7em}]{} \tens^{p\in\|n\|}X^{\tens k_p} \rto[\hspace*{-0.7em}\tens^{p\in\|n\|}f_{k_p}\hspace*{-0.7em}]{} &Y^{\tens\|n\|} \Bigr\rangle &&\dTTo<{Y\tilde{T}(\phi^\op)}
\\
&\ruTTo>{\pr_{(\sqcup_{i\in I}n_i\to I)_*k}} &&&&\luTTo<{\pr_n}
\\
X\tilde{T}(I) &&&\rTTo_{\hat{f}(I)} &&&Y\tilde{T}(I)
\end{diagram}
Notice that
\[ \phi_*(\sqcup_{i\in I}n_i\to I)_* =(\sqcup_{i\in I}n_i\to I\rto\phi J)_* =(\sqcup_{j\in J}(\phi_*n)_j\to J)_*,
\]
which implies commutativity of the left triangle.

Monoidality of $\hat{f}$ means commutativity of the exterior of the following diagram.
Its postcomposition with \(\pr_m:Y\tilde{T}(N)\to Y^{\tens\|m\|}\), \(m=(m_n)\in\NN^N\), \(N=\sqcup_{i\in I}N_i\), partitioned into commutative cells, involves two sums of braced paths with equal number of summands, indexed by \((k^1,\dots,k^I)\in[K]^{\|m^1\|}\times\dots\times[K]^{\|m^I\|}\) and \(k\in[K]^{\|m\|}\) corresponding to each other, \(m=m^1\oplus\dots\oplus m^I\in\NN^{N_1}\oplus\dots\oplus\NN^{N_I}\):
\begin{diagram}[nobalance,h=2.2em,w=2.2em]
\tens^{i\in I}X\tilde{T}(N_i) &&&\rTTo^{\chi^I} &&&X\tilde{T}(N)
\\
&\rdTTo>{\overbrace{\sss\tens^{i\in I}\pr_{(\sqcup_{n_i\in N_i}m^i_{n_i}\to N_i)_*k^i}}} &&\vphantom{\sum^1_1} &&\ldTTo<{\overbrace{\sss\pr_{(\sqcup_{n\in N}m_n\to N)_*k}}}
\\
&\hspace*{-1.5em}\sum_{k^1\in[K]^{\|m^1\|}} \hspace*{-.5em}\dots\hspace*{-.5em}\sum_{k^I\in[K]^{\|m^I\|}}\hspace*{-1em} &\tens^{i\in I}X^{\tens\|k^i\|} &\rTTo^{\hspace*{-2.4em}\lambda^{\sqcup_{i\in I}\|k^i\|\to I}\hspace*{-2.4em}} &X^{\tens\|k\|} &\hspace*{-1em}\sum_{k\in[K]^{\|m\|}}\hspace*{-0em}
\\
&&\dTTo<{\tens^{i\in I}(\lambda^{\sqcup_{p\in\|m^i\|}k^i_p\to\|m^i\|})^{-1}} &\vphantom\sum &\dTTo>{(\lambda^{\sqcup_{p\in\|m\|}k_p\to\|m\|})^{-1}}
\\
\dTTo>{\tens^{i\in I}\hat{f}(N_i)} &&\hspace*{-2.3em}\tens^{i\in I}\tens^{p\in\|m^i\|}X^{\tens k^i_p} &\rTTo^{\hspace*{-2.4em}\lambda^{\sqcup_{i\in I}\|m^i\|\to I}\hspace*{-2.4em}} &\tens^{p\in\|m\|}X^{\tens k_p}\hspace*{-2.3em} &&\dTTo<{\hat{f}(N)}
\\
&&\dTTo<{\tens^{i\in I}\tens^{p\in\|m^i\|}f_{k^i_p}} &&\dTTo>{\tens^{p\in\|m\|}f_{k_p}}
\\
&&\underbrace{\tens^{i\in I}Y^{\tens\|m^i\|}} &\rTTo^{\hspace*{-2.4em}\lambda^{\sqcup_{i\in I}\|m^i\|\to I}\hspace*{-2.4em}} &\underbrace{Y^{\tens\|m\|}}
\\
&\ruTTo>{\tens^{i\in I}\pr_{m^i}} &&&&\luTTo<{\pr_m}
\\
\tens^{i\in I}Y\tilde{T}(N_i) &&&\rTTo^{\chi^I} &&&Y\tilde{T}(N)
\end{diagram}
Notice that \((\sqcup_{n\in N}m_n\to N)_*k\in\NN^N\) is the sum of \((\sqcup_{n_i\in N_i}m^i_{n_i}\to N_i)_*k^i\in\NN^{N_i}\) under identification \(\NN^N\simeq\NN^{N_1}\oplus\dots\oplus\NN^{N_I}\), since
\[ ((\sqcup_{n\in N}m_n\to N)_*k)_n =\hspace*{-1em} \sum_{p\in m_n\hookrightarrow\|m\|} k_p, \quad ((\sqcup_{n_i\in N_i}m^i_{n_i}\to N_i)_*k^i)_{n_i} =\hspace*{-1em} \sum_{p\in m^i_{n_i}\hookrightarrow\|m^i\|} k^i_p.
\]
Therefore, the top trapezium commutes.
Thus, $\hat{f}$ is a homotopy coalgebra morphism.
\QED\end{proof}

\begin{remark}
In the particular case $I=1$ \propref{pro-f-XT-Y-defines-f} gives for $n\in\NN$
\[ \hat{f}(1)\pr_n =\sum_{k\in[K]^n} \Bigl\langle X\hat{T} \rTTo^{\pr_{\|k\|}} X^{\tens\|k\|} \rTTo^{(\lambda^{\sqcup_{p\in n}k_p\to n})^{-1}} \tens^{p\in n}X^{\tens k_p} \rTTo^{\tens^{p\in n}f_{k_p}} Y^{\tens n} \Bigr\rangle.
\]
\end{remark}

\begin{example}\label{exa-XT-eX-X-tilde-T}
Given a morphism \(f:XT\to Y\) with finite support in the sense of \propref{pro-f-XT-Y-defines-f} such that $f_0=0$ one produces a homomorphism of coalgebras \(\hat{f}:XT\to YT\) given by the explicit formula (ignoring $\lambda$)
\begin{equation}
\inj_m\hat{f}\pr_n =\sum_{k_1+\dots+k_n=m} f_{k_1}\tdt f_{k_n}: X^{\tens m} \to Y^{\tens n}.
\label{eq-inm-f-prn-sum}
\end{equation}
We claim that there is a commutative square of homotopy coalgebra morphisms
\begin{diagram}
XT &\rTTo^{e_X} &X\tilde{T}
\\
\dTTo<{\hat{f}} &= &\dTTo>{\hat{f}}
\\
YT &\rTTo^{e_Y} &Y\tilde{T}
\end{diagram}
In fact, due to \propref{pro-lCoalg(CXT)-V(CX)} it suffices to verify that
\[ \bigl(XT \rTTo^{\hat{f}} YT \rTTo^{e_Y(1)} Y\hat{T} \rTTo^{\pr_1} Y\bigr) =\bigl(XT \rTTo^{e_X(1)} X\hat{T} \rTTo^{\hat{f}(1)} Y\hat{T} \rTTo^{\pr_1} Y\bigr).
\]
One easily finds that the both sides are equal to \(f:XT\to Y\).

Notice that the morphism \(XT\rTTo^{e_X} X\tilde{T}\rTTo^{\hat{f}} Y\tilde{T}\) exists without the assumption $f_0=0$.
Its component \(XT\rTTo^{e_X(1)} X\hat{T}\rTTo^{\hat{f}(1)} Y\hat{T}\rTTo^{\pr_n} Y^{\tens n}\) is
\[ \sum_{k\in[K]^n} \Bigl\langle XT \rTTo^{\pr_{\|k\|}} X^{\tens\|k\|} \rTTo^{(\lambda^{\sqcup_{p\in n}k_p\to n})^{-1}} \tens^{p\in n}X^{\tens k_p} \rTTo^{\tens^{p\in n}f_{k_p}} Y^{\tens n} \Bigr\rangle,
\]
which coincides with \eqref{eq-inm-f-prn-sum} for $k_p\ge0$.
Furthermore, the composition \(e_X\hat{f}\) coincides with the morphism \(\hat{f}:XT\to Y\tilde{T}\) obtained from \(f:XT\to Y\) via \propref{pro-lCoalg(CXT)-V(CX)}.
In fact,
\[ \sum_{k\in\NN} \bigl\langle XT \rTTo^{\pr_k} X^{\tens k} \rTTo^{f_k} Y \bigr\rangle =f.
\]
The morphism \(\hat{f}:XT\to Y\tilde{T}\) is well defined for an arbitrary \(f:XT\to Y\), not necessarily having a finite support.
\end{example}

We say that a morphism \(\phi:\prod_{\alpha\in A}M_\alpha\to N\in\cv\) is \emph{supported on a subset} $B\subset A$ if $\phi$ admits a presentation
\begin{equation}
\prod_{\alpha\in A}M_\alpha \rEpi \prod_{\alpha\in B}M_\alpha \rTTo^{\text`\phi\text'} N.
\label{eq-MM-phi-N}
\end{equation}
The first arrow is a split epimorphism, hence the morphism `$\phi$' is unique.
If $B$ is finite,
\[ \text`\phi\text' =\bigl( \prod_{\alpha\in B}M_\alpha \rto\sim \bigoplus_{\alpha\in B}M_\alpha \rto\sim \coprod_{\alpha\in B}M_\alpha \to N \bigr)
\]
reduces to a family of components \(\text`\phi\text'_\alpha:M_\alpha\to N\), $\alpha\in B$.

For instance, due to \eqref{eq-nIk-k-iI} map~\eqref{eq-XT(I)-f(I)-YT(I)-pr-Y} is supported on
\[ [Kn] =\{ l\in\NN^I \mid l\le Kn \} =\{ l\in\NN^I \mid \forall i\in I\; l_i\le Kn_i \}.
\]

\begin{proposition}\label{pro-f-supported-on[Kn]}
Let $K\in\NN$.
Let a morphism \(\hat{f}:X\tilde{T}\to Y\tilde{T}\in\hCoalg\) be such that for all \(n\in\NN^I\) the morphism \(X\tilde{T}(I)\rTTo^{\hat{f}(I)} Y\tilde{T}(I)\rto{\pr_n} Y^{\tens\|n\|}\) is supported on $[Kn]$.
Then there is a unique morphism \(f:XT\to Y\) supported on $[K]$ such that \(\hat{f}\) is obtained from $f$ as in \eqref{eq-XT(I)-f(I)-YT(I)-pr-Y}.
\end{proposition}

\begin{proof}
Since \(f=\bigl\langle XT\rTTo^{e(1)} X\tilde{T}(1)\rTTo^{\hat{f}(1)} Y\tilde{T}(1)\rto{\pr_1} Y\bigr\rangle\) uniqueness of $f$ is obvious.
Let us prove that $\hat{f}$ is restored from this $f$ via equation~\eqref{eq-XT(I)-f(I)-YT(I)-pr-Y}.
Similarly to \eqref{dia-C-C-C-XT-XT-XT-X} there is a commutative diagram
\begin{diagram}[LaTeXeqno]
X\hat{T} &\rTTo^{\hat{f}(1)} &Y\hat{T}
\\
\dTTo<{X\tilde{T}((\con:I\to\mb1)^\op)} &&\dTTo<{Y\tilde{T}(\con^\op)} &\rdTTo^{\pr_I} &
\\
X\tilde{T}(I) &\rTTo^{\hat{f}(I)} &Y\tilde{T}(I) &\rTTo^{\pr_{(1,\dots,1)}} &Y^{\tens I}
\\
\uTTo<{\chi^I} &&\uTTo<{\chi^I} &\ruTTo_{\pr_1^{\tens I}} &
\\
X\hat{T}^{\tens I} &\rTTo^{\hat{f}(1)^{\tens I}} &Y\hat{T}^{\tens I}
\label{dia-XT-XT-XT-YT-YT-YT-Y}
\end{diagram}
In particular, both above pentagons commute.
All three compositions \(\hat{f}(-)\pr_-\) in this diagram can be written in form~\eqref{eq-MM-phi-N} by hypothesis.
The above pentagons take the form
\begin{diagram}[h=2.4em,nobalance]
X\hat{T} &\rEpi &\prod_{m\in[KI]}X^{\tens m}
\\
\dTTo<{X\tilde{T}((\con:I\to\mb1)^\op)} &&\dTTo<h &\rdTTo^{\text`\hat{f}(1)\pr_I\text'} &
\\
X\tilde{T}(I) &\rEpi &\prod_{n\in[K]^I}X^{\tens\|n\|} &\rTTo^{\text`\hat{f}(I)\pr_{(1,\dots,1)}\text'\qquad} &Y^{\tens I}
\\
\uTTo<{\chi^I} &&\uTTo<g &\ruTTo_{\text`f\text'^{\tens I}} &
\\
X\hat{T}^{\tens I} &\rEpi &\bigl(\prod_{m\in[K]}X^{\tens m}\bigr)^{\tens I}
\end{diagram}
The lower horizontal arrow being a tensor product of split epimorphisms is a split epimorphism itself.
The both triples of neighbour arrows on the left close up to a commutative square by new morphisms $h$ and $g$.
Namely, for all $n\in[K]^I$
\begin{equation}
\begin{split}
h\pr_n &=\pr_{\|n\|}: \prod_{m\in[KI]}X^{\tens m} \to X^{\tens\|n\|},
\\
g\pr_n &=\Bigl( \bigotimes_{i\in I} \prod_{m_i\in[K]}X^{\tens m_i} \rTTo^{\tens^{i\in I}\pr_{n_i}\;} \bigotimes_{i\in I} X^{\tens n_i} \rTTo^\lambda X^{\tens\|n\|} \Bigr),
\end{split}
\label{eq-hprn-gprn}
\end{equation}
see \eqref{eq-XT(fop)prn} and \eqref{eq-chi-pr-X-X-X}.
Since the top and the bottom horizontal arrows are epimorphisms, both triangles on the right commute.
Notice that $g$ is an isomorphism!
Therefore, \(\text`\hat{f}(1)\pr_I\text'=h\cdot g^{-1}\cdot\text`f\text'^{\tens I}\) is determined in a unique way as well as
\begin{multline}
\hat{f}(1)\pr_I =\Bigl\langle X\hat{T} \rTTo^{X\tilde{T}((\con:I\to\mb1)^\op)} X\tilde{T}(I) \rEpi \prod_{n\in[K]^I}X^{\tens\|n\|} \rTTo^{g^{-1}}_\sim \bigl(\prod_{m\in[K]}X^{\tens m}\bigr)^{\tens I}
\\
\rTTo^{\text`f\text'^{\tens I}} Y^{\tens I} \Bigr\rangle.
\label{eq-f(1)prI}
\end{multline}
Thus, all morphisms $\hat{f}$ that satisfy the hypothesis of proposition with fixed \(f=\hat{f}(1)\pr_1\) have the same $\hat{f}(1)$.
Clearly it is given by \eqref{eq-XT(I)-f(I)-YT(I)-pr-Y} for $I=1$.

The lower commutative square in diagram~\eqref{dia-XT-XT-XT-YT-YT-YT-Y} can be written as the pentagon commutative for arbitrary $n\in\NN^I$ in
\begin{diagram}
X\tilde{T}(I) &\rTTo^{\hat{f}(I)} &Y\tilde{T}(I) &\rTTo^{\pr_n} &Y^{\tens\|n\|}
\\
\uTTo<{\chi^I} &&&\ruTTo^{\chi^I\pr_n} &\uTTo>\lambda
\\
X\hat{T}^{\tens I} &\rTTo_{\hat{f}(1)^{\tens I}} &Y\hat{T}^{\tens I} &\rTTo_{\tens^{i\in I}\pr_{n_i}\;} &\otimes^{i\in I} Y^{\tens n_i}
\end{diagram}
By hypothesis the top and the bottom pairs of horizontal arrows can be presented in form~\eqref{eq-MM-phi-N}.
Thus the exterior of the diagram
\begin{diagram}[h=2.4em]
X\tilde{T}(I) &\rEpi &\prod_{m\in[Kn]}X^{\tens\|m\|} &\rTTo^{\text`\hat{f}(I)\pr_n\text'} &Y^{\tens\|n\|}
\\
\uTTo<{\chi^I} &&\uTTo<j &&\uTTo>\lambda
\\
X\hat{T}^{\tens I} &\rEpi &\bigotimes_{i\in I} \prod_{k_i\in[Kn_i]}X^{\tens k_i} &\rTTo^{\tens^{i\in I}\text`\hat{f}(1)\pr_{n_i}\text'\;} &\bigotimes_{i\in I} Y^{\tens n_i}
\end{diagram}
commutes.
There is an arrow $j$ which closes up the three arrows on the left to a commutative square, namely for all $m\in[Kn]$
\begin{equation}
j\pr_m =\Bigl( \bigotimes_{i\in I} \prod_{k_i\in[Kn_i]}X^{\tens k_i} \rTTo^{\tens^{i\in I}\pr_{m_i}\;} \bigotimes_{i\in I} X^{\tens m_i} \rTTo^\lambda X^{\tens\|m\|} \Bigr),
\label{eq-jprm-XXX}
\end{equation}
see \eqref{eq-chi-pr-X-X-X}.
Since the bottom left horizontal arrow is an epimorphism, the right square commutes as well.
Notice that $j$ is an isomorphism.
Hence there is a unique possible value for \(\text`\hat{f}(I)\pr_n\text'=j^{-1}[\tens^{i\in I}\text`\hat{f}(1)\pr_{n_i}\text']\lambda\) and for $\hat{f}(I)$ itself.
Clearly, $\hat{f}(I)$ is reconstructed from $f$ via \eqref{eq-XT(I)-f(I)-YT(I)-pr-Y}.
\QED\end{proof}

Equation~\eqref{eq-f(1)prI} can be also presented in the form
\begin{align}
\hat{f}(1)\pr_I &=\Bigl\langle X\hat{T} \rEpi \prod_{m\in[KI]}X^{\tens m} \rTTo^{\Delta^{\restr}} \bigl(\prod_{k\in[K]}X^{\tens k}\bigr)^{\tens I} \rTTo^{\text`f\text'^{\tens I}} Y^{\tens I} \Bigr\rangle,
\label{eq-f(1)prI-XX-Delta-rest-XY}
\\
\Delta^{\restr}\big|_{X^{\tens m}} &=\hspace*{-0.5em}\sum_{k_1+\dots+k_I=m}^{k_i\le K} \biggl\langle X^{\tens m} \rto[\sim]{\lambda^{-1}} \bigotimes_{i\in I} X^{\tens k_i} \rTTo^{\tens^{i\in I}\inj_{k_i}\;} \bigl(\prod_{k\in[K]}X^{\tens k}\bigr)^{\tens I} \biggr\rangle.
\label{eq-Delta-restr}
\end{align}
$\Delta^{\restr}$ is the restricted cut comultiplication.

Suppose that \(\hat{f}:X\tilde{T}\to Y\tilde{T}\) and \(\hat{g}:Y\tilde{T}\to Z\tilde{T}\) of $\hCoalg$ are obtained via \eqref{eq-XT(I)-f(I)-YT(I)-pr-Y} from \(f:XT\to Y\) and $g:YT\to Z$ supported on $[K]$ and $[L]$, respectively.
The top maps in the diagram for $n\in\NN^I$
\begin{diagram}
X\tilde{T}(I) &\rTTo^{\hat{f}(I)} &Y\tilde{T}(I) &\rTTo^{\hat{g}(I)} &Z\tilde{T}(I)
\\
\dEpi &&\dEpi &&\dTTo>{\pr_n}
\\
\prod_{m\le KLn}X^{\tens\|m\|} &\rTTo^{\exists!} &\prod_{l\le Ln}X^{\tens\|l\|} &\rTTo^{\exists!} &Z^{\tens\|n\|}
\end{diagram}
factorize through bottom arrows.
Therefore the composition \(\hat{f}(I)\hat{g}(I)\pr_n\) is supported on \([KLn]\subset\NN^I\).
By \propref{pro-f-supported-on[Kn]} the morphism \(\hat{h}=\hat{f}\hat{g}:X\tilde{T}\to Z\tilde{T}\) can be obtained from a unique morphism $h:XT\to Z$ supported on $[KL]$.
Namely,
\begin{align*}
h &= \bigl(XT \rTTo^{e_X(1)} X\hat{T} \rTTo^{\hat{f}(1)} Y\hat{T} \rTTo^{\hat{g}(1)} Z\hat{T} \rTTo^{\pr_1} Z\bigr)
\\
&= \sum_{k_1,\dots,k_l\le K}^{l\le L} \Bigl(XT \rTTo^{\pr_{k_1+\dots+k_l}} X^{\tens(k_1+\dots+k_l)} \rTTo^{f_{k_1}\tdt f_{k_l}} Y^{\tens l} \rTTo^{g_l} Z\Bigr),
\end{align*}
using \eqref{eq-f(1)prI-XX-Delta-rest-XY}.
Its components are
\begin{equation}
h_m =\sum_{k_1+\dots+k_l=m}^{k_i\le K,\;l\le L} \Bigl(X^{\tens m} \rTTo^{\lambda^{-1}}_\sim X^{\tens k_1} \tdt X^{\tens k_l} \rTTo^{f_{k_1}\tdt f_{k_l}} Y^{\tens l} \rTTo^{g_l} Z\Bigr).
\label{eq-asdfghjkl}
\end{equation}

Clearly \(\id:X\tilde{T}\to X\tilde{T}\) is of the form $\hat{f}$ for \(f=\pr_1:XT\to X\) supported on $\{1\}\subset[1]$.
Thus, there is a subcategory $\fhCoalg\subset\hCoalg$ (of homotopy coalgebras cofree with respect to ordinary coalgebras), whose objects are $X\tilde{T}$, $X\in\cv$, and morphisms \(\hat{f}:X\tilde{T}\to Y\tilde{T}\) come from \(f:XT\to Y\) with finite support as in \eqref{eq-XT(I)-f(I)-YT(I)-pr-Y}.

\subsection{Coderivations for homotopy coalgebras}
We define coderivations of homotopy coalgebras, describe coderivations from an ordinary coalgebra to $X\tilde{T}$, $X\in\Ob\cv$, give a supply of coderivations between $X\tilde{T}$ and $Y\tilde{T}$.

A \emph{natural transformation \(\xi:C\to G:\cc\to\overline\cv\) of degree $k$} is defined as a family \((\xi(I)\in\und\cv(C(I),G(I))^k)_{I\in\Ob\cc}\) satisfying relation \(\xi(J)G(\psi)=C(\psi)\xi(I)\) for all \(\psi:J\to I\in\cc\) (a commutative square in $\overline\cv$).
Bijection~\eqref{eq-V(XZ)k-V(XZk)} allows to rewrite such a natural transformation as an ordinary natural transformation \(\xi:C\to G[k]:\cc\to\cv\).
Nevertheless we shall use the first (degree $k$) version.

\begin{definition}\label{def-(fg)-coderivation}
Let \(f,g:(C,\chi)\to(G,\gamma)\) be morphisms of homotopy coalgebras.
An \emph{$(f,g)$-coderivation} \(\xi:(C,\chi)\to(G,\gamma)\) is a natural transformation \(\xi:C\to G:\co_\sk^\op\to\cv\) of certain degree such that
\begin{diagram}[width=5em,LaTeXeqno]
\tens^{i\in I}C(n_i) &\rTTo^{\chi^I} &C(n_1+\dots+n_I)
\\
\dTTo<{\sum_{u=1}^I\tens^I(\sS{^{u-1}}f,\xi,\sS{^{I-u}}g)} &= &\dTTo>\xi
\\
\tens^{i\in I}G(n_i) &\rTTo^{\gamma^I} &G(n_1+\dots+n_I)
\label{dia-CCGG-(fg)-coderivation}
\end{diagram}
\end{definition}

Here rigorous \(\tens^I(\sS{^{u-1}}f,\xi,\sS{^{I-u}}g)\) should be informally interpreted as \(f^{\tens(u-1)}\tens\xi\tens g^{\tens(I-u)}\).
When $g=f$, $(f,f)$-coderivations are called $f$\n-coderivations.
When \(f=g=\id\), $(\id,\id)$-coderivations are simply called coderivations.

\begin{example}
Let $\cv=\gr_\kk$.
For any $a\in\ZZ$ consider the coaugmented algebra \(\DD=\kk\oplus\kk[a]\) in $\cv$ with the unit \(\inj_1:\kk\to\DD\), the multiplication \(\kk[a]\cdot\kk[a]=0\) and the coaugmentation \(\pr_1:\DD\to\kk\).
Let $\cv_\DD$ be the category of commutative $\DD$\n-bimodules in $\gr_\kk$, which is the same as $\gr_\DD$.
There are monoidal functors $\cv\to\cv_\DD$, \(X\mapsto X_\DD=X\tens_\kk\DD\), and $\cv_\DD\to\cv$, \(W\mapsto W\tens_\DD\kk\).

Let \(f:(C,\chi)\to(G,\gamma)\) be a morphism of homotopy coalgebras in $\gr_\kk$.
We claim that the relevant infinitesimal morphisms \(\tilde{f}:C_\DD\to G_\DD:\co_\sk^\op\to\cv_\DD\in\hCoalg\) (such that \(\tilde{f}\tens_\DD\kk=f\)) are in bijection with $(f,f)$\n-coderivations of degree $a$.
In fact, a $\DD$\n-linear map $\tilde{f}(I)$ is determined by
\begin{multline*}
\bigl\langle C(I) \rTTo^{1\tens\inj_1} C(I)\tens_\kk\DD \rTTo^{\tilde{f}(I)} G(I)\tens_\kk\DD \bigr\rangle 
\ifx\chooseClass1
\\
\fi
=\bigl\langle (f,\xi\tens\sigma^a): C(I) \to G(I)\tens_\kk\kk \oplus G(I)\tens_\kk\kk[a] \bigr\rangle,
\end{multline*}
where \(\xi(I):C(I)\to G(I)\) is a $\kk$\n-linear map of degree $a$.
It's easy to see that naturality of $\tilde{f}$ is equivalent to naturality of $\xi$.
Monoidality of $\tilde{f}$, see \eqref{eq-dia-Monoidal-trans}(b), can be written as
\begin{multline}
\bigl\langle \tens^{i\in I}C(n_i) \rto{\chi^I} C(n_1+\dots+n_I) \rTTo^{\xi\tens\sigma^a} G(n_1+\dots+n_I)\tens\kk[a] \bigr\rangle
\\
=\sum_{u\in I} \bigl\langle \tens^{i\in I}C(n_i) \rTTo^{\tens^I(\sS{^{u-1}}f,\xi\tens\sigma^a,\sS{^{I-u}}f)} \tens^I\bigl((G(n_i))_{i=1}^{u-1},G(n_u)\tens\kk[a],(G(n_i))_{i=u+1}^I\bigr) \hfill
\\
\rTTo^{1^{\tens(u-1)}\tens\sigma^{-a}\tens1^{\tens(I-u)}\hspace*{-.3em}} \tens^{i\in I}G(n_i) \rto{\sigma^a} (\tens^{i\in I}G(n_i))\tens\kk[a] \rto{\hspace*{-.3em}\gamma\tens1\hspace*{-.3em}} G(n_1+\dots+n_I)\tens\kk[a] \bigr\rangle
\label{eq-chi-sum-sigma}
\end{multline}
since here
\[ (1^{\tens(u-1)}\tens\sigma^{-a}\tens1^{\tens(I-u)})\sigma^a =(I+1,I,\dots,u+2,u+1)_c
\]
is the signed permutation.
In fact, for an arbitrary $X\in\Ob\cv$
\[ \bigl( \kk[a]\tens X \rTTo^{\sigma^{-a}\tens1} X \rto{\sigma^a} X\tens\kk[a] \bigr) =c.
\]
Obtained equation \eqref{eq-chi-sum-sigma} is equivalent to \eqref{dia-CCGG-(fg)-coderivation}.
\end{example}

\begin{proposition}
Let \(f,g:C\to G\) be morphisms of ordinary coalgebras in $\cv$.
Then $(f,g)$-coderivations in the sense of \defref{def-(fg)-coderivation} are the same as ordinary $(f,g)$-coderivations, morphisms \(\xi:C\to G\in\und\cv^\bull\) such that
\begin{equation}
\xi\Delta =\Delta(f\tens\xi+\xi\tens g): C \to G\tens G.
\label{eq-xi-Delta-Delta(foxi-xiog)}
\end{equation}
\end{proposition}

\begin{proof}
Assume that $\xi(I)$, $I\in\NN$, satisfy conditions of \defref{def-(fg)-coderivation}.
Then for \(\phi=\con:\mb2\to\mb1\) the both squares of
\begin{diagram}
C &\rTTo^\Delta &C(2) &\lId^{\chi^2} &C\tens C
\\
\dTTo<\xi &&\dTTo>{\xi(2)} &&\dTTo>{f\tens\xi+\xi\tens g}
\\
G &\rTTo^\Delta &G(2) &\lId^{\chi^2} &G\tens G
\end{diagram}
commute.

Assume now that \eqref{eq-xi-Delta-Delta(foxi-xiog)} holds.
Using \eqref{dia-CCGG-(fg)-coderivation} for \(n_1=\dots=n_I=1\) we find that
\[ \xi(I) =\sum_{u=1}^I \tens^I(\sS{^{u-1}}f,\xi,\sS{^{I-u}}g): C^{\tens I} \to G^{\tens I}
\]
for $I\in\NN$.
These expressions satisfy \eqref{dia-CCGG-(fg)-coderivation} for arbitrary \(n_i\in\NN\), $i\in I$.
Equation~\eqref{eq-xi-Delta-Delta(foxi-xiog)} implies by induction that for all $I\in\NN$
\[ \xi\Delta_I =\Delta_I\xi(I): C \to G^{\tens I}.
\]
Hence, $\xi$ is natural with respect to maps \(\phi=\con:I\to\mb1\in\co_\sk\).
This implies naturality, \(\xi(J)\Delta^\phi_G=\Delta^\phi_C\xi(I)\), for an arbitrary \(\phi:I\to J\in\co_\sk\) due to presentation of $\Delta^\phi$ via $\Delta_{\phi^{-1}j}$.
\QED\end{proof}

\begin{proposition}\label{pro-(fg)coder(CXT)-V(CX)}
Let $C$ be an ordinary coalgebra in $\cv$, let $X\in\Ob\cv$ and let \(f,g:C\to X\tilde{T}\in\hCoalg\).
Then the map of graded abelian groups
\[ (f,g)\coder(C,X\tilde{T}) \to \overline\cv(C,X), \qquad \xi \mapsto \check\xi =(C \rTTo^{\xi(1)} X\hat{T} \rTTo^{\pr_1} X),
\]
is bijective.
\end{proposition}

\begin{proof}
Let us show that this map is injective.
First we notice that $\xi(1)$ is determined by $\check\xi$.
For an arbitrary \(I\in\NN\) consider the commutative diagram
\begin{diagram}[h=2.2em,LaTeXeqno]
C(1) &\rTTo^{\xi(1)} &X\tilde{T}(1)
\\
\dTTo>{C((\con:I\to\mb1)^\op)}~=<{\Delta_I} &&\dTTo<{X\tilde{T}(\con^\op)} &\rdTTo^{\pr_I} &
\\
C(I) &\rTTo^{\xi(I)} &X\tilde{T}(I) &\rTTo^{\pr_{(1,\dots,1)}} &X^{\tens I}
\\
\uId<{\chi^I=\lambda^{\id}=\id} &&\uTTo<{\chi^I} &\ruTTo_{\pr_1^{\tens I}} &
\\
C(1)^{\tens I} &\rTTo^{\hspace*{-2em}\sum_{i=1}^If(1)^{\tens(i-1)}\tens\xi(1)\tens g(1)^{\tens(I-i)}} &X\tilde{T}(1)^{\tens I}
\label{dia-C-xi-XT(1)}
\end{diagram}
It follows that
\begin{equation}
\bigl( C \rTTo^{\xi(1)} X\hat{T} \rTTo^{\pr_I} X^{\tens I} \bigr) =\bigl( C \rTTo^{\Delta_I} C^{\tens I} \rTTo^{\sum_{i=1}^I\check f^{\tens(i-1)}\tens\check\xi\tens\check g^{\tens(I-i)}} X^{\tens I} \bigr).
\label{eq-CXTX-CCX}
\end{equation}
The left bottom square of the diagram gives
\begin{equation*}
\xi(I) =\bigl( C^{\tens I} \rTTo^{\sum_{i=1}^If(1)^{\tens(i-1)}\tens\xi(1)\tens g(1)^{\tens(I-i)}} X\hat{T}^{\tens I} \rto{\chi^I} X\tilde{T}(I) \bigr).
\end{equation*}
Thus, $\xi$ can be recovered from $\check\xi$ in unique way.

Let us show that for an arbitrary \(\check\xi\in\und\cv(C,X)^l\) the above formulae give, indeed, an $(f,g)$-coderivation \(\xi:C\to X\tilde{T}\) of degree $l$.
Commutativity of the exterior of the following diagram for all \(\phi:I\to J\in\co_\sk\) is precisely naturality of $\xi$:
\begin{diagram}[w=2em,nobalance]
C^{\tens J} &\rTTo^{\hspace*{-1.6em}\sum_{y=1}^J\tens^J(\sS{^{y-1}}f(1),\xi(1),\sS{^{J-y}}g(1))\hspace*{-2.6em}} &X\hat{T}^{\tens J} &&\rTTo^{\chi^J} &&X\tilde{T}(J)
\\
&&\dTTo<{\tens^{j\in J}\pr_{\phi_*n)_j}} &= &&\ldTTo<{\pr_{\phi_*n}}
\\
&&\tens^{j\in J}X^{\tens(\phi_*n)_j} &\rTTo^{\hspace*{-1em}\lambda^{\sqcup_{j\in J}(\phi_*n)_j\to J}\hspace*{-1em}} &X^{\tens\|n\|} &=
\\
\dTTo~{\tens^{j\in J}\Delta_{\phi^{-1}j}} &&\uTTo<{\hspace*{-2em}\tens^{j\in J}\lambda^{\sqcup_{i\in\phi^{-1}j}n_i\to\phi^{-1}j}} &= &&\luTTo(2,6)>{\pr_n} &\dTTo~{X\tilde{T}(\phi^\op)}
\\
&&\tens^{j\in J}\tens^{i\in\phi^{-1}j}X^{\tens n_i} &&\uTTo<{\lambda^{\sqcup_{i\in I}n_i\to I}} &&
\\
&&\uTTo<{\hspace*{-1em}\tens^{j\in J}\tens^{i\in\phi^{-1}j}\pr_{n_i}} &\rdTTo^{\lambda^\phi}
\\
\tens^{j\in J}C^{\tens\phi^{-1}j} &\rTTo^P &\tens^{j\in J}X\hat{T}^{\tens\phi^{-1}j} &= &\tens^{i\in I}X^{\tens n_i} 
\\
\dTTo<{\lambda^\phi} &= &\dTTo<{\lambda^\phi} &\ruTTo>{\tens^{i\in I}\pr_{n_i}} &=
\\
C^{\tens I} &\rTTo^{\hspace*{-1.6em}\sum_{i=1}^I\tens^I(\sS{^{u-1}}f(1),\xi(1),\sS{^{I-u}}g(1))\hspace*{-2.6em}} &X\hat{T}^{\tens I} &&\rTTo^{\chi^I} &&X\tilde{T}(I)
\end{diagram}
Here
\begin{multline*}
P =\sum_{y=1}^J \tens^J\Bigl( f(1)^{\tens\phi^{-1}1},\dots,f(1)^{\tens\phi^{-1}(y-1)},
\\
\sum_{u\in\phi^{-1}y}\tens^{\phi^{-1}y}\bigl(\sS{^{u-1}}f(1),\xi(1),\sS{^{\phi^{-1}y-u}}g(1)\bigr), g(1)^{\tens\phi^{-1}(y+1)},\dots,g(1)^{\tens\phi^{-1}J} \Bigr).
\end{multline*}
All small cells of this diagram obviously commute.
The hexagon is a sum over $y\in J$ of commuting hexagons, each of which is a tensor product over $j\in J$ of commuting hexagons presented below.
Here $K$ denotes $\phi^{-1}j$ and \(m=(n_i)_{i\in\phi^{-1}j}\in\NN^K\).
If $j<y$, then
\begin{diagram}[w=4em]
C &&&\rTTo_{f(1)} &&&X\hat{T}
\\
&\rdTTo(4,2)^{\Delta_{\|m\|}}
\\
\dTTo<{\Delta_K} &&\tens^{i\in K}C^{\tens n_i} &\rTTo_{\lambda^{\sqcup_{i\in K}n_i\to K}} &C^{\tens\|m\|} &&\dTTo>{\pr_{\|m\|}}
\\
&\ruTTo^{\tens^{i\in K}\Delta_{n_i}} &&\rdTTo_{\tens^{i\in K}\check f^{\tens n_i}} &&\rdTTo^{\check f^{\tens\|m\|}}
\\
C^{\tens K} &\rTTo_{f(1)^{\tens K}} &X\hat{T}^{\tens K} &\rTTo_{\tens^{i\in K}\pr_{n_i}} &\tens^{i\in K}X^{\tens n_i} &\rTTo_{\lambda^{\sqcup_{i\in K}n_i\to K}} &X^{\tens\|m\|}
\end{diagram}
commutes.
If $j>y$, then similar hexagon occurs with $g$ in place of $f$.
If $j=y$, then the commutative hexagon is
\begin{diagram}[w=2em]
C &&&\rTTo_{\xi(1)} &&&X\hat{T}
\\
&\rdTTo^{\Delta_{\|m\|}} &&&&=
\\
&&C^{\tens\|m\|} &&&&\dTTo<{\pr_{\|m\|}}
\\
\dTTo>{\Delta_K} &&\uTTo<{\lambda^{\sqcup_{i\in K}n_i\to K}} &\rdTTo(4,4)>{\sum_{l=1}^{\|m\|}\tens^I(\sS{^{l-1}}{\check f},\check\xi,\sS{^{\|m\|-l}}{\check g})}
\\
&= &\tens^{i\in K}C^{\tens n_i}
\\
&\ruTTo^{\tens^{i\in K}\Delta_{n_i}} &&\rdTTo_Q &= &
\\
C^{\tens K} &\rTTo_{\sum_{u\in K}\tens^I(\sS{^{u-1}}f(1),\xi(1),\sS{^{K-u}}g(1))\hspace*{-1em}} &X\hat{T}^{\tens K} &\rTTo_{\hspace*{-.8em}\tens^{i\in K}\pr_{n_i}\hspace*{-.8em}} &\tens^{i\in K}X^{\tens n_i} &\rTTo_{\lambda^{\sqcup_{i\in K}n_i\to K}} &X^{\tens\|m\|}
\end{diagram}
where
\[ Q =\sum_{u\in K} \tens^K\Bigl( \check f^{\tens n_1},\dots,\check f^{\tens n_{u-1}},\sum_{l=1}^{n_u}\tens^{n_u}\bigl(\sS{^{l-1}}{\check f},\check\xi,\sS{^{n_u-l}}{\check g}\bigr),\check g^{\tens n_{u+1}},\dots,\check g^{\tens K} \Bigr).
\]
Here three quadrilaterals obviously commute.
The remaining one is a sum over $u\in K$ of tensor products over $i\in K$ of commuting quadrilaterals, see equations \eqref{eq-CXTX-CC(I)CX} and \eqref{eq-CXTX-CCX}.
Thus $\xi$ is natural.

Let us prove property~\eqref{dia-CCGG-(fg)-coderivation}.
It takes the form
\begin{diagram}[nobalance]
\tens^{i\in I}C^{\tens n_i} &\rTTo^R &\tens^{i\in I}X\hat{T}^{\tens n_i} &\rTTo^{\tens^{i\in I}\chi^{n_i}} &\tens^{i\in I}X\tilde{T}(n_i)
\\
\dTTo<{\lambda^{\sqcup_{i\in I}n_i\to I}} &&\dTTo<{\lambda^{\sqcup_{i\in I}n_i\to I}} &&\dTTo>{\chi^I}
\\
C^{\tens\|n\|} &\rTTo^{\sum_{k=1}^{\|n\|}\tens^{\|n\|}(\sS{^{k-1}}f(1),\xi(1),\sS{^{\|n\|-k}}g(1))} &X\hat{T}^{\tens\|n\|} &\rTTo^{\chi^{\|n\|}} &X\tilde{T}(\|n\|)
\end{diagram}
where
\begin{equation*}
R =\sum_{u\in I} \tens^I\Bigl( \bigl(f(1)^{\tens n_i}\bigr)_{i=1}^{u-1},\sum_{k=1}^{n_u}\tens^{n_u}\bigl(\sS{^{k-1}}f(1),\xi(1),\sS{^{n_u-k}}g(1)\bigr), \bigl(g(1)^{\tens n_i}\bigr)_{i=u+1}^I \Bigr).
\end{equation*}
Commutativity of both squares is obvious.
Thus, $\xi$ is an $(f,g)$-coderivation.
\QED\end{proof}

\begin{proposition}\label{pro-fgxi-XT-Y}
Let morphisms \(\check f=(f_k)\), \(\check g=(g_k):XT\to Y\in\cv\), \(\check\xi=(\xi_k):XT\to Y\in\bar\cv^a\) have finite support, \(\deg\check f=\deg\check g=0\), \(\deg\check\xi=a\).
Namely, there exists \(K\in\NN\) such that $f_k=g_k=\xi_k=0$ for all $k>K$.
Associate morphisms of homotopy coalgebras \(f,g:X\tilde{T}\to Y\tilde{T}\) with $\check f$, $\check g$ as in \propref{pro-f-XT-Y-defines-f}.
For each $I\in\co_\sk$ define degree $a$ morphisms \(\xi(I):X\tilde{T}(I)\to Y\tilde{T}(I)\) by the expressions, $n\in\NN^I$,
\begin{multline}
\bigl\langle X\tilde{T}(I) \rTTo^{\xi(I)} Y\tilde{T}(I) \rto{\pr_n} Y^{\tens\|n\|} \bigr\rangle =\sum_{k\in[K]^{\|n\|}} \bigl\langle X\tilde{T}(I) \rTTo^{\pr_{(\sqcup_{i\in I}n_i\to I)_*k}} X^{\tens\|k\|}
\\
\rTTo^{(\lambda^{\sqcup_{p\in\|n\|}k_p\to\|n\|})^{-1}} \tens^{p\in\|n\|}X^{\tens k_p} \rTTo^{\sum_{u=1}^{\|n\|}\tens^{\|n\|}((f_{k_p})_{p=1}^{u-1},\xi_{k_u},(g_{k_p})_{p=u+1}^{\|n\|})} Y^{\tens\|n\|} \bigr\rangle.
\label{eq-XT(I)-xi(I)-YT(I)-pr-Y}
\end{multline}
Then $\xi$ is an $(f,g)$-coderivation.
\end{proposition}

\begin{proof}
For an arbitrary map \(\phi:I\to J\in\co_\sk\) we have to prove naturality of $\xi$, expressed by the exterior of the following diagram.
\begin{diagram}
X\tilde{T}(J) &&&\rTTo^{\xi(J)} &&&Y\tilde{T}(J)
\\
&\rdTTo>{\pr_{(\sqcup_{j\in J}(\phi_*n)_j\to J)_*k}} &&&&\ldTTo<{\pr_{\phi_*n}}
\\
\dTTo>{X\tilde{T}(\phi^\op)} &&\sum_{k\in[K]^{\|n\|}} \Bigl\langle X^{\tens\|k\|} &\rto{(\lambda^{\sqcup_{p\in\|n\|}k_p\to\|n\|})^{-1}} \tens^{p\in\|n\|}X^{\tens k_p} \rto{Q} &Y^{\tens\|n\|} \Bigr\rangle &&\dTTo<{Y\tilde{T}(\phi^\op)}
\\
&\ruTTo>{\pr_{(\sqcup_{i\in I}n_i\to I)_*k}} &&&&\luTTo<{\pr_n}
\\
X\tilde{T}(I) &&&\rTTo^{\xi(I)} &&&Y\tilde{T}(I)
\end{diagram}
After postcomposing with \(\pr_n:Y\tilde{T}(I)\to Y^{\tens\|n\|}\), $n\in\NN^I$, the diagram becomes commutative.
Here
\[ Q =\sum_{u=1}^{\|n\|}\tens^{\|n\|}((f_{k_p})_{p=1}^{u-1},\xi_{k_u},(g_{k_p})_{p=u+1}^{\|n\|}).
\]

Equation~\eqref{dia-CCGG-(fg)-coderivation} for $\xi$ means commutativity of the exterior of the following diagram.
Its postcomposition with \(\pr_m:Y\tilde{T}(N)\to Y^{\tens\|m\|}\), \(m=(m_n)\in\NN^N\), \(N=\sqcup_{i\in I}N_i\), partitioned into commutative cells, involves two sums of braced paths with equal number of summands, indexed by \((k^1,\dots,k^I)\in[K]^{\|m^1\|}\times\dots\times[K]^{\|m^I\|}\) and \(k\in[K]^{\|m\|}\) corresponding to each other, \(m=m^1\oplus\dots\oplus m^I\in\NN^{N_1}\oplus\dots\oplus\NN^{N_I}\):
\begin{diagram}[nobalance,h=2.6em]
\tens^{i\in I}X\tilde{T}(N_i) &&&\rTTo^{\chi^I} &&&X\tilde{T}(N)
\\
&\rdTTo>{\overbrace{\sss\tens^{i\in I}\pr_{(\sqcup_{n_i\in N_i}m^i_{n_i}\to N_i)_*k^i}}} &&&&\ldTTo<{\overbrace{\sss\pr_{(\sqcup_{n\in N}m_n\to N)_*k}}}
\\
&\hspace*{-3em}\sum_{k^1\in[K]^{\|m^1\|}} \hspace*{-.5em}\dots\hspace*{-.5em}\sum_{k^I\in[K]^{\|m^I\|}}\hspace*{-1em} &\tens^{i\in I}X^{\tens\|k^i\|} &\rTTo^{\lambda^{\sqcup_{i\in I}\|k^i\|\to I}\hspace*{-1.2em}} &X^{\tens\|k\|} &\hspace*{-1em}\sum_{k\in[K]^{\|m\|}}\hspace*{-1em}
\\
&&\dTTo<{\hspace*{-3em}\tens^{i\in I}(\lambda^{\sqcup_{p\in\|m^i\|}k^i_p\to\|m^i\|})^{-1}} &&\dTTo>{(\lambda^{\sqcup_{p\in\|m\|}k_p\to\|m\|})^{-1}\hspace*{-3em}}
\\
&&\hspace*{-1.7em}\tens^{i\in I}\tens^{p\in\|m^i\|}X^{\tens k^i_p} &\rTTo^{\lambda^{\sqcup_{i\in I}\|m^i\|\to I}\hspace*{-1.2em}} &\tens^{p\in\|m\|}X^{\tens k_p} &&\dTTo<{\xi(N)}
\\
&&\dTTo<{\tens^{i\in I}P_v^i(k^i)} &&\dTTo~{\hspace*{-3em}\sum_{u=1}^{\|m\|}\tens^{\|m\|}((f_{k_p})_{p=1}^{u-1},\xi_{k_u},(g_{k_p})_{p=u+1}^{\|m\|})\hspace*{-3em}}
\\
\dTTo>{\sum_{v=1}^I\tens^I(\sS{^{v-1}}f,\xi,\sS{^{I-v}}g)\hspace*{-1em}} &&\underbrace{\tens^{i\in I}Y^{\tens\|m^i\|}} &\rTTo^{\lambda^{\sqcup_{i\in I}\|m^i\|\to I}\hspace*{-1.2em}} &\underbrace{Y^{\tens\|m\|}}
\\
&\ruTTo>{\tens^{i\in I}\pr_{m^i}} &&&&\luTTo<{\pr_m}
\\
\tens^{i\in I}Y\tilde{T}(N_i) &&&\rTTo^{\chi^I} &&&Y\tilde{T}(N)
\end{diagram}
Here
\[ P_v^i(k^i)=
\begin{cases}
\tens^{p\in\|m^i\|}f_{k^i_p} &\text{ for } i<v,
\\
\sum_{u=1}^{\|m^v\|} \tens^{\|m^v\|} ((f_{k_p^i})_{p=1}^{u-1},\xi_{k_u^i},(g_{k_p^i})_{p=u+1}^{\|m^v\|}) &\text{ for } i=v,
\\
\tens^{p\in\|m^i\|}g_{k^i_p} &\text{ for } i>v.
\end{cases}
\]
Thus $\xi$ is an $(f,g)$-coderivation.
\QED\end{proof}

\begin{remark}\label{rem-xi(1)-coderivation}
For $I=1$ the morphism $\xi(1)$ constructed in \propref{pro-fgxi-XT-Y} is found from equations 
\begin{multline*}
\xi(1)\pr_n =\sum_{k\in[K]^n} \Bigl\langle X\hat{T} \rTTo^{\pr_{\|k\|}} X^{\tens\|k\|} \rTTo^{(\lambda^{\sqcup_{p\in n}k_p\to n})^{-1}} \tens^{p\in n}X^{\tens k_p} 
\ifx\chooseClass1
\\
\fi
\rTTo^{\sum_{u=1}^n\tens^n((f_{k_p})_{p=1}^{u-1},\xi_{k_u},(g_{k_p})_{p=u+1}^n)} Y^{\tens n} \Bigr\rangle
\end{multline*}
for $n\in\NN$.
In the particular case \(f=g=\id\) we have \(f_k=g_k=\delta_{k,1}\) and
\[ \inj_m\xi(1)\pr_n =\sum_{u=1}^n 1^{\tens(u-1)}\tens\xi_{m+1-n}\tens1^{\tens(n-u)}: X^{\tens m} \to X^{\tens n}.
\]
\end{remark}

\begin{example}\label{exa-(fg)-coderivations-exi-xie}
Given morphisms \(f,g:XT\to Y\in\cv\) and \(\check\xi:XT\to Y\in\und\cv^a\) such that $f_0=g_0=0$ we can produce an $(f,g)$-coderivation \(\xi:XT\to YT\) by the formula
\begin{equation*}
\inj_m\xi\pr_n =\sum_{k_1+\dots+k_l+r+p_1+\dots+p_q=m}^{l+1+q=n} f_{k_1}\tdt f_{k_l}\tens\xi_r\tens g_{p_1}\tdt g_{p_q}: X^{\tens m} \to Y^{\tens n}.
\end{equation*}
Assume in addition that \(f,g,\check\xi\) have finite support.
Then both coderivations in
\begin{diagram}
XT &\rTTo^{e_X} &X\tilde{T}
\\
\dTTo<\xi &= &\dTTo>\xi
\\
YT &\rTTo^{e_Y} &Y\tilde{T}
\end{diagram}
are well-defined.
We claim that the square commutes, that is, the \((e_Xf,e_Xg)\)-coderivation $e_X\xi$ equals the \((fe_Y,ge_Y)\)-coderivation $\xi e_Y$.
Equalities \(e_Xf=fe_Y\), \(e_Xg=ge_Y\) were verified in \exaref{exa-XT-eX-X-tilde-T}.
In fact, due to \propref{pro-(fg)coder(CXT)-V(CX)} it suffices to verify that
\[ \bigl(XT \rTTo^\xi YT \rTTo^{e_Y(1)} Y\hat{T} \rTTo^{\pr_1} Y\bigr) =\bigl(XT \rTTo^{e_X(1)} X\hat{T} \rTTo^{\xi(1)} Y\hat{T} \rTTo^{\pr_1} Y\bigr).
\]
One easily finds that the both sides are equal to \(\check\xi:XT\to Y\).

The composition $\xi e_Y$ coincides with the morphism \(\xi:XT\to Y\tilde{T}\) obtained from \(\check\xi:XT\to Y\) via \propref{pro-fgxi-XT-Y}.
In fact,
\[ \sum_{k\in\NN} \bigl\langle XT \rTTo^{\pr_k} X^{\tens k} \rTTo^{\xi_k} Y \bigr\rangle =\check\xi.
\]
\end{example}

\begin{proposition}\label{pro-(fg)-coderivation-xi-obtained-from-check-xi}
Let morphisms \(\check f=(f_k)\), \(\check g=(g_k):XT\to Y\in\cv\) be supported on $[K]$.
Associate with them morphisms of homotopy coalgebras \(f,g:X\tilde{T}\to Y\tilde{T}\) as in \propref{pro-f-XT-Y-defines-f}.
Let an $(f,g)$-coderivation \(\xi:X\tilde{T}\to Y\tilde{T}\) be such that for all $n\in\NN^I$ the degree $a$ morphism \(X\tilde{T}(I)\rTTo^{\xi(I)} Y\tilde{T}(I)\rto{\pr_n} Y^{\tens\|n\|}\) be supported on $[Kn]$.
Then there is a unique degree $a$ morphism \(\check\xi=(\xi_k):XT\to Y\in\und\cv^a\) supported on $[K]$ such that \(\xi\) is obtained from $\check\xi$ as in \eqref{eq-XT(I)-xi(I)-YT(I)-pr-Y}.
\end{proposition}

\begin{proof}
Clearly, \eqref{eq-XT(I)-xi(I)-YT(I)-pr-Y} is supported on $[Kn]$.
Since \(\check\xi=\bigl\langle XT\rTTo^{e(1)} X\tilde{T}(1)\rTTo^{\xi(1)}\) \(Y\tilde{T}(1)\rto{\pr_1} Y\bigr\rangle\) uniqueness of $\check\xi$ is obvious.
Let us prove that $\xi$ is restored from this $\check\xi$ via equation~\eqref{eq-XT(I)-xi(I)-YT(I)-pr-Y}.
Similarly to \eqref{dia-C-xi-XT(1)} there is a commutative diagram
\begin{diagram}[LaTeXeqno]
X\hat{T} &\rTTo^{\xi(1)} &Y\hat{T}
\\
\dTTo>{X\tilde{T}((\con:I\to\mb1)^\op)} &&\dTTo<{Y\tilde{T}(\con^\op)} &\rdTTo^{\pr_I} &
\\
X\tilde{T}(I) &\rTTo^{\xi(I)} &Y\tilde{T}(I) &\rTTo^{\pr_{(1,\dots,1)}} &Y^{\tens I}
\\
\uTTo>{\chi^I} &&\uTTo<{\chi^I} &\ruTTo_{\pr_1^{\tens I}} &
\\
X\hat{T}^{\tens I} &\rTTo^{\hspace*{-1.6em}\sum_{i=1}^If(1)^{\tens(i-1)}\tens\xi(1)\tens g(1)^{\tens(I-i)}\hspace*{-.6em}} &Y\hat{T}^{\tens I}
\label{dia-XT(I)-xi(I)-YT(I)}
\end{diagram}
In particular, both above pentagons commute.
All three compositions \(\xi(-)\pr_-\) in this diagram can be written in form~\eqref{eq-MM-phi-N} by hypothesis.
We extend the notation `-' to morphisms of degree $a$.
The above pentagons take the form
\begin{diagram}[h=2.4em]
X\hat{T} &\rEpi &\prod_{m\in[KI]}X^{\tens m}
\\
\dTTo<{X\tilde{T}((\con:I\to\mb1)^\op)} &&\dTTo<h &\rdTTo^{\text`\xi(1)\pr_I\text'} &
\\
X\tilde{T}(I) &\rEpi &\prod_{n\in[K]^I}X^{\tens\|n\|} &\rTTo^{\text`\xi(I)\pr_{(1,\dots,1)}\text'\qquad} &Y^{\tens I}
\\
\uTTo<{\chi^I} &&\uTTo<g &\ruTTo_{\sum_{i=1}^I\check f^{\tens(i-1)}\tens\check\xi\tens\check g^{\tens(I-i)}} &
\\
X\hat{T}^{\tens I} &\rEpi &\bigl(\prod_{k\in[K]}X^{\tens k}\bigr)^{\tens I}
\end{diagram}
The both triples of neighbour arrows on the left close up to a commutative square by morphisms $h$ and $g$ from \eqref{eq-hprn-gprn}.
Since the top and the bottom horizontal arrows are epimorphisms, both triangles on the right commute.
Since $g$ is an isomorphism, \(\text`\xi(1)\pr_I\text'=h\cdot g^{-1}\cdot\sum_{i=1}^I\check f^{\tens(i-1)}\tens\check\xi\tens\check g^{\tens(I-i)}\) is determined in a unique way as well as
\begin{multline}
\xi(1)\pr_I =\Bigl\langle X\hat{T} \to \!\prod_{m\in[KI]}\!X^{\tens m} \rto h \!\prod_{n\in[K]^I}\!X^{\tens\|n\|} \rto[\sim]{g^{-1}} \bigl(\prod_{k\in[K]}\!X^{\tens k}\bigr)^{\tens I}
\ifx\chooseClass1
\\ \hfill
\fi
\rTTo^{\sum_{i=1}^I\check f^{\tens(i-1)}\tens\check\xi\tens\check g^{\tens(I-i)}} Y^{\tens I} \Bigr\rangle
\\
=\Bigl\langle X\hat{T} \to \!\prod_{m\in[KI]}\!X^{\tens m} \rTTo^{\Delta^{\restr}} \bigl(\prod_{k\in[K]}\!X^{\tens k}\bigr)^{\tens I} \rTTo^{\sum_{i=1}^I\check f^{\tens(i-1)}\tens\check\xi\tens\check g^{\tens(I-i)}} Y^{\tens I} \Bigr\rangle,
\label{eq-xi(1)prI}
\end{multline}
where the restricted cut comultiplication is given by \eqref{eq-Delta-restr}.
All possible $\xi(1)$ coincide hence are given by \eqref{eq-XT(I)-xi(I)-YT(I)-pr-Y} for $I=1$.

The lower commutative square in diagram~\eqref{dia-XT(I)-xi(I)-YT(I)} can be written as the pentagon commutative for arbitrary $n\in\NN^I$ in
\begin{diagram}
X\tilde{T}(I) &\rTTo^{\xi(I)} &Y\tilde{T}(I) &\rTTo^{\pr_n} &Y^{\tens\|n\|}
\\
\uTTo<{\chi^I} &&&\ruTTo^{\chi^I\pr_n} &\uTTo>\lambda
\\
X\hat{T}^{\tens I} &\rTTo_{\sum_{i=1}^If(1)^{\tens(i-1)}\tens\xi(1)\tens g(1)^{\tens(I-i)}} &Y\hat{T}^{\tens I} &\rTTo_{\tens^{i\in I}\pr_{n_i}} &\otimes^{i\in I} Y^{\tens n_i}
\end{diagram}
By hypothesis the top and the bottom pairs of horizontal arrows can be presented in form~\eqref{eq-MM-phi-N}.
Thus the exterior of the diagram
\begin{diagram}[h=2.4em,w=2em]
X\tilde{T}(I) &\rEpi &\prod_{m\in[Kn]}X^{\tens\|m\|} &\rTTo_{\text`\xi(I)\pr_n\text'} &Y^{\tens\|n\|}
\\
\uTTo<{\chi^I} &&\uTTo<j &&\uTTo>\lambda
\\
X\hat{T}^{\tens I} &\rEpi &\bigotimes_{i\in I} \prod_{k_i\in[Kn_i]}X^{\tens k_i} &\rTTo_{\hspace*{-1.3em}\sum_{u=1}^I\tens^{i\in I}((\text`\hat{f}(1)\pr_{n_i}\text')_{i=1}^{u-1},\text`\xi(1)\pr_{n_u}\text',(\text`\hat{g}(1)\pr_{n_i}\text')_{i=u+1}^I)\hspace*{-.3em}} &\bigotimes_{i\in I} Y^{\tens n_i}
\end{diagram}
commutes.
The isomorphism $j$ closing up the three arrows on the left to a commutative square is given by \eqref{eq-jprm-XXX}.
Since the bottom left horizontal arrow is an epimorphism, the right square commutes as well.
Hence there is a unique possible value for \(\text`\xi(I)\pr_n\text'=j^{-1}[\sum_{u=1}^I\tens^{i\in I}((\text`\hat{f}(1)\pr_{n_i}\text')_{i=1}^{u-1},\text`\xi(1)\pr_{n_u}\text',(\text`\hat{g}(1)\pr_{n_i}\text')_{i=u+1}^I)]\lambda\) and for $\xi(I)$ itself.
Clearly, $\xi(I)$ is reconstructed from $\check\xi$ via \eqref{eq-XT(I)-xi(I)-YT(I)-pr-Y}.
\QED\end{proof}

\subsection{Components of coderivations for homotopy coalgebras}
\label{Components-coderivations-lax-coalgebras}
Suppose that \(f,g:X\tilde{T}\to Y\tilde{T}\in\hCoalg\), \(h:W\tilde{T}\to X\tilde{T}\in\hCoalg\) and an $(f,g)$-coderivation \(\xi:X\tilde{T}\to Y\tilde{T}\) are obtained as in Propositions \ref{pro-f-XT-Y-defines-f} and \ref{pro-fgxi-XT-Y} from maps \(\check f,\check g:XT\to Y\in\cv\), \(\check h:WT\to X\in\cv\) and \(\check\xi:XT\to Y\in\und\cv^a\).
Assume that $\check f$, $\check g$ and $\check\xi$ are supported on $[L]$ and $\check h$ is supported on $[K]$.
Then $h\xi$ is an $(hf,hg)$-coderivation.
For all $n\in\NN^I$ the degree $a$ morphisms \(h(I)\xi(I)\pr_n\) are supported on \([KLn]\).
By \propref{pro-(fg)-coderivation-xi-obtained-from-check-xi} the $(hf,hg)$-coderivation \(\theta=h\xi:W\tilde{T}\to Y\tilde{T}\) can be obtained from a map $\check\theta:WT\to Y$ of degree $a$ supported on $[KL]$.
Namely,
\begin{align*}
\check\theta &= \bigl(WT \rTTo^{e(1)} W\hat{T} \rTTo^{h(1)} X\hat{T} \rTTo^{\xi(1)} Y\hat{T} \rTTo^{\pr_1} Y\bigr)
\\
&= \sum_{k_1,\dots,k_l\le K}^{l\le L} \Bigl(WT \rTTo^{\pr_{k_1+\dots+k_l}} W^{\tens(k_1+\dots+k_l)} \rTTo^{h_{k_1}\tdt h_{k_l}} X^{\tens l} \rTTo^{\xi_l} Y\Bigr),
\end{align*}
using \eqref{eq-f(1)prI-XX-Delta-rest-XY}.
Its components are
\[ \theta_m =\sum_{k_1+\dots+k_l=m}^{k_i\le K,\;l\le L} \Bigl(W^{\tens m} \rTTo^{\lambda^{-1}}_\sim W^{\tens k_1}\tdt W^{\tens k_l} \rTTo^{h_{k_1}\tdt h_{k_l}} X^{\tens l} \rTTo^{\xi_l} Y\Bigr).
\]

Suppose that \(f,g:X\tilde{T}\to Y\tilde{T}\in\hCoalg\), \(h:Y\tilde{T}\to Z\tilde{T}\in\hCoalg\) and an $(f,g)$-coderivation \(\xi:X\tilde{T}\to Y\tilde{T}\) are obtained as in Propositions \ref{pro-f-XT-Y-defines-f} and \ref{pro-fgxi-XT-Y} from maps \(\check f,\check g:XT\to Y\in\cv\), \(\check h:YT\to Z\in\cv\) and \(\check\xi:XT\to Y\in\und\cv^a\).
Assume that $\check f$, $\check g$ and $\check\xi$ are supported on $[K]$ and $\check h$ is supported on $[L]$.
Then $\xi h$ is an $(fh,gh)$-coderivation.
For all $n\in\NN^I$ the degree $a$ morphisms \(\xi(I)h(I)\pr_n\) are supported on \([KLn]\).
By \propref{pro-(fg)-coderivation-xi-obtained-from-check-xi} the $(fh,gh)$-coderivation \(\theta=\xi h:X\tilde{T}\to Z\tilde{T}\) can be obtained from a map $\check\theta:XT\to Z$ of degree $a$ supported on $[KL]$.
Namely,
\begin{multline*}
\check\theta = \bigl(XT \rTTo^{e(1)} X\hat{T} \rTTo^{\xi(1)} Y\hat{T} \rTTo^{h(1)} Z\hat{T} \rTTo^{\pr_1} Z\bigr) =\sum_{\substack{k_1,\dots,k_a,n,l_1,\dots,l_c\le K\\ m=k_1+\dots+k_a+n+l_1+\dots+l_c}}^{a+1+c\le L} \Bigl(XT
\\
\rTTo^{\pr_m} X^{\tens m} \rTTo^{f_{k_1}\tdt f_{k_a}\tens\xi_n\tens g_{l_1}\tdt g_{l_c}} Y^{\tens(a+1+c)} \rTTo^{g_{a+1+c}} Z\Bigr),
\end{multline*}
using \eqref{eq-xi(1)prI}.
Its components are
\[ \theta_m =\hspace*{-1em}\sum_{\substack{k_1,\dots,k_a,n,l_1,\dots,l_c\le K\\ k_1+\dots+k_a+n+l_1+\dots+l_c=m}}^{a+1+c\le L}\hspace*{-1em} \Bigl(X^{\tens m} \rTTo^{f_{k_1}\tdt f_{k_a}\tens\xi_n\tens g_{l_1}\tdt g_{l_c}} Y^{\tens(a+1+c)} \rTTo^{g_{a+1+c}} Z\Bigr).
\]

Let a coderivation \(\xi:X\tilde{T}\to X\tilde{T}\) of odd degree $a$ come as in \propref{pro-fgxi-XT-Y} from a map \(\check\xi:XT\to X\in\und\cv^a\) supported on $[K]$, $K\ge1$.
Then \(\xi^2:X\tilde{T}\to X\tilde{T}\) is a coderivation of degree $2a$.
For all $n\in\NN^I$ the degree $2a$ morphisms \(\xi(I)^2\pr_n\) are supported on \([K^2n]\).
By \propref{pro-(fg)-coderivation-xi-obtained-from-check-xi} the coderivation \(\theta=\xi^2:X\tilde{T}\to X\tilde{T}\) can be obtained from a map $\check\theta:XT\to X$ of degree $2a$ supported on $[K^2]$.
Namely,
\begin{align*}
\check\theta &= \bigl(XT \rTTo^{e(1)} X\hat{T} \rTTo^{\xi(1)} X\hat{T} \rTTo^{\xi(1)} X\hat{T} \rTTo^{\pr_1} X\bigr)
\\
&= \sum_{n\le K,\;a+1+c\le K} \Bigl(XT \rTTo^{\pr_{a+n+c}} X^{\tens(a+n+c)} \rTTo^{1^{\tens a}\tens\xi_n\tens1^{\tens c}} X^{\tens(a+1+c)} \rTTo^{\xi_{a+1+c}} X\Bigr),
\end{align*}
using \eqref{eq-xi(1)prI}.
Its components are
\[ \theta_m =\sum_{n\le K,\;a+1+c\le K}^{a+n+c=m} \Bigl(X^{\tens m} \rTTo^{1^{\tens a}\tens\xi_n\tens1^{\tens c}} X^{\tens(a+1+c)} \rTTo^{\xi_{a+1+c}} X\Bigr).
\]
It is clear from this formula that, actually, $\check\theta$ is supported on $[2K-1]$.

\section{Curved algebras and coalgebras}
Similarly to Positselski \cite{0905.2621} we define curved (co)algebras and their morphisms.
An extra bit of structure, splitting of the (co)unit, is added in order to provide explicitly the bar- and cobar-constructions in the next section.
Also we define curved homotopy coalgebras and their morphisms.
Instead of combining curved coalgebras and curved homotopy coalgebras in a single category we describe a bimodule over these two categories.

\subsection{Curved algebras}
As in Positselski work \cite[Section~3.1]{0905.2621}  we give

\begin{definition}
A \emph{curved algebra} \((A,m_2,m_1\), \(m_0,\eta)\) in $\cv$ is a unital associative algebra \((A,m_2,\eta)\) with a degree~1 derivation $m_1:A\to A$ and a degree~2 \emph{curvature} element $m_0\in A^2$ such that
\[ m_1^2 =(m_0\tens1 -1\tens m_0)m_2, \qquad m_0m_1 =0.
\]
\end{definition}

\begin{proposition}\label{pro-curved-algebra-A-algebra-hatA}
Let \(A=(A,m_2,\eta)\) be an (associative unital) algebra in $\cv$ equipped with a degree~1 morphism $m_1:A\to A$ and a degree~2 element $m_0\in A^2$.
Consider the degree~1 morphism
\[ d =\bigl( \1 \rto{\sigma^{-1}} \1[-1] \simeq \1[-1]\tens\1 \rTTo^{\id\tens\eta} \1[-1]\tens A \bigr).
\]
Existence of an associative unital algebra structure on \(\ddot{A}=A\oplus \1[-1]\tens A\simeq A\oplus A[-1]\) such that
\begin{enumerate}
\renewcommand{\labelenumi}{(\arabic{enumi})}
\item \(\inj_1:A\to\ddot{A}\) is a unital algebra morphism;

\item multiplication \((\1[-1]\tens A)\tens A\to\ddot{A}\) coincides with the obvious right $A$\n-module structure of \(\1[-1]\tens A\);

\item \((1\tens d)m_2=(d\tens1)m_2+m_1:A\to\ddot{A}\);
\label{eq-(1d)m2-(d1)m2-m1}

\item \((d\tens d)m_2=-m_0\);
\label{eq-d2-m0}
\end{enumerate}
is equivalent to \((A,m_2,m_1,m_0,\eta)\) being a curved algebra.
\end{proposition}

Strange choice of sign in \eqref{eq-d2-m0} agrees with the \ainf-algebra approach to curved algebras \cite{Lyu-curved-coalgebras}.

\begin{proof}
Let us show that properties (1)--(4) imply that $A$ is a curved algebra.
By (3) $m_1$ extends to an inner derivation of $\ddot{A}$.
By (1) $m_1:A\to A$ is a derivation.
Its square is the commutator with $d^2=-m_0$ by (4).
Property (3) implies that \(m_0m_1=-(d^2)m_1=0\).

When \((A,m_2,m_1,m_0,\eta)\) is a curved algebra properties (1)--(4) fix the multiplication in $\ddot{A}$.
For instance,
\begin{multline*}
\bigl( A\tens(\1[-1]\tens A) \rTTo^{\inj_1\tens\inj_2} \ddot{A}\tens\ddot{A} \rto m \ddot{A} \bigr)
\\
\hskip\multlinegap =\bigl( A\tens\1[-1]\tens A \rTTo^{(12)} \1[-1]\tens A\tens A \rTTo^{1\tens m_2} \1[-1]\tens A \rTTo^{\inj_2} \ddot{A} \bigr) \hfill
\\
+\bigl( A\tens\1[-1]\tens A \rTTo^{1\tens\sigma\tens1} A\tens\1\tens A \rTTo^{m_1\tens1} A\tens A \rto{m_2} A \rTTo^{\inj_1} \ddot{A} \bigr).
\end{multline*}
As a corollary
\begin{multline*}
\inj' =\bigl( A\tens\1[-1] \rTTo^{\inj_1\tens\sigma d\inj_2} \ddot{A}\tens\ddot{A} \rto m \ddot{A} \bigr)
\\
=\bigl( A\tens\1[-1] \rTTo^{(12)} \1[-1]\tens A \rto{\inj_2} \ddot{A} \bigr) +\bigl( A\tens\1[-1] \rTTo^{1\tens\sigma} A\tens\1 =A \rto{m_1} A \rto{\inj_1} \ddot{A} \bigr).
\end{multline*}
Therefore, $\ddot{A}$ can be presented as a direct sum using injections $\inj'$ and $\inj_1$.
Notice also that
\[ \inj_2 =\bigl( \1[-1]\tens A \rTTo^{\sigma d\inj_2\tens\inj_1} \ddot{A}\tens\ddot{A} \rto m \ddot{A} \bigr).
\]
It's easy to compute
\begin{multline*}
\bigl( (A\tens\1[-1])\tens(\1[-1]\tens A) \rTTo^{\inj'\tens\inj_2} \ddot{A}\tens\ddot{A} \rto m \ddot{A} \bigr)
\\
=\bigl( A\tens\1[-1]\tens\1[-1]\tens A \rTTo^{1\tens\sigma\tens\sigma\tens1} A\tens\1\tens A \rTTo^{1\tens m_0\tens1} A\tens A\tens A \rto{m_2^{(3)}} A \rto{\inj_1} \ddot{A} \bigr).
\end{multline*}

Associativity of the multiplication in $\ddot{A}$ is verified case by case.
\QED\end{proof}

\begin{definition}[similar to Positselski \cite{0905.2621} Section~3.1]
A \emph{morphism of curved algebras} \(\sff:A\to B\) is a pair \((\sff_1,\sff_0)\) consisting of a homomorphism of unital algebras \(\sff_1:A\to B\) and a degree~1 \emph{change-of-connection} element \(\sff_0\in B^1\) such that
\begin{equation}
\sff_1m_1 +\sff_1(1\tens\sff_0-\sff_0\tens1)m_2 =m_1\sff_1, \qquad m_0 -\sff_0m_1 -(\sff_0\tens\sff_0)m_2 =m_0\sff_1.
\label{eq-f1m1-f1ff1m-m1f1}
\end{equation}
\end{definition}

\begin{proposition}\label{pro-algebra-homomorphism-morphism-curved-algebras}
Let $A$, $B$ be curved algebras in $\cv$, let \(\sff_1:A\to B\) be a homomorphism of unital algebras and let \(\sff_0\in B^1\).
Existence of a unital algebra homomorphism \(\ddot\sff:\ddot{A}\to\ddot{B}\) such that
\begin{enumerate}
\renewcommand{\labelenumi}{(\arabic{enumi})}
\item \(\ddot\sff|_A=\sff_1\);

\item \(d\ddot\sff=\sff_0+d\);
\end{enumerate}
is equivalent to \((\sff_1,\sff_0):A\to B\) being a morphism of curved algebras.
\end{proposition}

\begin{proof}
Assume that $\ddot\sff$ is an algebra morphism.
Applying $\ddot\sff$ to identity~\eqref{eq-(1d)m2-(d1)m2-m1} of \propref{pro-curved-algebra-A-algebra-hatA} in $\ddot{A}$ we get
\[ [\sff_1\tens(\sff_0+d)]m_2 =[(\sff_0+d)\tens\sff_1]m_2+m_1\sff_1: A\to\ddot{B}.
\]
Using \eqref{eq-(1d)m2-(d1)m2-m1} of \propref{pro-curved-algebra-A-algebra-hatA} in $\ddot{B}$ we deduce the first of equations~\eqref{eq-f1m1-f1ff1m-m1f1}.

Applying $\ddot\sff$ to identity~\eqref{eq-d2-m0} of \propref{pro-curved-algebra-A-algebra-hatA} in $\ddot{A}$ we get
\[ [(\sff_0+d)\tens(\sff_0+d)]m_2 =-m_0\sff_1 \in \ddot{B}^2.
\]
Using \eqref{eq-d2-m0} of \propref{pro-curved-algebra-A-algebra-hatA} in $\ddot{B}$ and the identity \((\sff_0\tens d+d\tens\sff_0)m_2=\sff_0m_1\) we deduce the second of equations~\eqref{eq-f1m1-f1ff1m-m1f1}.

Assume that $(\sff_1,\sff_0)$ is a curved algebra morphism.
The map $\ddot\sff$ is determined by \(\ddot\sff|_A=\sff_1\),
\[ \inj_2\ddot\sff =\bigl( \1[-1]\tens A \rTTo^{\sigma(\sff_0\inj_1+d\inj_2)\tens\sff_1\inj_1} \ddot{B}\tens\ddot{B} \rto m \ddot{B} \bigr).
\]
Case by case one checks that $\ddot\sff$ is an algebra morphism.
\QED\end{proof}

The category whose objects are curved algebras and whose morphisms are morphisms of curved algebras is denoted $\cAlg$.
The composition \((\sfh_1,\sfh_0)=\bigl(A\rTTo^{(\sff_1,\sff_0)} B\) \(\rTTo^{(\sfg_1,\sfg_0)} C\bigr)\) is chosen so that \(\ddot\sfh=\ddot\sff\ddot\sfg\), namely, \(\sfh_1=\sff_1\sfg_1\) and \(\sfh_0=\sff_0\sfg_1+\sfg_0\).
The identity morphisms are \((\id,0)\).
The assignment $\cAlg\to\Alg$, \(A\mapsto\ddot{A}\), \(\sff=(\sff_1,\sff_0)\mapsto\ddot\sff\), is a functor.

\begin{definition}[\textit{cf.} \cite{Lyu-curved-coalgebras} Definition~1.4]
A \emph{unit-complemented curved algebra} \((A,m_2\), \(m_1,m_0,\eta,\sfv)\) is a curved algebra \((A,m_2,m_1,m_0,\eta)\) equipped with a \emph{splitting of the unit} \(\sfv:A\to\1\in\cv\), a morphism such that \(\eta\sfv=\id_\1\).
Morphisms of such algebras are morphisms of curved algebras (ignoring the splitting).
The category of unit-complemented curved algebras is denoted $\ucAlg$.
\end{definition}

The same notions as above can be expressed through shifted operations \(b_2:A[1]\tens A[1]\to A[1]\), \(b_1:A[1]\to A[1]\), \(b_0\in A[1]^1\), \(\deg b_n=1\) for \(n=0,1,2\), a degree $-1$ map \(\bfeta:\1\to A[1]\) and a degree 1 map \(\bv:A[1]\to\1\) (splitting of $\bfeta$).
The precise relationship with non-shifted operations is
\begin{gather*}
m_n =(-1)^n\sigma^{\tens n}\cdot b_n\cdot\sigma^{-1}: A^{\tens n} \to A, \qquad n=0,1,2,
\\
\eta =\bigl( \1 \rTTo^{\bfeta} A[1] \rTTo^{\sigma^{-1}} A \bigr),
\\
\sfv =\bigl( A \rTTo^\sigma A[1] \rTTo^\bv \1 \bigr).
\end{gather*}
The shifted operations define a (unit-complemented) curved algebra iff
\begin{gather*}
(1\tens b_2)b_2 +(b_2\tens1)b_2 =0, \qquad b_2b_1 +(1\tens b_1 +b_1\tens1)b_2 =0,
\\
b_1^2 +(b_0\tens1 +1\tens b_0)b_2 =0, \qquad b_0b_1 =0,
\\
(1\tens\bfeta)b_2 =1_{A[1]}, \quad (\bfeta\tens1)b_2 =-1_{A[1]}, \qquad \bfeta b_1 =0, \qquad \bfeta\bv =1_\1,
\end{gather*}
This makes clear relationship with \ainf-algebras, see \cite{Lyu-curved-coalgebras}.
\textit{E.g.} the first four equations can be combined into
\[ \sum_{r+k+t=n} (1^{\tens r}\tens b_k\tens1^{\tens t})b_{r+1+t} =0: A[1]^{\tens n} \to A[1], \qquad \forall\, n\ge0.
\]

A morphism of (unit-\hspace{0pt}complemented) curved algebras in shifted version is given by \(f_1:A[1]\to B[1]\in\cv\) and \(f_0\in B[1]^0\) such that
\begin{align*}
\sff_1 &=\bigl( A \rTTo^\sigma A[1] \rTTo^{f_1} B[1] \rTTo^{\sigma^{-1}} B \bigr),
\\
\sff_0 &=\bigl( \1 \rTTo^{f_0} B[1] \rTTo^{\sigma^{-1}} B \bigr).
\end{align*}
The equations satisfied by shifted data of a morphism are
\begin{gather*}
(f_1\tens f_1)b_2 =b_2f_1, \qquad f_1b_1 +(f_1\tens f_0)b_2 +(f_0\tens f_1)b_2 =b_1f_1,
\\
b_0 +f_0b_1 +(f_0\tens f_0)b_2 =b_0f_1, \qquad \bfeta_Af_1 =\bfeta_B, \qquad h_0=g_0+f_0g_1.
\end{gather*}
The first three equations can be combined into
\[ \sum_{i_1+\dots+i_k=n} (f_{i_1}\tens f_{i_2}\tdt f_{i_k})b^B_k =\sum_{r+k+t=n} (1^{\tens r}\tens b^A_k\tens1^{\tens t})f_{r+1+t}, \quad \forall\, n\ge0.
\]

\subsection{Curved coalgebras}
Let us dualize the notions of the previous section considering the symmetric monoidal functor \(\1[-]':\cz\to\cv^\op\), \(n\mapsto\1[n]'=\1[-n]\), in place of given \(\1[-]:\cz\to\cv\), \(n\mapsto\1[n]\).
Of course, the symmetric Monoidal additive category \((\cv^\op,\tens,\lambda^{-1})\) is no longer closed.
Nevertheless it inherits from $\cv$ the translation structure.
A morphism $X\to Z$ of degree $k$ in $\cv$ is the same as a morphism \(X\to Z\tens\1[k]\in\cv\) by \eqref{eq-V(XZ)k-V(XZk)}.
Correspondingly a morphism $Z\to X$ of degree $k$ in $\cv^\op$ is by definition \(Z\tens\1[k]\to X\in\cv^\op\), which is the same as \(Z\to X\tens\1[-k]=X\tens\1[k]'\in\cv^\op\).
In particular, the degree $k$ morphism \(\sigma^{-k}:Z[k]\to Z\) in $\cv$ leads to a degree $k$ morphism \(Z\to Z[k]=Z[-k]'\in\cv^\op\) which we may denote by \(\sigma^{-k}_{\cv^\op}\).
Considering curved algebras in $\cv^\op$ we come to the following coalgebras in $\cv$.
A little bit of confusion is introduced by our choice of dual operations: \(\delta_2=m_2^\op\), \(\delta_1=m_1^\op\) but \(\delta_0=-m_0^\op\); \(\sfg_1=\sff_1^\op\) but \(\sfg_0=-\sff_0^\op\).
We shall comment on this on appropriate occasions.

\begin{definition}[similar to Positselski \cite{0905.2621} Section~4.1]
A \emph{curved coalgebra} \((C,\delta_2,\delta_1\), \(\delta_0,\eps)\) in $\cv$ is a coassociative counital coalgebra \((C,\delta_2,\eps)\) with a degree~1 coderivation $\delta_1:C\to C$ and a degree~2 \emph{curvature} functional $\delta_0:C\to\1$ such that
\[ \delta_1^2 =\delta_2(1\tens\delta_0 -\delta_0\tens1), \qquad \delta_1\delta_0 =0.
\]
\end{definition}

\begin{proposition}
Let \(C=(C,\delta_2,\eps)\) be a (coassociative counital) coalgebra in $\cv$ equipped with a degree~1 morphism $\delta_1:C\to C$ and a degree~2 functional $\delta_0:C\to\1$.
Consider the degree~1 functional
\[ d =\bigl( \1[1]\tens C \rTTo^{\id\tens\eps} \1[1]\tens\1 \simeq \1[1] \rto{\sigma^{-1}} \1 \bigr).
\]
Existence of a coassociative counital coalgebra structure $\Delta$ on \(\ddot{C}=C\oplus\1[1]\tens C\simeq C\oplus C[1]\) such that
\begin{enumerate}
\renewcommand{\labelenumi}{(\arabic{enumi})}
\item \(\pr_1:\ddot{C}\to C\) is a counital coalgebra morphism;

\item \(\bigl(\ddot{C}\rto\Delta \ddot{C}\tens\ddot{C}\rTTo^{\pr_2\tens\pr_1} (\1[1]\tens C)\tens C\bigr)=\bigl(\ddot{C}\rTTo^{\pr_2} \1[1]\tens C\rTTo^{1\tens\delta_2} \1[1]\tens C\tens C\bigr)\);

\item \(\Delta(\pr_1\tens\pr_2d)=\Delta(\pr_2d\tens\pr_1)+\pr_1\delta_1:\ddot{C}\to C\);

\item \(\Delta(\pr_2d\tens\pr_2d)=-\pr_1\delta_0:\ddot{C}\to\1\);
\end{enumerate}
is equivalent to \((C,\delta_2,\delta_1,\delta_0,\eps)\) being a curved coalgebra.
\end{proposition}

\begin{proof}
It is obtained by dualising the proof of \propref{pro-curved-algebra-A-algebra-hatA}.
In order to see that the above condition~(4) is dual to condition~(4) of \propref{pro-curved-algebra-A-algebra-hatA} notice that \((d\tens1)(1\tens d)=d\tens d\) but \((1\tens d)(d\tens1)=-d\tens d\).
Another sign comes from the relation \(\delta_0=-m_0^\op\).

For instance, when $C$ is a curved coalgebra, we have
\begin{multline*}
\bigl( \ddot{C} \rto\Delta \ddot{C}\tens\ddot{C} \rTTo^{\pr_1\tens\pr_2} C\tens(\1[1]\tens C) \bigr)
\\
\hskip\multlinegap =\bigl( \ddot{C} \rTTo^{\pr_2} \1[1]\tens C \rTTo^{1\tens\delta_2} \1[1]\tens C\tens C \rTTo^{(12)} C\tens\1[1]\tens C \bigr) \hfill
\\
+\bigl( \ddot{C} \rTTo^{\pr_1} C \rto{\delta_2} C\tens C \rTTo^{\delta_1\tens\sigma\tens1} C\tens\1[1]\tens C \bigr).
\end{multline*}
As a corollary
\begin{multline*}
\pr' =\bigl( \ddot{C} \rto\Delta \ddot{C}\tens\ddot{C} \rTTo^{\pr_1\tens\pr_2d\sigma} C\tens\1[1] \bigr)
\\
=\bigl( \ddot{C} \rTTo^{\pr_2} \1[1]\tens C \rTTo^{(12)} C\tens\1[1] \bigr) +\bigl( \ddot{C} \rTTo^{\pr_1} C \rTTo^{\delta_1\tens\sigma} C\tens\1[1] \bigr).
\end{multline*}
Therefore, $\ddot{C}$ can be presented as a direct sum using projections $\pr'$ and $\pr_1$.
Notice also that
\[ \pr_2 =\bigl( \ddot{C} \rto\Delta \ddot{C}\tens\ddot{C} \rTTo^{\pr_2d\sigma\tens\pr_1} \1[1]\tens C \bigr).
\]
It's easy to compute
\begin{multline*}
\bigl( \ddot{C} \rto\Delta \ddot{C}\tens\ddot{C} \rTTo^{\pr'\tens\pr_2} (C\tens\1[1])\tens(\1[1]\tens C) \bigr)
\\
=\bigl( \ddot{C} \rTTo^{\pr_1} C \rTTo^{\delta_2^{(3)}} C\tens C\tens C \rTTo^{1\tens\delta_0\tens1} C\tens\1\tens\1\tens C \rTTo^{1\tens\sigma\tens\sigma\tens1} C\tens\1[1]\tens\1[1]\tens C \bigr).
\end{multline*}

Coassociativity of the comultiplication in $\ddot{C}$ can be verified case by case.
\QED\end{proof}

\begin{definition}[similar to Positselski \cite{0905.2621} Section~4.1]
A \emph{morphism of curved coalgebras} \(\sfg:C\to D\) is a pair \((\sfg_1,\sfg_0)\) consisting of a coalgebra morphism \(\sfg_1:C\to D\) and a degree~1 functional \(\sfg_0:C\to\1\) such that
\begin{equation}
\delta^C_1\sfg_1 +\delta^C_2(\sfg_0\tens\sfg_1 -\sfg_1\tens\sfg_0) =\sfg_1\delta^D_1, \qquad
\delta^C_0 -\delta^C_1\sfg_0 -\delta^C_2(\sfg_0\tens\sfg_0) =\sfg_1\delta^D_0.
\label{eq-deltas-gees}
\end{equation}
\end{definition}

\begin{proposition}
Let $C$, $D$ be curved coalgebras in $\cv$, let \(\sfg_1:C\to D\) be a homomorphism of counital coalgebras and let \(\sfg_0:C\to\1\) be a degree~1 functional.
Existence of a counital coalgebra homomorphism \(\ddot\sfg:\ddot{C}\to\ddot{D}\) such that
\begin{enumerate}
\renewcommand{\labelenumi}{(\arabic{enumi})}
\item \(\ddot\sfg\pr_1=\pr_1\sfg_1:\ddot{C}\to D\);

\item \(\ddot\sfg\pr_2d=-\pr_1\sfg_0+\pr_2d:\ddot{C}\to\1\);
\end{enumerate}
is equivalent to \((\sfg_1,\sfg_0):C\to D\) being a morphism of curved coalgebras.
\end{proposition}

Different signs in the second equations of \eqref{eq-f1m1-f1ff1m-m1f1} and \eqref{eq-deltas-gees} are due to the fact that \((\sff_0\tens1)(1\tens\sff_0)=\sff_0\tens\sff_0\) but \((1\tens\sfg_0)(\sfg_0\tens1)=-\sfg_0\tens\sfg_0\).

\begin{proof}
Dualize the proof of \propref{pro-algebra-homomorphism-morphism-curved-algebras}, taking into account that \(\sfg_0=-\sff_0^\op\).
Thus, properties (1)--(2) imply \eqref{eq-deltas-gees}.

Given a curved coalgebra morphism $(\sfg_1,\sfg_0)$, the map $\ddot\sfg$ is determined by \(\ddot\sfg\pr_1=\pr_1\sfg_1\) and
\[ \ddot\sfg\pr_2 =\bigl( \ddot{C} \rto\Delta \ddot{C}\tens\ddot{C} \rTTo^{(-\pr_1\sfg_0+\pr_2d)\sigma\tens\pr_1\sfg_1} \1[1]\tens D \bigr).
\]
Case by case one checks that $\ddot\sfg$ is an coalgebra morphism.
\QED\end{proof}

The category whose objects are curved coalgebras and whose morphisms are morphisms of curved coalgebras is denoted $\cCoalg$.
The composition \((\sfh_1,\sfh_0)=\bigl(C\rTTo^{(\sff_1,\sff_0)} D\) \(\rTTo^{(\sfg_1,\sfg_0)} E\bigr)\) is chosen so that \(\ddot\sfh=\ddot\sff\ddot\sfg\), namely, \(\sfh_1=\sff_1\sfg_1\) and \(\sfh_0=\sff_0+\sff_1\sfg_0\).
The identity morphisms are \((\id,0)\).
The assignment $\cCoalg\to\Coalg$, \(C\mapsto\ddot{C}\), \(\sfg=(\sfg_1,\sfg_0)\mapsto\ddot\sfg\), is a functor.

\begin{definition}
A \emph{counit-complemented curved coalgebra} \((C,\delta_2,\delta_1,\delta_0,\eps,\sfw)\) is a curved coalgebra \((C,\delta_2,\delta_1,\delta_0,\eps)\) equipped with a \emph{splitting of the counit} \(\sfw:\1\to C\in\cv\), an element of $C^0$ such that \(\sfw\cdot\eps=1_\1\).
Morphisms of such coalgebras are morphisms of curved coalgebras (ignoring the splitting).
The category of counit-complemented curved coalgebras is denoted $\ucCoalg$.
\end{definition}

The same notions as above can be expressed through shifted operations \(\xi_2:C[-1]\to C[-1]\tens C[-1]\), \(\xi_1:C[-1]\to C[-1]\), \(\xi_0:C[-1]\to\1\), \(\deg\xi_n=1\) for \(n=0,1,2\), a degree $-1$ map \(\beps:C[-1]\to\1\) and a degree 1 element \(\bw\in C[-1]^1\) (splitting of $\beps$).
The precise relationship with non-shifted operations is
\begin{gather*}
\delta_n =(-1)^n\sigma^{-1}\cdot\xi_n\cdot\sigma^{\tens n}: C\to C^{\tens n}, \qquad  n=0,1,2,
\\
\eps =\bigl( C \rTTo^{\sigma^{-1}} C[-1] \rTTo^\beps \1 \bigr),
\\
\sfw =\bigl( \1 \rTTo^\bw C[-1] \rTTo^\sigma C \bigr).
\end{gather*}
The shifted operations define a (counit-complemented) curved coalgebra iff
\begin{gather*}
\xi_2(1\tens\xi_2) +\xi_2(\xi_2\tens1) =0, \qquad \xi_1\xi_2 +\xi_2(1\tens\xi_1 +\xi_1\tens1) =0,
\\
\xi_1^2 +\xi_2(1\tens\xi_0 +\xi_0\tens1) =0, \qquad \xi_1\xi_0 =0,
\\
\xi_2(1\tens\beps) =-1_{C[-1]}, \qquad \xi_2(\beps\tens1) =1_{C[-1]}, \qquad \xi_1\beps =0, \qquad \sfw \cdot \eps =1_\1.
\end{gather*}
These relations except the last one arise in the theory of \ainf-coalgebras, \textit{e.g.} \cite[Section~1.8]{Lyu-curved-coalgebras}.
In fact, the first four equations can be combined into
\[ \sum_{r+k+t=n} \xi_{r+1+t}(1^{\tens r}\tens\xi_k\tens1^{\tens t}) =0: C[-1]\to C[-1]^{\tens n}, \qquad \forall\, n\ge0.
\]

A morphism of (counit-complemented) curved coalgebras in shifted version is given by \(g_1:C[-1]\to D[-1]\in\cv\) and \(g_0:C[-1]\to\1\in\cv\) such that
\begin{align*}
\sfg_1 &=\bigl( C \rTTo^{\sigma^{-1}} C[-1] \rTTo^{g_1} D[-1] \rTTo^\sigma D \bigr),
\\
\sfg_0 &=\bigl( C \rTTo^{\sigma^{-1}} C[-1] \rTTo^{g_0} \1 \bigr).
\end{align*}
The equations satisfied by shifted data of a morphism are
\begin{gather*}
\xi^C_2(g_1\tens g_1) =g_1\xi^D_2, \qquad \xi^C_1g_1 +\xi^C_2(g_0\tens g_1 +g_1\tens g_0) =g_1\xi^D_1,
\\
\xi^C_0 +\xi^C_1g_0 +\xi^C_2(g_0\tens g_0) =g_1\xi^D_0, \qquad g_1\beps^D =\beps^C, \qquad h_0 =f_0+f_1g_0.
\end{gather*}
The first three equations can be combined into
\begin{multline}
\sum_{r+k+t=n} g_{r+1+t}(1^{\tens r}\tens\xi_k\tens1^{\tens t}) =\sum_{i_1+\dots+i_k=n} \xi_k(g_{i_1}\tens g_{i_2}\tdt g_{i_k}):
\ifx\chooseClass1
\\
\fi
C[-1]\to D[-1]^{\tens n}.
\label{eq-g1xi1-xiggg}
\end{multline}

\subsection{Curved homotopy coalgebras}
Category $\chCoalg$ of \emph{curved homotopy coalgebras} consists of homotopy coalgebras $X\tilde{T}$ equipped with extra data.
Objects of the category $\chCoalg$ by definition are objects $X\in\Ob\cv$ equipped with a degree~1 coderivation \(\delta_1:XT\to XT\) coming from a degree~1 map \(\check\delta_1=\delta_1\pr_1:XT\to X\) with finite support as in \remref{rem-xi(1)-coderivation} and a degree~2 functional \(\delta_0:XT\to\1\) with finite support such that \((XT,\delta_2,\delta_1,\delta_0,\pr_0,\inj_0)\) is a counit-\hspace{0pt}complemented curved coalgebra for the cut comultiplication $\delta_2$, counit \(\eps=\pr_0\) and the splitting \(\sfw=\inj_0\).
In terms of components \(\delta_{1;m}:X^{\tens m}\to X\), \(\deg\delta_{1;m}=1\), and \(\delta_{0;m}:X^{\tens m}\to\1\), \(\deg\delta_{0;m}=2\), $m\in\NN$, for each $n\in\NN$ we require that
\begin{align}
\sum_{a+k+c=n} (1^{\tens a}\tens\delta_{1;k}\tens1^{\tens c})\delta_{1;a+1+c} &=1\tens\delta_{0;n-1} -\delta_{0;n-1}\tens1: X^{\tens n}\to X,
\label{eq-1delta1-delta-1delta-delta1-XX}
\\
\sum_{a+k+c=n} (1^{\tens a}\tens\delta_{1;k}\tens1^{\tens c})\delta_{0;a+1+c} &=0: X^{\tens n} \to \1.
\label{eq-1delta1-delta-X1}
\end{align}
The coderivations $\delta_1$ and \(\theta=\delta_2(1\tens\delta_0-\delta_0\tens1):XT\to XT\) extend to coderivations of homotopy coalgebras \(\tilde\delta_1:X\tilde{T}\to X\tilde{T}\) and \(\tilde\theta:X\tilde{T}\to X\tilde{T}\) coming from \(\check\delta_1=\delta_1\pr_1\) and \(\check\theta=\delta_2(1\tens\delta_0-\delta_0\tens1)\pr_1:XT\to X\) as in \propref{pro-fgxi-XT-Y} and \exaref{exa-(fg)-coderivations-exi-xie}.
It follows from \secref{Components-coderivations-lax-coalgebras} that \(\tilde\delta_1^2=\tilde\theta\) and \(\tilde\delta_1(1)\tilde\delta_0=0:X\hat{T}\to\1\), where \(\tilde\delta_0:X\hat{T}\to\1\) extends \(\delta_0:XT\to\1\) with finite support by zeroes, see \eqref{eq-MM-phi-N}.

A \emph{morphism of curved homotopy coalgebras} \(\sfg:(X,\delta_1,\delta_0)\to(Y,\delta_1,\delta_0)\) is a pair \((\sfg_1,\sfg_0)\) consisting of a homotopy coalgebra morphism \(\sfg_1:X\tilde{T}\to Y\tilde{T}\in\fhCoalg\) coming from a morphism \(\check\sfg_1:XT\to Y\) with finite support as in \propref{pro-f-XT-Y-defines-f} and a degree~1 functional \(\sfg_0:XT\to\1\) with finite support such that
\begin{align}
\tilde\delta_1\sfg_1 +\wt{[\delta_2(\sfg_0\tens1 -1\tens\sfg_0)]}\sfg_1 &=\sfg_1\tilde\delta_1: X\tilde{T} \to Y\tilde{T},
\label{eq-del1g1}
\\
\tilde\delta_0 -\tilde\delta_1\tilde\sfg_0 -\wt{[\delta_2(\sfg_0\tens\sfg_0)]} &=\sfg_1(1)\tilde\delta_0: X\hat{T} \to \1.
\label{eq-delta0-delta1gee0}
\end{align}
In the first equation all three terms are $\sfg_1$\n-coderivations.
The coderivation in brackets \(\wt{[\delta_2(\sfg_0\tens1 -1\tens\sfg_0)]}:X\tilde{T}\to X\tilde{T}\) comes from the map \(\delta_2(\sfg_0\tens1 -1\tens\sfg_0)\pr_1:XT\to X\) as in \propref{pro-fgxi-XT-Y}.
By \secref{Components-coderivations-lax-coalgebras} these equations can be written in components:
\begin{align*}
&\sum_{a+k+c=n} (1^{\tens a}\tens\delta_{1;k}\tens1^{\tens c})\sfg_{1,a+1+c} +\sum_{m=0}^{n-1} (\sfg_{0,m}\tens1^{\tens(n-m)})\sfg_{1,1+n-m}
\\
&-\sum_{m=0}^{n-1} (1^{\tens(n-m)}\tens\sfg_{0,m})\sfg_{1,n-m+1} =\hspace*{-.6em}\sum_{k_1+\dots+k_l=n}\hspace*{-.6em} (\sfg_{1,k_1}\tdt \sfg_{1,k_l})\delta_{1;l}: X^{\tens n} \to Y,
\\
&\delta_{0;n} -\sum_{a+k+c=n} (1^{\tens a}\tens\delta_{1;k}\tens1^{\tens c}) \sfg_{0,a+1+c} -\sum_{a+c=n}\sfg_{0,a}\tens\sfg_{0,c}
\\
&\hphantom{-\sum (1^{\tens(n-m)}\tens\sfg_{0,m})\sfg_{1,n-m+1}} =\sum_{k_1+\dots+k_l=n} (\sfg_{1,k_1}\tdt \sfg_{1,k_l})\delta_{0;l}: X^{\tens n} \to \1.
\end{align*}
Therefore, equations \eqref{eq-del1g1}, \eqref{eq-delta0-delta1gee0} can be presented in equivalent form
\begin{align}
\delta_1\check\sfg_1 +\delta_2(\sfg_0\tens1 -1\tens\sfg_0)\check\sfg_1 &=e(1)\sfg_1(1)\tilde\delta_1\pr_1: XT \to Y,
\label{eq-del1e(1)g1(1)}
\\
\delta_0 -\delta_1\sfg_0 -\delta_2(\sfg_0\tens\sfg_0) &=e(1)\sfg_1(1)\tilde\delta_0: XT \to \1.
\label{eq-delta0-delta1gee0-e(1)g1(1)}
\end{align}
Recovering $\sfg_1(1)$ from $\check\sfg_1$ via \eqref{eq-f(1)prI-XX-Delta-rest-XY} we get an equivalent form of these equations
\begin{align}
\delta_1\check\sfg_1 +\delta_2(\sfg_0\tens1 -1\tens\sfg_0)\check\sfg_1 &=\bigl(XT \rto{\hat\Delta} XT\hat{T} \rTTo^{\check\sfg_1\hat{T}} Y\hat{T} \rto{\check\delta_1} Y\bigr) \notag
\\
&\equiv \sum_{k\in\NN} \Bigl(XT \rto{\Delta_k} XT^{\tens k} \rTTo^{\check\sfg_1^{\tens k}} Y^{\tens k} \rto{\delta_{1;k}} Y\Bigr),
\label{eq-XT-XT-Y-Y}
\\
\delta_0 -\delta_1\sfg_0 -\delta_2(\sfg_0\tens\sfg_0) &=\bigl(XT \rto{\hat\Delta} XT\hat{T} \rTTo^{\check\sfg_1\hat{T}} Y\hat{T} \rto{\tilde\delta_0} \1\bigr) \notag
\\
&\equiv \sum_{k\in\NN} \Bigl(XT \rto{\Delta_k} XT^{\tens k} \rTTo^{\check\sfg_1^{\tens k}} Y^{\tens k} \rto{\delta_{0;k}} \1\Bigr).
\label{eq-XT-XT-Y-1}
\end{align}

The composition \((\sfh_1,\sfh_0)\) of morphisms \((\sff_1,\sff_0)\) and \((\sfg_1,\sfg_0)\) is given by \(\sfh_1=\sff_1\sfg_1\in\fhCoalg\) (see \eqref{eq-asdfghjkl}) and \(\sfh_0=\sff_0+e(1)\sff_1(1)\tilde\sfg_0\).
Note that \(XT\rTTo^{e(1)} X\hat{T}\rTTo^{\sff_1(1)}\) \(Y\hat{T}\rto{\tilde\sfg_0} \1\) has finite support due to \eqref{eq-inm-f-prn-sum}.
Equation~\eqref{eq-del1g1} for $\sfh$ is proven as follows:
\begin{multline}
\sfh_1\tilde\delta_1 =\sff_1\sfg_1\tilde\delta_1 =\sff_1\tilde\delta_1\sfg_1 +\sff_1\wt{[\delta_2(\sfg_0\tens1 -1\tens\sfg_0)]}\sfg_1
\\
=\tilde\delta_1\sff_1\sfg_1 +\wt{[\delta_2(\sff_0\tens1 -1\tens\sff_0)]}\sff_1\sfg_1 +\wt{[\delta_2(e(1)\sff_1(1)\sfg_0\tens1 -1\tens e(1)\sff_1(1)\sfg_0)]}\sff_1\sfg_1
\\
=\tilde\delta_1\sfh_1 +\wt{[\delta_2(\sfh_0\tens1 -1\tens\sfh_0)]}\sfh_1,
\label{eq-h1delta1-delta1h1-delta2}
\end{multline}
since the following two $\sff_1$\n-coderivations are equal
\[ \sff_1\wt{[\delta_2(\sfg_0\tens1 -1\tens\sfg_0)]} =\wt{[\delta_2(e(1)\sff_1(1)\sfg_0\tens1 -1\tens e(1)\sff_1(1)\sfg_0)]}\sff_1: X\tilde{T} \to Y\tilde{T}.
\]
In fact, by \secref{Components-coderivations-lax-coalgebras} both sides come from the same map, $n\in\NN$:
\[ \sum_{k_1+\dots+k_l=n} (\sff_{1,k_1}\tdt\sff_{1,k_l}) (\sfg_{0,l-1}\tens1 -1\tens\sfg_{0,l-1}): X^{\tens n} \to Y.
\]

In order to prove equation~\eqref{eq-delta0-delta1gee0-e(1)g1(1)} for $\sfh$ use equations \eqref{eq-del1g1}--\eqref{eq-delta0-delta1gee0-e(1)g1(1)} for $\sff$ and $\sfg$ and the observation that all factors below have finite support:
\begin{multline}
e(1)\sfh_1(1)\tilde\delta_0 =e(1)\sff_1(1)\sfg_1(1)\tilde\delta_0 =e(1)\sff_1(1) \bigl( \tilde\delta_0 -\tilde\delta_1\tilde\sfg_0 -\wt{[\delta_2(\sfg_0\tens\sfg_0)]} \bigr)
\\
=\delta_0 -\delta_1\sff_0 -\delta_2(\sff_0\tens\sff_0) -[\delta_1e(1)\sff_1(1) +\delta_2(\sff_0\tens1 -1\tens\sff_0)e(1)\sff_1(1)]\tilde\sfg_0
\ifx\chooseClass1
\\ \hfill
\fi
-\delta_2(e(1)\sff_1(1)\tilde\sfg_0\tens e(1)\sff_1(1)\tilde\sfg_0)
\\
=\delta_0 -\delta_1(\sff_0+e(1)\sff_1(1)\tilde\sfg_0) -\delta_2[(\sff_0+e(1)\sff_1(1)\tilde\sfg_0)\tens(\sff_0+e(1)\sff_1(1)\tilde\sfg_0)]
\ifx\chooseClass1
\\
\fi
=\delta_0 -\delta_1\sfh_0 -\delta_2(\sfh_0\tens\sfh_0).
\label{eq-e(1)h(1)delta0-delta0}
\end{multline}

The identity morphisms for so defined composition are \((\id,0)\).
The category $\chCoalg$ is described in full.

\subsection{Bimodule over categories of curved coalgebras}
Let us construct a $\cCoalg$-$\chCoalg$-bimodule, a functor \(\CCoalg:\cCoalg^\op\times\chCoalg\to\Set\).
Let $C$ be a curved coalgebra and $X\tilde{T}$ be a curved homotopy coalgebra.
Define \(\CCoalg(C,X\tilde{T})\) as the set of pairs \((\sfg_1,\sfg_0)\) consisting of a homotopy coalgebra morphism \(\sfg_1:C\to X\tilde{T}\) and a degree~1 functional \(\sfg_0:C\to\1\) such that
\begin{align}
\delta^C_1\sfg_1 +\delta^C_2(\sfg_0\tens1 -1\tens\sfg_0)\sfg_1 &=\sfg_1\tilde\delta_1: C \to X\tilde{T},
\label{eq-g1delta1-CXT}
\\
\delta^C_0 -\delta^C_1\sfg_0 -\delta^C_2(\sfg_0\tens\sfg_0) &=\sfg_1(1)\tilde\delta_0: C \to \1.
\label{eq-g1delta0-C1}
\end{align}
The left action of $\cCoalg$ is the precomposition: for any \(\sff=(\sff_1,\sff_0):C\to D\in\cCoalg\) there is a map
\begin{align}
\CCoalg(\sff,X\tilde{T}): \CCoalg(D,X\tilde{T}) &\longrightarrow \CCoalg(C,X\tilde{T}),
\label{eq-CCoalg(DXT)-CCoalg(CXT)}
\\
(\sfg_1,\sfg_0) &\longmapsto (\sff_1,\sff_0) \cdot (\sfg_1,\sfg_0) =(\sff_1\sfg_1,\sff_0+\sff_1\sfg_0). \notag
\end{align}
The fact that the last pair satisfies \eqref{eq-g1delta1-CXT} and \eqref{eq-g1delta0-C1} is verified similarly to \eqref{eq-h1delta1-delta1h1-delta2} and \eqref{eq-e(1)h(1)delta0-delta0}.
The right action of $\chCoalg$ is the postcomposition: for any \(\sfh=(\sfh_1,\sfh_0):X\tilde{T}\to Y\tilde{T}\in\chCoalg\) there is a map
\begin{align}
\CCoalg(C,\sfh): \CCoalg(C,X\tilde{T}) &\longrightarrow \CCoalg(C,Y\tilde{T}),
\label{eq-CCoalg(CXT)-CCoalg(CYT)}
\\
(\sfg_1,\sfg_0) &\longmapsto (\sfg_1,\sfg_0) \cdot (\sfh_1,\sfh_0) =(\sfg_1\sfh_1,\sfg_0+\sfg_1(1)\tilde\sfh_0). \notag
\end{align}
The fact that the last pair satisfies \eqref{eq-g1delta1-CXT} and \eqref{eq-g1delta0-C1} is verified similarly to \eqref{eq-h1delta1-delta1h1-delta2} and \eqref{eq-e(1)h(1)delta0-delta0}.
Associativity of the actions ($\CCoalg$ being a functor) is obvious from the given formulae.

I doubt that it makes sense to consider also morphisms \(X\tilde{T}\to C\).
They are not needed for this paper anyway.

Due to Propositions \ref{pro-lCoalg(CXT)-V(CX)}, \ref{pro-(fg)coder(CXT)-V(CX)} the set \(\CCoalg(C,X\tilde{T})\) is in bijection with the set of pairs \((\check\sfg_1,\sfg_0)\in\cv(C,X)\times\und\cv(C,\1)^1\) such that
\begin{align}
\delta^C_1\check\sfg_1 +\delta^C_2(\sfg_0\tens1 -1\tens\sfg_0)\check\sfg_1 &=\bigl(C \rto{\hat\Delta} C\hat{T} \rTTo^{\check\sfg_1\hat{T}} X\hat{T} \rto{\check\delta_1} X\bigr) \notag
\\
&\equiv \sum_{k\in\NN} \Bigl(C \rto{\Delta_k} C^{\tens k} \rTTo^{\check\sfg_1^{\tens k}} X^{\tens k} \rto{\delta_{1;k}} X\Bigr),
\label{eq-CCXX-C-XT}
\\
\delta^C_0 -\delta^C_1\sfg_0 -\delta^C_2(\sfg_0\tens\sfg_0) &=\bigl(C \rto{\hat\Delta} C\hat{T} \rTTo^{\check\sfg_1\hat{T}} X\hat{T} \rto{\tilde\delta_0} \1\bigr) \notag
\\
&\equiv \sum_{k\in\NN} \Bigl(C \rto{\Delta_k} C^{\tens k} \rTTo^{\check\sfg_1^{\tens k}} X^{\tens k} \rto{\delta_{0;k}} \1\Bigr).
\label{eq-CCX1-C-XT}
\end{align}
In fact, all three summands of \eqref{eq-g1delta1-CXT} are $\sfg_1$\n-coderivations and $\sfg_1$ is recovered from $\check\sfg_1$ via \eqref{eq-t(1)-C-XT}.

\section{Bar and cobar constructions}
We describe cobar- and bar-constructions.

\subsection{Cobar-construction}
Let us construct a functor \(\Cobar:\ucCoalg\to\ucAlg\), the cobar-\hspace{0pt}construction.
Let \(C=(C,\xi_2,\xi_1,\xi_0,\beps,\bw)\) be a counit-complemented curved coalgebra.
Decompose the idempotent \(1-\eps\sfw:C\to C\) into a projection \(\opr:C\to\bar{C}\) and an injection \(\oin:\bar{C}\to C\) so that \(\oin\opr=1_{\bar{C}}\).
Clearly, \(\bar{C}=\Ker\eps\).
The morphisms \(\sigma\opr\sigma^{-1}:C[-1]\to\bar{C}[-1]\) and \(\sigma\oin\sigma^{-1}:\bar{C}[-1]\to C[-1]\) are also denoted $\opr$ and $\oin$ by abuse of notations.
They split the idempotent \(1-\sigma\eps\sfw\sigma^{-1}=1-\beps\cdot\bw:C[-1]\to C[-1]\) and \(\bar{C}[-1]=\Ker\beps\).

\begin{proposition}
Define \(\Cobar C\) as \(\bar{C}[-1]T\) equipped with the multiplication \(m^{\Cobar C}_2\) in the tensor algebra, the unit
\(\eta^{\Cobar C}=\inj_0:\1=\bar{C}[-1]T^0\hookrightarrow\bar{C}[-1]T\), the splitting \(\sfv^{\Cobar C}=\pr_0:\bar{C}[-1]T\to\bar{C}[-1]T^0=\1\), the degree~1 derivation \(m^{\Cobar C}_1=\bar\xi:\bar{C}[-1]T\to\bar{C}[-1]T\) given by its components
\[ \bar\xi_n =\bigl( \bar{C}[-1] \rMono^\oin C[-1] \rTTo^{\xi_n} C[-1]^{\tens n} \rTTo^{\opr^{\tens n}} \bar{C}[-1]^{\tens n} \bigr), \qquad n=0,1,2,
\]
and a degree~2 element
\begin{equation}
m^{\Cobar C}_0 =-\Bigl(\bw\tens\bw +\sum_{n=0}^2\bw\xi_n\Bigr) (\opr T)\in (\bar{C}[-1]T)^2.
\label{eq-mCobarC0-CT}
\end{equation}
Then \(\Cobar C\) is a unit-complemented curved algebra.
\end{proposition}

\begin{proof}
If $n=0,1$, then \(\bar\xi_n\oin^{\tens n}=\oin\xi_n:\bar{C}[-1]\to C[-1]^{\tens n}\).
Furthermore,
\[ \bar\xi_2\oin^{\tens2} =\oin\xi_2 [(1-\beps\bw)\tens(1-\beps\bw)].
\]
By abuse of notation the morphism
\begin{equation}
\xi_2[(1-\beps\bw)\tens(1-\beps\bw)] =\xi_2 +1\tens\bw -\bw\tens1 -\beps(\bw\tens\bw): C[-1] \to C[-1]^{\tens2}.
\label{eq-xi2-pr-pr}
\end{equation}
is also denoted $\bar\xi_2$.
Actually, projections in expressions
\[ (m^{\Cobar C}_0)_2 =-(\bw\tens\bw +\bw\xi_2)\opr^{\tens2} =-\bw\xi_2\opr^{\tens2}, \quad (m^{\Cobar C}_0)_k =-\bw\xi_k\opr^{\tens k}
\]
if $k=0,1$, may be informally omitted since
\begin{gather*}
(\bw\tens\bw +\bw\xi_2)\opr^{\tens2}\oin^{\tens2} =(\bw\tens\bw +\bw\xi_2)[(1-\beps\bw)\tens(1-\beps\bw)] =\bw\tens\bw +\bw\xi_2,
\\
\bw\xi_1\opr\oin =\bw\xi_1(1-\beps\bw) =\bw\xi_1,
\end{gather*}
see \eqref{eq-xi2-pr-pr}.
Thus, \(\bw\tens\bw+\bw\xi_2\) belongs to the subset \((\bar{C}[-1]^{\tens2})^2\hookrightarrow(C[-1]^{\tens2})^2\) and $\bw\xi_1$ belongs to the subset \(\bar{C}[-1]^2\hookrightarrow C[-1]^2\).

Both sides of the equation
\[ (m^{\Cobar C}_1)^2 =(m^{\Cobar C}_0\tens1 -1\tens m^{\Cobar C}_0) m^{\Cobar C}_2
\]
are derivations.
It is equivalent to its restriction to generators $\bar{C}[-1]$:
\begin{multline}
\sum_{r+k+t=n}\bar\xi_{r+1+t}(1^{\tens r}\tens\bar\xi_k\tens1^{\tens t}) =(m^{\Cobar C}_0)_{n-1}\tens1 -1\tens(m^{\Cobar C}_0)_{n-1}:
\\
\bar C[-1]\to\bar C[-1]^{\tens n}.
\label{eq-xi-xi-m0-m0}
\end{multline}
In fact, \eqref{eq-xi-xi-m0-m0} is obvious for $n=0$.
It says for $n=1$ that
\[ \xi_1^2 +\bar\xi_2(1\tens\xi_0 +\xi_0\tens1) =(1\tens\bw -\bw\tens1) (1\tens\xi_0 +\xi_0\tens1) =(\bw\xi_0) -\xi_0\bw +\xi_0\bw -(\bw\xi_0) =0
\]
as it has to be.
If $n=2$ or $n\ge4$, then the left hand side of \eqref{eq-xi-xi-m0-m0} is
\begin{align*}
&\xi_1\xi_n +\bar\xi_2(\xi_{n-1}\tens1 +1\tens\xi_{n-1}) +\dots+\xi_{n-1}\sum_{r+2+t=n}1^{\tens r}\tens\bar\xi_2\tens1^{\tens t} +\dots
\\
&=(1\tens\bw -\bw\tens1)(\xi_{n-1}\tens1 +1\tens\xi_{n-1})
\\
&\hspace*{7em} +\xi_{n-1}\sum_{r+2+t=n}(1^{\tens(r+1)}\tens\bw\tens1^{\tens t} -1^{\tens r}\tens\bw\tens1^{\tens(1+t)})
\\
&=-\xi_{n-1}\tens\bw +1\tens\bw\xi_{n-1} -\bw\xi_{n-1}\tens1 -\bw\tens\xi_{n-1}
\\
&\hspace*{16em} +\xi_{n-1}(1^{\tens(n-1)}\tens\bw -\bw\tens1^{\tens(n-1)})
\\
&=1\tens\bw\xi_{n-1} -\bw\xi_{n-1}\tens1,
\end{align*}
as claimed.
If $n=3$, then the left hand side of \eqref{eq-xi-xi-m0-m0} is
\begin{align*}
&\xi_1\xi_3 +\bar\xi_2(\bar\xi_2\tens1 +1\tens\bar\xi_2) +\dots
\\
&=(\xi_2 +1\tens\bw -\bw\tens1)[(\xi_2 -\bw\tens1)\tens1 +1\tens(\xi_2 +1\tens\bw)] -\xi_2(\xi_2\tens1 +1\tens\xi_2)
\\
&=1\tens(\bw\xi_2 +\bw\tens\bw) -(\bw\xi_2 +\bw\tens\bw)\tens1,
\end{align*}
as claimed.

The $n$\n-th component of the expression \(m^{\Cobar C}_0m^{\Cobar C}_1\) is
\begin{align}
&m^{\Cobar C}_0m^{\Cobar C}_1\pr_n \notag
\\
&=-\bw\tens\bw\bar\xi_{n-1} +\bw\bar\xi_{n-1}\tens\bw -\sum_{r+k+t=n}\bw\xi_{r+1+t}(1^{\tens r}\tens\bar\xi_k\tens1^{\tens t}) \notag
\\
&=-\bw\tens\bw\bar\xi_{n-1} +\bw\bar\xi_{n-1}\tens\bw \notag
\\
&\hspace*{4em} -\bw\xi_{n-1}\sum_{r+2+t=n}(1^{\tens(r+1)}\tens\bw\tens1^{\tens t} -1^{\tens r}\tens\bw\tens1^{\tens(1+t)}) \notag
\\
&=-\bw\tens\bw\bar\xi_{n-1} +\bw\bar\xi_{n-1}\tens\bw -\bw\xi_{n-1}(1^{\tens(n-1)}\tens\bw -\bw\tens1^{\tens(n-1)}).
\label{eq-m0m1-wwwwwwwwwwww}
\end{align}
If $n\ne3$, then \(\bar\xi_{n-1}=\xi_{n-1}\) and the obtained expression equals
\[ -\bw\tens\bw\xi_{n-1} +\bw\xi_{n-1}\tens\bw -\bw\xi_{n-1}\tens\bw +\bw\tens\bw\xi_{n-1} =0.
\]
If $n=3$, then \eqref{eq-m0m1-wwwwwwwwwwww} equals
\[ -\bw\tens[\bw(1\tens\bw -\bw\tens1)] +[\bw(1\tens\bw -\bw\tens1)]\tens\bw =\bw\tens\bw\tens\bw(-1-1+1+1) =0.
\]
Thus \(m^{\Cobar C}_0m^{\Cobar C}_1=0\).
We obtain a map \(\Ob\ucCoalg\to\Ob\ucAlg\).
\QED\end{proof}

\begin{proposition}
The map \(\Ob\ucCoalg\to\Ob\ucAlg\) extends to a functor \(\Cobar:\ucCoalg\to\ucAlg\), taking a morphism \(g=(g_1,g_0):C\to D\) to the morphism
\[ \sfCobar g =\sff =(\sff_1,\sff_0): \bar{C}[-1]T \to \bar{D}[-1]T,
\]
where the algebra homomorphism \(\sfCobar_1g=\sff_1=\bar{g}\) is specified by its components
\begin{equation*}
\begin{split}
\bar g_1 &=\bigl( \bar{C}[-1] \rMono^\oin C[-1] \rTTo^{g_1} D[-1] \rTTo^\opr \bar{D}[-1] \bigr),
\\
\bar g_0 &=\bigl( \bar{C}[-1] \rMono^\oin C[-1] \rTTo^{g_0} \1 \bigr),
\end{split}
\end{equation*}
and the degree~1 element is
\begin{equation*}
\sfCobar_0g =\sff_0 =\bigl( \1 \rTTo^\bw C[-1] \rTTo^{g_0+g_1} D[-1]T \rTTo^{\opr T} \bar{D}[-1]T \bigr).
\end{equation*}
\end{proposition}

\begin{proof}
Let us check that $\sff$ is indeed a morphism of $\ucAlg$.
It is required that
\[ \bar g\bar\xi +(\bar g\tens\sff_0 -\sff_0\tens\bar g)m_2 =\bar\xi\bar g: \bar{C}[-1]T \to \bar{D}[-1]T.
\]
All three terms of this equation are $\bar{g}$\n-derivations.
Hence, the equation is equivalent to its restriction to $\bar{C}[-1]$:
\[ (\bar g_0+\bar g_1)\bar\xi +[(\bar g_0+\bar g_1)\tens\sff_0 -\sff_0\tens(\bar g_0+\bar g_1)]m_2 =(\bar\xi_0+\bar\xi_1+\bar\xi_2)\bar g: \bar{C}[-1] \to \bar{D}[-1]T,
\]
which means that for all $n\ge0$
\begin{multline*}
\bar g_1\bar\xi_n +(\bar g_n\tens\bw g_0 +\bar g_{n-1}\tens\bw g_1\opr -\bw g_0\tens\bar g_n -\bw g_1\opr\tens\bar g_{n-1})m_2
\\
=\sum_{i_1+\dots+i_k=n} \bar\xi_k(\bar g_{i_1}\tens\bar g_{i_2}\tdt\bar g_{i_k}): \bar C[-1] \to \bar D[-1]^{\tens n}.
\end{multline*}
Notice that \(\bar g_1\oin=\oin g_1:\bar C[-1]\to D[-1]\).
Hence, the above equation can be written as
\begin{multline*}
\oin g_1\xi_n\opr^{\tens n} +\oin(g_n\opr^{\tens n}\tens\bw g_0 +g_{n-1}\opr^{\tens(n-1)}\tens\bw g_1\opr
\\
\hfill -\bw g_0\tens g_n\opr^{\tens n} -\bw g_1\opr\tens g_{n-1}\opr^{\tens(n-1)}) \quad
\\
\hskip\multlinegap =\oin\sum_{i_1+\dots+i_k=n} \xi_k(g_{i_1}\tdt g_{i_k})\opr^{\tens n} \hfill
\\
+\oin\sum_{i_1+i_2=n} \bigl( g_{i_1}\tens\bw g_{i_2} -\bw g_{i_1}\tens g_{i_2} \bigr)\opr^{\tens n}: \bar C[-1] \to \bar D[-1]^{\tens n},
\end{multline*}
and this follows from \eqref{eq-g1xi1-xiggg}.

Another equation has to be verified:
\[ m^{\Cobar D}_0 -\sff_0\bar\xi -(\sff_0\tens\sff_0)m_2 =m^{\Cobar C}_0\bar g: \1 \to \bar{D}[-1]T.
\]
Compose it with $-\oin T$ and expand terms:
\begin{multline*}
\bw\tens\bw +\bw\check\xi +\bw g_1(1-\beps\bw)\check\xi[(1-\beps\bw)T] +\bw\tilde{g}\tens\bw\tilde{g}
\ifx\chooseClass1
\\
\fi
=\bw\tilde{g}\tens\bw\tilde{g} +\sum_{k\ge0}\bw\xi_k\tilde{g}^{\tens k}: \1 \to D[-1]T,
\end{multline*}
where \(\check\xi=\xi_0+\xi_1+\xi_2:D[-1]\to D[-1]T\) and \(\tilde{g}=\tilde{g}_0+\tilde{g}_1\overset{\text{def}}=g_0+g_1(1-\beps\bw):C[-1]\to D[-1]T\).
After cancelling two summands the equation is written as
\begin{multline*}
\delta_{n,2}\bw\tens\bw +\bw\xi_n +(\bw g_1 -\bw)\xi_n(1-\beps\bw)^{\tens n} 
\ifx\chooseClass1
\\
\fi
=\sum_{i_1+\dots+i_k=n} \bw\xi_k(g_{i_1}\tens g_{i_2}\tdt g_{i_k})(1-\beps\bw)^{\tens n}.
\end{multline*}
Cancelling $\bw g_1\xi_n(1-\beps\bw)^{\tens n}$ against the right hand side we come to the valid equation
\begin{equation*}
\delta_{n,2}\bw\tens\bw +\bw\xi_n -\bw\xi_n(1-\beps\bw)^{\tens n} =0: \1 \to D[-1]^{\tens n}.
\end{equation*}
It is obvious for $n=0,1$.
For $n=2$ we have by \eqref{eq-xi2-pr-pr}
\[ \bw\tens\bw +\bw\xi_2 -\bw\xi_2 -\bw(1\tens\bw) +\bw(\bw\tens1) +\bw\beps(\bw\tens\bw) =0.
\]

The identity morphism $g=(\id,0)$ is mapped to the identity morphism $\sfCobar g=(\id,0)$.
Let us verify that $\Cobar$ agrees with the composition.
If $h=\bigl(C\rto fD\rto gE\bigr)$ in $\ucCoalg$, \(h_1=f_1g_1\), \(h_0=f_0+f_1g_0\), then
\[ (\sfCobar_1f)\cdot\sfCobar_1g =\bar f\bar g =\bar h =\sfCobar_1h.
\]
In fact, the equation
\[ \sum_{i_1+\dots+i_k=n} \bar f_k(\bar g_{i_1}\tens\bar g_{i_2}\tdt\bar g_{i_k}) =\bar h_n: \bar C[-1] \to \bar E[-1]^{\tens n}
\]
for $n=1$ holds due to
\[ \bar f_1\bar g_1 =\oin f_1(1-\beps\bw)g_1\opr =\oin f_1g_1\opr =\oin h_1\opr =\bar{h}_1: \bar C[-1] \to \bar E[-1],
\]
and for $n=0$ it holds due to
\[ \bar f_0 +\bar f_1\bar g_0 =\oin f_0 +\oin f_1(1-\beps\bw)g_0 =\oin(f_0+f_1g_0) =\oin h_0 =\bar{h}_0: \bar C[-1] \to \1.
\]

Furthermore,
\[ \sfCobar_0g +(\sfCobar_0f) \cdot \sfCobar_1g =\sfCobar_0h
\]
composed with $\oin T$ holds true since
\[ \bw\tilde{g} +\bw\tilde{f}g[(1-\beps\bw)T]=\bw\tilde{h}: \1 \to E[-1]T.
\]
In fact, the $n$\n-th component of the left hand side is
\[ \bw\tilde{g}_n +\sum_{i_1+\dots+i_k=n} \bw\tilde{f}_k(g_{i_1}\tens g_{i_2}\tdt g_{i_k})(1-\beps\bw)^{\tens n}: \1 \to E[-1]^{\tens n},
\]
which for $n=1$ transforms to
\begin{multline*}
\bw g_1(1-\beps\bw) +\bw f_1(1-\beps\bw)g_1(1-\beps\bw) =\bw f_1g_1(1-\beps\bw) =\bw h_1(1-\beps\bw)
\ifx\chooseClass1
\\
\fi
=\bw\tilde{h}_1: \1 \to E[-1],
\end{multline*}
and for $n=0$ equals
\[ \bw g_0 +\bw f_0 +\bw f_1(1-\beps\bw)g_0 =\bw(f_0+f_1g_0) =\bw h_0 =\bw\tilde{h}_0.
\]
Hence the functor \(\Cobar:\ucCoalg\to\ucAlg\), the cobar-construction.
\QED\end{proof}

\subsection{Bar-construction}
Let us construct a functor \(\Bbar:\ucAlg\to\chCoalg\), the bar-\hspace{0pt}construction.
Let \(A=(A,b_2,b_1,b_0,\bfeta,\bv)\) be a unit-complemented curved algebra.
Decompose the idempotent \(1-\sfv\cdot\eta:A\to A\) into a projection \(\opr:A\to\bar{A}\) and an injection \(\oin:\bar{A}\to A\) so that \(\oin\opr=1_{\bar{A}}\).
The morphisms \(\sigma^{-1}\opr\sigma:A[1]\to\bar{A}[1]\) and \(\sigma^{-1}\oin\sigma:\bar{A}[1]\to A[1]\) are also denoted $\opr$ and $\oin$ by abuse of notations.
They split the idempotent \(1-\bv\cdot\bfeta:A[1]\to A[1]\).

\begin{proposition}
Define \(\Bbar A\) as the homotopy coalgebra \(\bar{A}[1]\tilde{T}\) equipped with the degree~1 coderivation \(\delta^{\Bbar A}_1=\bar{b}:\bar{A}[1]\tilde{T}\to\bar{A}[1]\tilde{T}\) given by its components
\[ \bar{b}_n =\bigl( \bar{A}[1]^{\tens n} \rTTo^{\oin^{\tens n}} A[1]^{\tens n} \rTTo^{b_n} A[1] \rTTo^\opr \bar{A}[1] \bigr), \qquad n=0,1,2,
\]
and a degree~2 functional
\[ \delta^{\Bbar A}_0 =-\bigl( \bar{A}[1]T \rTTo^{\oin T} A[1]T \rTTo^{\check{b}} A[1] \rTTo^\bv \1 \bigr).
\]
Then \(\Bbar A\) is a curved homotopy coalgebra.
\end{proposition}

\begin{proof}
Equation~\eqref{eq-1delta1-delta-1delta-delta1-XX} for $n\ge0$ takes the form
\begin{multline*}
\sum_{r+k+t=n} (1^{\tens r}\tens\bar b_k\tens1^{\tens t})\bar b_{r+1+t} 
\ifx\chooseClass1
\\
\fi
=\oin^{\tens(n-1)}b_{n-1}\bv\tens1 -1\tens\oin^{\tens(n-1)}b_{n-1}\bv: \bar A[1]^{\tens n} \to \bar A[1].
\end{multline*}
This holds true due to computation
\begin{multline*}
\oin^{\tens n}\sum_{r+k+t=n} (1^{\tens r}\tens b_k(1-\bv\bfeta)\tens1^{\tens t})b_{r+1+t}\opr
\\
=-\oin^{\tens n}(1\tens b_{n-1}\bv\bfeta +b_{n-1}\bv\bfeta\tens1)b_2\opr =\oin^{\tens(n-1)}b_{n-1}\bv\tens1 -1\tens\oin^{\tens(n-1)}b_{n-1}\bv.
\end{multline*}

Furthermore, the left hand side of \eqref{eq-1delta1-delta-X1} vanishes due to
\begin{multline*}
-\oin^{\tens n}\sum_{r+k+t=n} (1^{\tens r}\tens b_k(1-\bv\bfeta)\tens1^{\tens t})b_{r+1+t}\bv
\\
=\oin^{\tens n}(1\tens b_{n-1}\bv\bfeta +b_{n-1}\bv\bfeta\tens1)b_2\bv  =\oin^{\tens n}(\bv\tens b_{n-1}\bv -b_{n-1}\bv\tens\bv) =0,
\end{multline*}
because \(\oin\bv=0\).
\QED\end{proof}

\begin{proposition}
The map \(\Ob\ucAlg\to\Ob\chCoalg\) extends to a functor \(\Bbar:\ucAlg\to\chCoalg\), taking a morphism \(f=(f_1,f_0):A\to B\) to the morphism
\[ \sfBar f =\sfg =(\sfg_1,\sfg_0): \bar{A}[1]\tilde{T}\to \bar{B}[1]\tilde{T},
\]
where the coalgebra homomorphism \(\sfBar_1f=\sfg_1=\bar{f}\) is specified by its components
\begin{align*}
\bar f_1 &=\bigl( \bar{A}[1] \rTTo^\oin A[1] \rTTo^{f_1} B[1] \rTTo^\opr \bar{B}[1] \bigr),
\\
\bar f_0 &=\bigl( \1 \rTTo^{f_0} B[1] \rTTo^\opr \bar{B}[1] \bigr),
\end{align*}
and the degree~1 functional is
\begin{equation*}
\sfBar_0f =\sfg_0 =\bigl( \bar{A}[1]T \rTTo^{\oin T} A[1]T \rTTo^{\check{f}} B[1] \rTTo^\bv \1 \bigr).
\end{equation*}
\end{proposition}

\begin{proof}
Let us check that $\sfg$ is indeed a morphism of $\chCoalg$.
Equation~\eqref{eq-XT-XT-Y-Y} takes the form
\[ \bar b^A\check{\bar{f}} +\Delta(\sfg_0\tens1 -1\tens\sfg_0)\check{\bar{f}} =\hat\Delta\cdot \check{\bar{f}}\hat{T}\cdot \check{\bar{b}}^B: \bar{A}[1]T \to \bar{B}[1],
\]
that is, for all $n\ge0$
\begin{multline*}
\bar b_n\bar f_1 +f_0\bv\tens\bar f_n +\oin f_1\bv\tens\bar f_{n-1} -\bar f_n\tens f_0\bv -\bar f_{n-1}\tens\oin f_1\bv
\\
=\sum_{i_1+\dots+i_k=n} (\bar f_{i_1}\tens\bar f_{i_2}\tdt\bar f_{i_k})\bar b_k: \bar A[1]^{\tens n} \to \bar B[1].
\end{multline*}
In detail,
\begin{multline*}
\oin^{\tens n}b_n(1-\bv\bfeta)f_1\opr +\oin^{\tens n}\sum_{i_1+i_2=n} \bigl( f_{i_1}\bv\tens f_{i_2}\opr -f_{i_1}\opr\tens f_{i_2}\bv \bigr)
\\
=\oin^{\tens n}\sum_{i_1+\dots+i_k=n} [f_{i_1}(1-\bv\bfeta)\tdt f_{i_k}(1-\bv\bfeta)]b_k\opr.
\end{multline*}
Cancelling the summands without $\bv$ we reduce the equation to the valid identity
\begin{equation*}
\oin^{\tens n}\sum_{i_1+i_2=n} \bigl( f_{i_1}\bv\tens f_{i_2} -f_{i_1}\tens f_{i_2}\bv \bigr)\opr =-\oin^{\tens n}\sum_{i_1+i_2=n} \bigl( f_{i_1}\bv\bfeta\tens f_{i_2} +f_{i_1}\tens f_{i_2}\bv\bfeta \bigr)b_2\opr.
\end{equation*}

Equation~\eqref{eq-XT-XT-Y-1} with inverted signs,
\[ \oin^{\tens n}\check b^A_n\bv +\bar b^A_n\oin f_1\bv +\oin^{\tens n} \Delta(\check{f}\bv\tens\check{f}\bv) =\hat\Delta\cdot (\check{\bar{f}}\oin)\hat{T}\cdot \check{b}^B\bv: \bar A[1]^{\tens n} \to \1,
\]
is written explicitly as
\begin{multline*}
\oin^{\tens n}b_n\bv +\oin^{\tens n}b_n(1-\bv\bfeta)f_1\bv +\oin^{\tens n} \sum_{i_1+i_2=n} f_{i_1}\bv\tens f_{i_2}\bv 
\\
=\oin^{\tens n} \sum_{i_1+\dots+i_k=n} [f_{i_1}(1-\bv\bfeta)\tdt f_{i_k}(1-\bv\bfeta)]b_k\bv: \bar{A}[1]^{\tens n} \to \1.
\end{multline*}
Cancelling the first and the third summands as well as summands that contain $\bv$ only at the end, we obtain the valid equation
\begin{multline*}
\oin^{\tens n} \sum_{i_1+i_2=n} f_{i_1}\bv\tens f_{i_2}\bv
\ifx\chooseClass1
\\
\fi
=-\oin^{\tens n} \sum_{i_1+i_2=n} \bigl[ (f_{i_1}\bv\bfeta\tens f_{i_2} +f_{i_1}\tens f_{i_2}\bv\bfeta)b_2\bv +(f_{i_1}\bv\tens f_{i_2}\bv)\bfeta\bv \bigr].
\end{multline*}

The identity morphism $f=(\id,0)$ is mapped to the identity morphism $\sfBar f=(\id,0)$.
Let us verify that $\Bbar$ agrees with the composition.
If $h=fg$ in $\ucAlg$, \(h_1=f_1g_1\), \(h_0=g_0+f_0g_1\), then $\bar h=\bar f\bar g$.
In fact, the equation
\[ \sum_{i_1+\dots+i_k=n} (\bar f_{i_1}\tens\bar f_{i_2}\tdt\bar f_{i_k})\bar g_k =\bar h_n
\]
follows from
\begin{gather*}
\bar f_1\bar g_1 =\oin f_1(1-\bv\bfeta)g_1\opr =\oin f_1g_1\opr =\bar h_1,
\\
\bar g_0 +\bar f_0\bar g_1 =g_0\opr +f_0(1-\bv\bfeta)g_1\opr =(g_0 +f_0g_1)\opr =\bar h_0.
\end{gather*}

One has to prove also that
\[ \sfBar_0f +(\sfBar_1f) \cdot \sfBar_0g =\sfBar_0h,
\]
which amounts to equations
\[ \oin^{\tens n}f_n\bv +\bar f(\oin T)\check g\bv =\oin^{\tens n}h_n\bv: \bar{A}[1]^{\tens n} \to \1,
\]
for all $n\ge0$.
In fact, the left hand side is
\[ \oin^{\tens n}f_n\bv +\oin^{\tens n}f_n(1-\bv\bfeta)g_1\bv +\delta_{n,0}g_0\bv =\oin^{\tens n}(f_ng_1 +\delta_{n,0}g_0)\bv =\oin^{\tens n}h_n\bv.
\]
The functor \(\Bbar:\ucAlg\to\chCoalg\) is constructed.
\QED\end{proof}

\section{Twisting cochains}
We approach as close as possible to an adjunction between cobar- and bar-\hspace{0pt}constructions using twisting cochains.
If additional grading is present, we show that \(XT\simeq X\tilde{T}\) for some \(X\in\Ob\cv\) and obtain a true adjunction on full subcategories of categories of counit-complemented curved coalgebras and unit-complemented curved algebras.

\begin{definition}
Let $C$ be a curved coalgebra and $A$ be a curved algebra.
A degree~1 map \(\theta:C\to A\) which satisfies the equation
\begin{equation}
\theta m_1^A +\delta^C_1\theta =\delta^C_0\eta^A +\eps^Cm^A_0 -\delta^C_2(\theta\tens\theta)m^A_2: C \to A
\label{eq-theta-b}
\end{equation}
is called a \emph{twisting cochain}.
The set of twisting cochains is denoted \(\Tw(C,A)\).
\end{definition}

Assume additionally that $C$ is counit-complemented and $A$ is unit-complemented.
The reason to consider twisting cochains is given by the following diagram where the top bijections have to be constructed
\begin{diagram}[LaTeXeqno,bottom]
\label{dia-C-1T-A}
\ucAlg(\Cobar C,A) &\rDashTo &\Tw(C,A) &\lDashTo &\CCoalg(C,\Bbar A)
\\
\dMono &&&&\dMono
\\
\Alg(\bar{C}[-1]T,A)\times A^1 &&&&\hCoalg(C,\bar A[1]\tilde T)\times\und\cv(C,\1)^1
\\
\dTTo<\wr>{i_1} &&\dMono &&\dTTo<{i_3}>\wr
\\
\cv(\bar{C}[-1],A)\times\cv(\1[-1],A) &&&&\cv(C,\bar A[1])\times\cv(C,\1[1])
\\
\dTTo<\wr>{i_2} &&&&\dTTo<{i_4}>\wr
\\
\cv(C[-1],A) &\rTTo_\sim^{\sigma^{-1}\cdot-} &\und\cv(C,A)^1 &\lTTo_\sim^{-\cdot\sigma^{-1}} &\cv(C,A[1])
\\
\dLine &&&&\uTTo
\\
\HmeetV &&\rLine^{[1]} &&\HmeetV
\end{diagram}
The bijection $i_3$ is described in \propref{pro-lCoalg(CXT)-V(CX)}, bijections $i_2$ and $i_4$ are due to direct sum decompositions $C=\bar C\oplus\1$ and $A=\bar A\oplus\1$, and bijection $i_1$ describes homomorphisms from a free algebra.
A pair of morphisms of $\cv$ \((\check\sff_1:\bar{C}[-1]\to A,\sigma\sff_0:\1[-1]\to A)\) is mapped to a degree~1 map
\[ \sigma^{-1}[(\check\sff_1,\sigma\sff_0)i_2] =\theta=
\Bigl(
\begin{matrix}
\sigma^{-1}\check\sff_1
\\
\sff_0
\end{matrix}
\Bigr)
:\bar C\oplus\1 \to A.
\]
A pair of morphisms of $\cv$ \((\check\sfg_1:C\to\bar A[1],\sfg_0\sigma:C\to\1[1])\) is mapped to a degree~1 map
\[ [(\check\sfg_1,\sfg_0\sigma)i_4]\sigma^{-1} =\theta=
\bigl(
\begin{matrix}
\check\sfg_1\sigma^{-1} & \sfg_0
\end{matrix}
\bigr)
: C \to \bar A\oplus\1.
\]

\begin{proposition}\label{pro-unique-top-horizontal-bijections}
There are unique top horizontal bijections in diagram~\eqref{dia-C-1T-A} which make both hexagons commutative.
\end{proposition}

\begin{proof}
Denote by \(\sff_1:\bar{C}[-1]T\to A\in\Alg\) and \(\sfg_1:C\to\bar A[1]\tilde{T}\in\hCoalg\) morphisms coming from $\check\sff_1$ and $\check\sfg_1$, respectively.
We are going to show that systems of equations \eqref{eq-f1m1-f1ff1m-m1f1} on \((\sff_1,\sff_0)\) is equivalent to equation \eqref{eq-theta-b} on
\[ \theta=
\Bigl(
\begin{matrix}
\oin\theta
\\
\sfw\theta
\end{matrix}
\Bigr)=
\Bigl(
\begin{matrix}
\sigma^{-1}\check\sff_1
\\
\sff_0
\end{matrix}
\Bigr).
\]

All three terms of the first equation of \eqref{eq-f1m1-f1ff1m-m1f1} are $\sff_1$\n-derivations $\bar{C}[-1]T\to A$, hence, the equation is equivalent to its precomposition with \(\inj_1:\bar{C}[-1]\to\bar{C}[-1]T\):
\begin{equation}
\begin{split}
\check\sff_1m^A_1 +\check\sff_1(1\tens\sff_0-\sff_0\tens1)m^A_2 &=\check m^{\Cobar C}_1\sff_1: \bar{C}[-1]\to A,
\\
m^A_0 -\sff_0m^A_1 -(\sff_0\tens\sff_0)m^A_2 &=m^{\Cobar C}_0\sff_1: \1 \to A.
\end{split}
\label{eq-fm-mf-m-mf}
\end{equation}
In more detail equations~\eqref{eq-fm-mf-m-mf} read
\begin{subequations}
\begin{align}
\check\sff_1m^A_1 +\check\sff_1(1\tens\sff_0-\sff_0\tens1)m^A_2 &=\bar\xi_0\eta +\bar\xi_1\check\sff_1 +\bar\xi_2(\check\sff_1\tens\check\sff_1)m^A_2: \bar{C}[-1]\to A,
\\
m^A_0 -\sff_0m^A_1 -(\sff_0\tens\sff_0)m^A_2 &=-\bw\xi_0\eta^A -\bw\xi_1\opr\check\sff_1 \notag
\\
&-(\bw\tens\bw+\bw\xi_2)(\opr\check\sff_1\tens\opr\check\sff_1)m^A_2: \1 \to A.
\end{align}
\label{eq-fm-xe-xf-xffm}
\end{subequations}

We have
\begin{equation}
\theta=(\opr\oin+\eps\sfw)\theta=\opr\sigma^{-1}\check\sff_1+\eps\sff_0:C\to A.
\label{eq-theta=(prinepsw)theta}
\end{equation}
Equation~\eqref{eq-theta-b} can be rewritten as
\begin{equation}
\sigma\theta\sigma b_1 +\xi_1\sigma\theta\sigma +\xi_0\bfeta +\beps b_0 +\xi_2(\sigma\theta\sigma\tens\sigma\theta\sigma)b_2 =0: C[-1] \to A[1].
\label{eq-sigma-theta-sigma-b}
\end{equation}
Let us plug in it the expression
\(\sigma\theta\sigma=\opr\check\sff_1\sigma+\beps\sff_0\sigma:C[-1] \to A[1]\) and postcompose with $\sigma^{-1}$.
The obtained equation is
\begin{multline}
-\opr\check\sff_1m_1 -\beps\sff_0m_1 +\xi_1\opr\check\sff_1 +\xi_0\eta +\beps m_0 +\xi_2(\opr\check\sff_1\tens\opr\check\sff_1)m_2
\\
+(\sff_0\tens\opr\check\sff_1)m_2 -(\opr\check\sff_1\tens\sff_0)m_2 -\beps(\sff_0\tens\sff_0)m_2 =0: C[-1] \to A.
\label{eq-dovge-theta-ff}
\end{multline}
This equation is equivalent to the system of two equations obtained by precomposing \eqref{eq-dovge-theta-ff} with \(\oin:\bar C[-1]\rMono C[-1]\) and \(\bw:\1\to C[-1]\).
One verifies immediately that these two equations are nothing else but (\ref{eq-fm-xe-xf-xffm}a) and (\ref{eq-fm-xe-xf-xffm}b).

Let us show now that system of equations \eqref{eq-g1delta1-CXT}--\eqref{eq-g1delta0-C1} on \((\sfg_1,\sfg_0)\) taken in the form \eqref{eq-CCXX-C-XT}--\eqref{eq-CCX1-C-XT},
\begin{equation}
\begin{split}
\delta^C_1\check\sfg_1 +\delta^C_2(\sfg_0\tens1 -1\tens\sfg_0)\check\sfg_1 &=\sfg_1(1)\check\delta^{\Bbar A}_1: C\to \bar A[1],
\\
\delta^C_0 -\delta^C_1\sfg_0 -\delta^C_2(\sfg_0\tens\sfg_0) &=\sfg_1(1)\delta^{\Bbar A}_0: C\to \1,
\end{split}
\label{eq-dg-dg11gg-gd-d-dg}
\end{equation}
is equivalent to equations \eqref{eq-theta-b} and \eqref{eq-sigma-theta-sigma-b} on
\(\theta=\bigl(
\begin{matrix}
\theta\opr & \theta\sfv
\end{matrix}
\bigr)=\bigl(
\begin{matrix}
\check\sfg_1\sigma^{-1} & \sfg_0
\end{matrix}
\bigr)\).
We have
\begin{equation}
\theta=\theta(\opr\oin+\sfv\eta)=\check\sfg_1\sigma^{-1}\oin+\sfg_0\eta:C\to A,
\label{eq-theta=theta(prinveta)}
\end{equation}
hence \(\sigma\theta\sigma=\sigma\check\sfg_1\oin+\sigma\sfg_0\bfeta:C[-1]\to A[1]\).
Plug this expression in equation~\eqref{eq-sigma-theta-sigma-b} and precompose with $\sigma^{-1}$:
\begin{multline*}
\check\sfg_1\oin b_1 -\delta_1\check\sfg_1\oin -\delta_1\sfg_0\bfeta +\delta_0\bfeta +\eps b_0 +\delta_2(\check\sfg_1\oin\tens\check\sfg_1\oin)b_2
\\
+\delta_2(\check\sfg_1\oin\tens\sfg_0) -\delta_2(\sfg_0\tens\check\sfg_1\oin) -\delta_2(\sfg_0\tens\sfg_0)\bfeta =0: C \to A[1].
\end{multline*}
This equation is equivalent to the system of two equations obtained by postcomposing it with \(\opr:A[1]\rEpi \bar A[1]\) and \(\bv:A[1]\to\1\).
And these two equations are nothing else but \eqref{eq-dg-dg11gg-gd-d-dg} taken in the expanded form
\begin{gather*}
\delta^C_1\check\sfg_1 +\delta^C_2(\sfg_0\tens1 -1\tens\sfg_0)\check\sfg_1 =\eps^C\bar b^A_0 +\check\sfg_1\bar b^A_1 +\delta^C_2(\check\sfg_1\tens\check\sfg_1)\bar b^A_2: C\to \bar A[1],
\\
\delta^C_0 -\delta^C_1\sfg_0 -\delta^C_2(\sfg_0\tens\sfg_0) =-\eps^Cb^A_0\bv -\check\sfg_1\oin b^A_1\bv -\delta^C_2(\check\sfg_1\oin\tens\check\sfg_1\oin)b^A_2\bv: C\to \1.
\end{gather*}
The sets of solutions of these systems coincide, hence the second bijection at the top of diagram~\eqref{dia-C-1T-A}.
\QED\end{proof}

\begin{proposition}
The actions
\begin{alignat*}3
\cCoalg(C,D) \times \Tw(D,A) &\longrightarrow \Tw(C,A), &\quad (\sfg,\theta) &\longmapsto \sfg\rightharpoonup\theta &&=\sfg_1\theta +\sfg_0\eta,
\\
\Tw(C,A) \times \cAlg(A,B) &\longrightarrow \Tw(C,B), &\quad (\theta,\sff) &\longmapsto \theta\leftharpoonup\sff &&=\theta\sff_1 +\eps\sff_0
\end{alignat*}
make $\Tw$ into a $\cCoalg\text-\cAlg$-bimodule, in other words into a functor \(\Tw:\cCoalg^\op\) \(\times\cAlg\to\Set\).
\end{proposition}

\begin{proof}
One easily verifies that \(\sfg\rightharpoonup\theta\) and \(\theta\leftharpoonup\sff\) satisfy \eqref{eq-theta-b}, both actions are associative and \((\sfg\rightharpoonup\theta)\leftharpoonup\sff =\sfg\rightharpoonup(\theta\leftharpoonup\sff)\).
\QED\end{proof}

Let us show naturality of bijections constructed in \propref{pro-unique-top-horizontal-bijections}.

\begin{proposition}
The bijection
\begin{equation}
\cAlg(\Cobar C,A) \to \Tw(C,A)
\label{eq-cAlg(CobarCA)-Tw(CA)}
\end{equation}
is an isomorphism of $\ucCoalg$-$\cAlg$-bimodules.
The bijection
\begin{equation}
\Tw(C,A) \leftarrow \CCoalg(C,\Bbar A)
\label{eq-Tw(CA)-CCoalg(CBarA)}
\end{equation}
is an isomorphism of $\cCoalg$-$\ucAlg$-bimodules.
\end{proposition}

\begin{proof}
First of all, \eqref{eq-cAlg(CobarCA)-Tw(CA)} is a homomorphism with respect to the right action of \(\sfh\in\cAlg(A,B)\).
In fact, for an arbitrary \(\sff\in\cAlg(\Cobar C,A)\)
\[ 
\begin{pmatrix}
\sigma^{-1}\widecheck{(\sff\cdot\sfh)}_1
\\
(\sff\cdot\sfh)_0
\end{pmatrix}
=
\begin{pmatrix}
\sigma^{-1}\check\sff_1\sfh_1
\\
\sff_0\sfh_1 +\sfh_0
\end{pmatrix}
=
\begin{pmatrix}
\sigma^{-1}\check\sff_1
\\
\sff_0
\end{pmatrix}
\leftharpoonup\sfh: \bar C\oplus\1 \to A.
\]

Also \eqref{eq-cAlg(CobarCA)-Tw(CA)} is a homomorphism with respect to the left action of \(\sfj\in\ucCoalg(D,C)\), that is, with the use of \eqref{eq-theta=(prinepsw)theta} for an arbitrary \(\sff\in\cAlg(\Cobar C,A)\)
\begin{equation}
\opr\sigma^{-1}\widecheck{[(\Cobar j)\sff]}_1+\eps[(\Cobar j)\sff]_0 =\sfj\rightharpoonup(\opr\sigma^{-1}\check\sff_1+\eps\sff_0): D\to A.
\label{eq-pr(Cobarj)f1-eps(Cobarj)f0}
\end{equation}
In fact, the left hand side is
\begin{align*}
&\opr\sigma^{-1}\widecheck{(\Cobar_1j)}\sff_1 +\eps[(\Cobar_0j)\sff_1+\sff_0]
\\
&=\opr\sigma^{-1}\oin j_0\eta +\opr\sigma^{-1}\oin j_1\opr\check\sff_1 +\eps[\bw j_0\eta +\bw j_1\opr\check\sff_1 +\sff_0]
\\
&=(1-\eps\sfw)\sigma^{-1}j_0\eta +(1-\eps\sfw)\sigma^{-1}j_1\opr\check\sff_1 +\eps\sfw\sigma^{-1}j_0\eta +\eps\sfw\sigma^{-1}j_1\opr\check\sff_1 +\eps\sff_0
\\
&=\sigma^{-1}j_0\eta +\sigma^{-1}j_1\opr\check\sff_1 +\sfj_1\eps\sff_0 =\sfj_1\opr\sigma^{-1}\check\sff_1 +\sfj_1\eps\sff_0 +\sfj_0\eta,
\end{align*}
which is the right hand side of \eqref{eq-pr(Cobarj)f1-eps(Cobarj)f0}.

Thirdly, \eqref{eq-Tw(CA)-CCoalg(CBarA)} is a homomorphism with respect to the left action of \(\sfj\in\cCoalg(D,C)\).
In fact, due to \eqref{eq-CCoalg(DXT)-CCoalg(CXT)} for an arbitrary \((\sfg_1,\sfg_0)\in\CCoalg(C,\Bbar A)\) the element \(\sfj\cdot(\sfg_1,\sfg_0)=(\sfj_1\sfg_1,\sfj_0+\sfj_1\sfg_0)\) is mapped to
\[ 
\begin{pmatrix}
\sfj_1\check\sfg_1\sigma^{-1} & \sfj_0+\sfj_1\sfg_0
\end{pmatrix}
=\sfj\rightharpoonup
\begin{pmatrix}
\check\sfg_1\sigma^{-1} & \sfg_0
\end{pmatrix}
:C\to \bar A\oplus\1.
\]

Finally, \eqref{eq-Tw(CA)-CCoalg(CBarA)} is a homomorphism with respect to the right action of \(\sfh\in\ucAlg(A,B)\), that is, with the use of \eqref{eq-theta=theta(prinveta)} for an arbitrary \(\sfg=(\sfg_1,\sfg_0)\in\CCoalg(C,\Bbar A)\)
\begin{equation}
\widecheck{(\sfg\cdot\Bbar\sfh)}_1\sigma^{-1}\oin+(\sfg\cdot\Bbar\sfh)_0\eta
=(\check\sfg_1\sigma^{-1}\oin+\sfg_0\eta)\leftharpoonup\sfh.
\label{eq-gBarh-(gin)h}
\end{equation}
In fact, due to \eqref{eq-CCoalg(CXT)-CCoalg(CYT)} the left hand side is
\begin{align*}
&\sfg_1\widecheck{(\Bbar_1\sfh)}\sigma^{-1}\oin +(\sfg_0+\sfg_1(1)\wt{\Bbar_0\sfh})\eta
\\
&=\check\sfg_1\oin h_1\opr\sigma^{-1}\oin +\eps h_0\opr\sigma^{-1}\oin +\sfg_0\eta +\check\sfg_1\oin h_1\bv\eta +\eps h_0\bv\eta
\\
&=\check\sfg_1\oin h_1\sigma^{-1}(1-\sfv\eta) +\eps h_0\sigma^{-1}(1-\sfv\eta) +\sfg_0\eta +\check\sfg_1\oin h_1\sigma^{-1}\sfv\eta +\eps h_0\sigma^{-1}\sfv\eta
\\
&=\check\sfg_1\oin h_1\sigma^{-1} +\eps h_0\sigma^{-1} +\sfg_0\eta\sfh_1 =\check\sfg_1\sigma^{-1}\oin\sfh_1 +\sfg_0\eta\sfh_1 +\eps\sfh_0,
\end{align*}
which is the right hand side of \eqref{eq-gBarh-(gin)h}.
\QED\end{proof}

\subsection{$\ZZ\times\NN$-graded $\kk$-modules}
Let $\cw$ be the monoidal category of $\ZZ\times\NN$-graded abelian groups.
Its object $X$ is a collection \((X^p_i)^{p\in\ZZ}_{i\in\NN}\), \(X^p_i\in\Ob\Ab\).
The tensor product is given by
\[ (\sS{_1}X\tdt\sS{_n}X)^p_i =\bigsqcup^{p_1+\dots+p_n=p}_{i_1+\dots+i_n=i} \sS{_1}X^{p_1}_{i_1}\tdt\sS{_n}X^{p_n}_{i_n}.
\]
The monoidal category $\cw$ is equipped with the symmetry
\begin{equation}
c: X^p_i\tens Y^q_j \ni x\tens y \mapsto (-1)^{pq}y\tens x \in Y^q_j\tens X^p_i.
\label{eq-c:XpiOYqj}
\end{equation}
Let $\kk$ be a $\ZZ\times\NN$-graded commutative ring, that is, a commutative algebra in $\cw$.
Denote by $\cv=\kk\modul$ the category of $\ZZ\times\NN$-graded $\kk$\n-modules, that is, commutative $\kk$\n-bimodules in $\cw$.
It has the tensor product $\tens_\kk$ and the symmetry inherited from $\cw$ as in \eqref{eq-c:XpiOYqj}.
There are epimorphisms \((\sS{_1}X\tdt\sS{_n}X)^p_i\rEpi (\sS{_1}X\tens_\kk\dots\tens_\kk\sS{_n}X)^p_i\).

\begin{lemma}\label{lem-e:XT-isomorphism}
Let $X\in\Ob\cv$, $X^p_0=0$ for all $p\in\ZZ$.
Then the homotopy coalgebra morphism \(e:XT\to X\tilde T\) of \exaref{exa-e:XT-tildeT} is an isomorphism.
\end{lemma}

\begin{proof}
The morphisms \(e(I):V=XT^{\tens_\kk I}\to\prod_{n\in\NN^I}X^{\tens_\kk\|n\|}=W\) given by \eqref{eq-e(I)prn} embed $V$ into $W$.
Here
\[ V^p_j =\coprod_{n\in\NN^I} (X^{\tens_\kk\|n\|})^p_j, \qquad W^p_j =\prod_{n\in\NN^I} (X^{\tens_\kk\|n\|})^p_j
\]
for all $p\in\ZZ$, $j\in\NN$.
Notice that \((X^{\tens\|n\|})^p_j=0\) if $\|n\|>j$, hence, \((X^{\tens_\kk\|n\|})^p_j=0\) for such $n$.
Therefore,
\[ V^p_j =\bigoplus^{\|n\|\le j}_{n\in\NN^I} (X^{\tens_\kk\|n\|})^p_j =W^p_j
\]
and the morphisms $e(I)$ are isomorphisms.
\QED\end{proof}

Denote by $\chCoalg_+$ the full subcategory of $\chCoalg$ formed by curved homotopy coalgebras \((X\tilde T,\delta_1,\delta_0)\) with $X_0^\bullet=0$.
Denote by $\ucCoalg_+$ the full subcategory of $\ucCoalg$ formed by counit-complemented curved coalgebras $C$ with $\bar C_0^\bullet=0$.

\begin{proposition}
There is a full embedding \(\chCoalg_+\rMono \ucCoalg_+\).
\end{proposition}

\begin{proof}
The functor to be constructed takes a curved homotopy coalgebra \((X\tilde T,\delta_1,\delta_0)\) to the counit-complemented curved coalgebra \((XT,\delta_2,\delta_1,\delta_0,\pr_0,\inj_0)\).
A morphism of curved homotopy coalgebras \((\tilde\sfg_1,\sfg_0):X\tilde T\to Y\tilde T\) with $Y_0^\bullet=0$ is taken to \((\sfg_1,\sfg_0):XT\to YT\), where \(\sfg_1:XT\to YT\in\Coalg\) is the only morphism equal to
\[ \bigl( XT \rTTo^{e_X} X\tilde T \rTTo^{\tilde\sfg_1} Y\tilde T \rTTo^{e_X^{-1}} YT \bigr) \in \hCoalg,
\]
see \lemref{lem-e:XT-isomorphism} and \propref{pro-functor-Coalg-lCoalg}.
In fact, postcomposing the first equation of \eqref{eq-deltas-gees} with $e_Y$ and precomposing \eqref{eq-del1g1} with $e_X$ we get equivalent equations.
Precomposing \eqref{eq-delta0-delta1gee0} with $e_X$ we get the second equation of \eqref{eq-deltas-gees}.
Therefore, the assignment on morphisms is bijective.
Identity morphisms go to identity morphisms.
Clearly the composition is preserved.
\QED\end{proof}

Denote by $\ucAlg_+$ the full subcategory of $\ucAlg$ formed by unit-complemented curved algebras $A$ with $\bar A_0^\bullet=0$.
For such $A$ the bar-construction $\Bbar A$ is in $\chCoalg_+$, so we may compose:
\[ \Bbar_+ \overset{\text{def}}= \bigl( \ucAlg_+ \rTTo^\Bbar \chCoalg_+ \rMono \ucCoalg_+ \bigr).
\]
Restricting $\Cobar$ we get \(\Cobar_+=\Cobar\big|:\ucCoalg_+\to\ucAlg_+\).

\begin{theorem}
The functors \(\Cobar_+:\ucCoalg_+\leftrightarrows\ucAlg_+:\Bbar_+\) are adjoint to each other.
\end{theorem}

\begin{proof}
For any \(X\tilde T\in\chCoalg_+\) there is an element \((e_X,0)\in\CCoalg(XT,X\tilde T)\).
Map~\eqref{eq-CCoalg(DXT)-CCoalg(CXT)} in the form \(\cCoalg(C,XT)\times\CCoalg(XT,X\tilde{T})\to\CCoalg(C,X\tilde{T})\) restricted to $(e_X,0)$ leads to
\begin{align*}
p: \cCoalg(C,XT) &\longrightarrow \CCoalg(C,X\tilde{T}),
\\
(\sff_1,\sff_0) &\longmapsto (\sff_1,\sff_0) \cdot (e_X,0) =(\sff_1e_X,\sff_0).
\end{align*}
By \lemref{lem-e:XT-isomorphism} this map admits an inverse
\begin{align*}
q: \CCoalg(C,X\tilde{T}) &\longrightarrow \cCoalg(C,XT),
\\
(\tilde\sfg_1,\sfg_0) &\longmapsto (\sfg_1=\tilde\sfg_1 e_X^{-1},\sfg_0).
\end{align*}
Therefore, $p$ is a bijection.

By construction $p$ is natural in $C\in\cCoalg$.
It is natural also in \(X\tilde T\in\chCoalg_+\) by an easy computation based on the relation \(\sfh_1e_X=e_X\tilde\sfh_1:XT\to Y\tilde T\in\hCoalg\) for an arbitrary \(\tilde\sfh_1:X\tilde T\to Y\tilde T\in\hCoalg\), \(X_0^\bullet=Y_0^\bullet=0\).
(Note that \(\sfh_1:XT\to YT\) may have non-trivial component \(\sfh_{1;0}:\1\to Y\).)
As a corollary we get a bijection \(\cCoalg(C,\Bbar_+A)\to\CCoalg(C,\Bbar A)\) natural in $C\in\cCoalg$, \(A\in\ucAlg_+\).
Combining it with the top row of diagram~\eqref{dia-C-1T-A} we obtain a bijection \(\ucAlg(\Cobar C,A)\rto\sim \ucCoalg(C,\Bbar_+A)\) natural in $C\in\ucCoalg$, \(A\in\ucAlg_+\).
Assuming additionally that \(C\in\ucCoalg_+\) we deduce the adjunction stated in the theorem.
\QED\end{proof}

\ifx\chooseClass1
\bibliographystyle{spmpsci}
\bibliography{yuri}

\begin{thebibliography}{M{ac}88}

\bibitem[BLM08]{BesLyuMan-book}
Yuri Bespalov, V.~V. Lyubashenko, and Oleksandr Manzyuk, \emph{Pretriangulated
  ${A}_\infty$-categories}, - Proceedings of the Institute of Mathematics of
  NAS of Ukraine. Mathematics and its Applications, vol.~76, Institute of
  Mathematics of NAS of Ukraine, Kyiv, 2008,
  \url{http://www.math.ksu.edu/~lub/papers.html}.

\bibitem[Lei99]{math/9912084}
Tom Leinster, \emph{Up-to-homotopy monoids}, 1999,
  \href{http://arXiv.org/abs/math/9912084}{{\tt
  arXiv:\linebreak[1]math/9912084}}.

\bibitem[Lei00]{math.QA/0002180}
\bysame, \emph{Homotopy algebras for operads}, 2000,
  \href{http://arXiv.org/abs/math/0002180}{{\tt
  arXiv:\linebreak[1]math/0002180}}.

\bibitem[Lei03]{math.CT/0305049}
\bysame, \emph{Higher operads, higher categories}, - London Mathematical
  Society Lecture Notes Series, Cambridge University Press, Boston, Basel,
  Berlin, 2003, \href{http://arXiv.org/abs/math/0305049}{{\tt
  arXiv:\linebreak[1]math/0305049}}.

\bibitem[Lyu13]{Lyu-curved-coalgebras}
V.~V. Lyubashenko, \emph{Bar and cobar constructions for curved algebras and
  coalgebras}, Matematychni Studii (2014?), to appear,
  \href{http://arXiv.org/abs/1311.7681}{{\tt arXiv:\linebreak[1]1311.7681}}.

\bibitem[M{ac}88]{MacLane}
Saunders M{ac Lane}, \emph{Categories for the working mathematician}, - GTM,
  vol.~5, Springer-Verlag, New York, 1971, 1988.

\bibitem[Pos11]{0905.2621}
Leonid Positselski, \emph{Two kinds of derived categories, {K}oszul duality,
  and comodule-contramodule correspondence}, Memoirs Amer. Math. Soc.
  \textbf{212} (2011), no.~996, vi+133,
  \href{http://arXiv.org/abs/0905.2621}{{\tt arXiv:\linebreak[1]0905.2621}}.

\bibitem[Pos12]{1202.2697}
\bysame, \emph{Weakly curved ${A}_\infty$-algebras over a topological local
  ring}, 2012, \href{http://arXiv.org/abs/1202.2697}{{\tt
  arXiv:\linebreak[1]1202.2697}}.

\bibitem[Pri10]{0908.0116}
Jonathan~Paul Pridham, \emph{The homotopy theory of strong homotopy algebras
  and bialgebras}, Homology, Homotopy Appl. \textbf{12} (2010), no.~2, 39--108,
  \href{http://arXiv.org/abs/0908.0116}{{\tt arXiv:\linebreak[1]0908.0116}}.

\bibitem[Ver96]{MR1453167}
Jean-Louis Verdier, \emph{Des cat\'egories d\'eriv\'ees des cat\'egories
  ab\'eliennes}, Ast{\'e}risque (1996), no.~239, xii+253 pp., With a preface by
  Luc Illusie, Edited and with a note by Georges Maltsiniotis. \MR{1453167}

\end{thebibliography}
	\else
\bibliographystyle{amsalpha}
\providecommand{\bysame}{\leavevmode\hbox to3em{\hrulefill}\thinspace}
\providecommand{\MR}{\relax\ifhmode\unskip\space\fi MR }
\providecommand{\MRhref}[2]{%
  \href{http://www.ams.org/mathscinet-getitem?mr=#1}{#2}
}
\providecommand{\href}[2]{#2}

\fi

\tableofcontents

\end{document}